\documentclass{article}
\pdfoutput=1

\usepackage[utf8]{inputenc}
\usepackage[all,2cell]{xy}
\usepackage{graphicx}
\usepackage{amsmath}
\usepackage{amsthm}
\usepackage{amsfonts}
\usepackage{amssymb}
\usepackage{indentfirst}
\usepackage{bussproofs}
\usepackage{tikz-cd}
\usepackage{xcolor}
\DeclareMathAlphabet{\mathpzc}{OT1}{pzc}{m}{it}
\usepackage{enumitem} 
\usepackage{geometry}
\usepackage{hyperref}
\usepackage{longtable}
\usepackage{mathabx}

\hypersetup{pdfauthor={Name}}

\theoremstyle{plain} %italico
\newtheorem{theorem}{Theorem}[section]

\newtheorem{lemma}[theorem]{Lemma}
\newtheorem{proposition}[theorem]{Proposition}

\theoremstyle{definition} %stampatello
\newtheorem{definition}[theorem]{Definition}

\newtheorem{remark}[theorem]{Remark}

\newtheorem{corollary}[theorem]{Corollary}

\newcommand{\setem}{{\it Set_{mf}}}
\newcommand{\type}{{\it type}}
\newcommand{\mf}{{\bf MF}}
\newcommand{\hott}{{\bf HoTT}}
\newcommand{\mtt}{\mbox{{\bf mTT}}}
\newcommand{\emtt}{\mbox{{\bf emTT}}}
\newcommand{\mltt}{\mbox{{\bf MLTT}}}
\newcommand{\isprop}{\mbox{{\tt IsProp}}}
\newcommand{\isset}{\mbox{{\tt IsSet}}}
\newcommand{\prs}{ \mathsf{pr_S}}
\newcommand{\prp}{ \mathsf{pr_P}}

\newcommand{\gam}{[\Gamma^{\sqbullet}]}

\newcommand{\ourdef}{{\ :\equiv\ }}
\newcommand{\funi} {\mathcal{U}_0}
\newcommand{\suni} {\mathcal{U}_1}
\newcommand{\guni} {\mathcal{U}_i}
\newcommand{\set} {\mathsf{Set}}
\newcommand{\fpr} {\mathsf{pr}_1}
\newcommand{\spr} {\mathsf{pr}_2}
\newcommand{\ipr}{\mathsf{pr}_i}
\newcommand{\prop} {\mathsf{Prop}}
\newcommand{\propo} {\mathsf{Prop}_{\mathcal{U}_0}}
\newcommand{\inter} {{\bullet}}
\newcommand{\trasp} {\mathsf{trp}}
\newcommand{\can} {canonical}
\newcommand{\cqis}{\mbox{\rm Q({\bf mTT})/$\cong$}}
\newcommand{\setis}{ {\mathsf{Set}_{mf}}_{/{\cong_c}}}

\newcommand{\fintd}{\textbf{In}_{D}}
\newcommand{\keywords}[1]{\textit{Keywords:}#1}

\begin{document}

\newgeometry{left= 3.2cm, right= 3.2cm}

\title{The Compatibility of the Minimalist Foundation with Homotopy Type Theory}

\author{Michele Contente \\ \footnotesize Scuola Normale Superiore, Pisa, Italy \\ \and Maria Emilia Maietti \\ \footnotesize Department of Mathematics, University of Padova, Italy}

\date{\today }

\maketitle

\begin{abstract}
The  Minimalist Foundation, \mf\ for short,  is a two-level foundation for constructive mathematics ideated by Maietti and Sambin in 2005 and then fully formalized by Maietti in 2009.
\mf\  serves as  a  common core  among  the most relevant foundations for  mathematics in the literature by choosing for each of them
 the appropriate level of \mf\  to be translated in a compatible way, namely by  preserving the meaning of logical and set-theoretical constructors.   The two-level structure
 consists of an intensional level, an extensional one, and an interpretation of the latter in the former  in order to extract intensional computational content  from mathematical proofs involving extensional constructions
 used in everyday mathematical practice.
 
 In 2013 a completely new foundation for constructive mathematics appeared in the literature, called Homotopy Type Theory, for short \hott,
 which is an example of Voevodsky's  Univalent Foundations with a computational nature.
%  has been recently discovered with the introduction of cubical type theories.
 
 So far no level of \mf\ has been proved to be compatible with any of the  Univalent Foundations in the literature. Here we show that both levels of \mf\ are compatible with \hott.
 This result  is made possible thanks to the peculiarities of \hott\ which combines intensional features of type theory with extensional ones
 by assuming  Voevodsky's Univalence Axiom and higher inductive quotient types.
As a  relevant consequence, \mf\ inherits entirely new computable models.

\end{abstract}

\noindent

\keywords{
Foundations for Constructive Mathematics\ -  Dependent Type Theory\ - Homotopy Type Theory\ - Two-level foundations\ - Many-sorted logic}
\\
{\bf MSC classification}:  03838, 03F50, 03F03
\tableofcontents

\section{Introduction}
Constructive mathematics is distinguished from ordinary classical mathematics for developing proofs governed by a constructive way of reasoning which confers them
an algorithmic nature. In the literature there are foundations for constructive mathematics that are suitable to make this visible by allowing to view constructive proofs
as programs.  Examples of these foundations can be found in type theory and they include  Martin-L{\"o}f's intensional dependent type theory\cite{PMTT} and Coquand-Huet's Calculus of Constructions \cite{Coq88}.
However, there is no standard foundation for constructive mathematics, but a plurality of different approaches.

In 2005 in \cite{mtt} Maietti and Sambin embarked on the project of building a Minimalist Foundation, called \mf, to serve as a common core among the most relevant foundations
for constructive mathematics in type theory, category theory and axiomatic set theory.  Indeed, \mf\  is intended to be ``minimalist in set existence assumptions'' but ``maximalist in conceptual distinctions and compatibility with other foundations''.

To meet  this purpose,  \mf\ was conceived   as a two-level theory consisting of an extensional  level, called \emtt, formulated in a language close to that of everyday mathematical practice and
interpreted via a quotient model in a further intensional level, called \mtt,  designed as a type-theoretic base for a proof-assistant.
The key idea is that the two-level structure should allow the extraction of intensional computational contents  from constructive mathematical proofs involving extensional constructions
typical of usual  mathematical practice.

A complete two-level formal system for \mf\ was finally designed in 2009 in  \cite{m09}. There, some of  the most relevant constructive and classical foundations have been related to \mf\
by choosing the appropriate level of \mf\  to be translated  into it  in a {\it compatible} way, namely by  preserving the meaning of logical and set-theoretical constructors so that proofs of mathematical theorems in one theory are understood
as proofs of mathematical theorems in the target theory with the same meaning.

Moreover, computational models  for   \mf\  and its extensions with inductive  and coinductive topological definitions have been presented  in  \cite{peff}, \cite{IMMS},\cite{mmr21} and \cite{cind}
in the form of Kleene realizability interpretations which validate  the  Formal Church's Thesis stating that all the number-theoretic functions are computable.  

In 2013 the book \cite{hottbook}  presented a completely new foundation for constructive  mathematics,  called Homotopy Type Theory, for short \hott,
 as an example of Voevodsky's  Univalent Foundations, for short UFs. Voevodsky introduced UFs with the aim of  better formalizing  his mathematical work on abstract homotopy theory and higher category theory and at the same time fully checking the correctness of his proofs on a modern proof-assistant.
 
 More precisely,  \hott\ is an intensional type theory extending
  Martin-L\"of's theory as presented in \cite{PMTT} with the so-called Univalence Axiom proposed by  Voevodsky to   guarantee  that \lq\lq isomorphic'' structures can be treated as equal besides deriving  some other extensional principles.
  Another remarkable property of \hott\ is that it can be equipped with primitive higher inductive types, including set quotients (see \cite{hottbook}).

 The computational contents of \hott-proofs as programs have been recently explored with the introduction of cubical type theories in \cite{CCHM17,canhott} and a normalization procedure for a variant of them has been given in \cite{SA21}.

So far no level of \mf\ has been proved to be compatible with  Univalent Foundations.   
Here we show that both levels of \mf\ are compatible with \hott. 
This result  is made possible thanks to the peculiarities of \hott\ which combines intensional features of type theory with extensional ones
 by assuming  Voevodsky's Univalence Axiom and higher inductive quotient types. In particular,  we will crucially use the Univalence Axiom instantiated for homotopy propositions
 and   {\it function extensionality}. The fact that we can interpret both levels of \mf\  into a single framework is a remarkable property of \hott\, which is not shared by any other foundation for mathematics to our knowledge. 

In more detail, we interpret \mf-types as homotopy sets and \mf-propositions as h-propositions and both the  \mtt-collection
of small propositions and the \emtt-power collection of subsets  as the homotopy set of h-propositions in the first universe of \hott.

This should be contrasted with the relationship between \mf\ and the intensional version of Martin-L{\"o}f Type Theory, for short \mltt,  shown in  \cite{m09}: in \mltt\ we can interpret only the intensional level of \mf\ by identifying propositions with sets.

\iffalse
{\tt toglierei questa frase perche' gia' detta

''
The new machinery available within the context of \hott\ including the Univalence Axiom and higher inductive quotient of homotopy sets, are essential to interpret both levels of \mf\ in a compatible way. ''
}
\fi

The main difficulty encountered in this work concerns the interpretation of the extensional level \emtt\ of \mf. Indeed, the  interpretation of  \mtt\ in \hott\
just required a careful handling of proof terms witnessing the fact that certain \hott-types are h-propositions and h-sets. Instead,  there is no straightforward way of interpreting \emtt\  in \hott,  because \emtt\ includes Martin-L\"of's extensional propositional equality in the style of \cite{ML84}.

We managed to solve this issue by employing a technique already used in \cite{m09} to interpret \emtt\ over the intensional level of \mf: \mbox{ \emtt -types} and terms are interpreted as \mbox{\hott -types} and terms up to a special class of isomorphisms, called {\it canonical} as in \cite{m09}, by providing a kind of
realizability interpretation in the spirit of  the interpretation of {\it true judgements}
in Martin-L{\"o}f's type theory described in \cite{ML84,sienlec}. 
We introduce the category $\setis$ of the h-sets contained in the  non-univalent universe $\setem$ (which is an inductive  universe of h-sets in the univalent universe $\suni$) equated under canonical isomorphisms and then we define an interpretation of \emtt-judgements into it.
In particular \emtt-type and term judgements are interpreted as \hott-type and term judgements up to canonical isomorphisms.
Furthermore,
the \emtt-definitional equality $A =B\  type \ [\Gamma]$ of two \emtt-types $A\ type \ [\Gamma]$
and $B\ type \ [\Gamma]$ is interpreted
as the existence of a canonical isomorphism that connects the \hott-type representatives interpreting the \emtt-types $A\ \ type \ [\Gamma]$
and $B\ type \ [\Gamma]$,  which turn out to be  propositionally equal in \hott\ thanks to Univalence. In turn, this interpretation is based on another auxiliary partial (multi-functional) interpretation of \emtt-raw syntax into \hott-raw syntax, which makes use of canonical isomorphisms. 

%It must be underlined that the idea of using canonical isomorphisms to interpret extensional aspects of type theory into intensional type theory  in \cite{m09}  
%was already presented the work by M. Hofmann with the main difference that the target theory in \cite{hofmann} is not the pure intensional type theory as  in \cite{m09}
%where a setoid model is used. Moreover, in \cite{hofmann} the interpretation makes use of the axiom of choice. This use has been avoided in the recent work by \cite{???}
%who produced an effective translation through the use of an heterogeneous equality. 
 
It must be stressed that the resulting interpretation of \emtt\ into \hott\ is  simpler than that of \emtt\ within \mtt\ in \cite{m09}, since we can avoid any quotient model construction thanks to (effective) set-quotients and Univalence. This interpretation turns out to be very similar to that presented in \cite{Oury05,WinST19} which makes effective the interpretation of extensional aspects of type theory into  an intensional base theory  originally presented in \cite{Hof95}. However, in \cite{Oury05,WinST19} there is a use of an heterogenous equality instead of canonical isomorphisms as in \cite{Hof95}. Moreover, the interpretation of \emtt\ into \mtt\  does not show the compatibility of \emtt\ with \mtt\, exactly because of the lack of Univalence and effective quotients in \mtt. 

Observe that it does not appear possible to identify  \lq \lq compatible'' subsystems of \hott\ corresponding to each level of \mf: in \hott\ the interpretation of the existential quantifier allows to derive both the axiom of  unique choice and the rule of unique choice as it happens in the internal logic of a topos like that described in \cite{M05},  while in each level of \mf\ these principles are not generally valid \cite{M17, MR16, MR13}, since the existential quantifier in \mf\ is defined in a primitive way.

As a   relevant consequence of the results presented here, both levels of  \mf\ inherit  new computable  models where constructive functions are seen as computable 
  as those  in \cite{SA21}  and 
  in  \cite{su22}.  We leave to future work to relate them with those available for \mf\ and in particular with the predicative variant of Hyland's Effective Topos
  in \cite{peff}.

%  from Sterling we show that our computability model (Theo-
%rem 21) lets us compute the normal form of every syntactic
%type, implying the (external to E) decidability of type equality
%in cubical type theory, and the injectivity of type constructo

\section{Preliminaries about \mf\ and \hott}

In this section we recall some basic facts about \mf\ and \hott\ that will turn out to be useful later. We will refer mainly to \cite{m09} for \mf\ and to \cite{hottbook} for \hott. 

\subsection{The two levels of \mf}
\mf\  is a two-level foundation for constructive mathematics, that was first conceived in \cite{mtt} and then fully developed in \cite{m09}.
It consists of an intensional level, called \mtt,  and an extensional one, called \emtt, together with an interpretation of the latter in the first. Both levels of \mf\ extend a version of Martin-L{\"o}f's type theory with a primitive notion
of proposition: \mtt\ extends the intensional type theory in \cite{PMTT}, while \emtt\ extends the extensional version presented in \cite{ML84}.

The resulting two-level theory is strictly predicative in the sense of Feferman as first shown in \cite{peff}.

A peculiarity of  \mf\  with respect to Martin-L\"of's type theories is that types at each level of \mf\  are built by using four basic distinct sorts: small propositions, propositions, sets and collections. The relations between these sorts are shown on the following diagram where the inclusion mimics a subtyping relation:

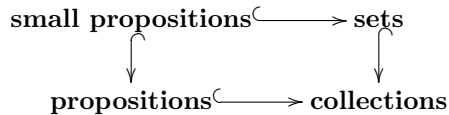
\begin{center}
$$
\xymatrix@C=0.5em@R=0.5em{
 {\mbox{\bf small propositions}} \ar@{^{(}->}[dd]\ar@{^{(}->}[rr]&  &
 {\mbox{\bf sets}} \ar@{^{(}->}[dd] \\ 
&&\\
  {\mbox{\bf propositions}} \ar@{^{(}->}[rr] && {\mbox{\bf collections}}
}
$$
\end{center}
\vspace{1.0em}
In particular, the distinction between sets and collections is meant to recall that between sets and classes in axiomatic set theory,  while the word \lq\lq small'' attached to propositions is
taken from algebraic set theory \cite{jm95}. Indeed,  small propositions are defined as those propositions that are closed
under intuitionistic connectives and quantifiers and whose equalities are restricted to sets.

More formally, the basic forms of judgement in \mf\ include
\begin{center}
    
 $ A \ set\ [\Gamma] \qquad B\ coll\ [\Gamma] 
\qquad \phi\ prop\
 [\Gamma]\qquad \psi\ prop_s\
 [\Gamma] 
 $
\end{center}
 to which we add the meta-judgement
$$A\ \type\ [\Gamma]$$
where `\type ' is to be interpreted as a meta-variable ranging over the four basic sorts.

We warn the reader that the  type constructors of both levels of  \mtt\ and \emtt\ are respectively defined in an inductive way mutually involving all the four sorts, i.e.
we can not give a definition of  collections independently from that of sets or  propositions or  small propositions  and the same holds for  the definition of each of these sorts.

The set-constructors of \mtt\ and \emtt\  include those of first order  Martin-L{\"o}f's type theory, respectively  as presented in \cite{PMTT} and \cite{ML84},  together with list types.
We just recall their notation:  $N_0$ stands for the empty set,
$N_1$ for the singleton set, $List(A)$ for the set of lists over the set $A$,
$\Sigma_{x\in A} B(x)$  and  $\Pi_{x\in A} B(x)$ stand respectively for the indexed sum  and the dependent product of the family of sets $B(x)\ set \ [x\in A]$ indexed on the set $A$,  $A+B$ for
the disjoint sum of the set $A$ with the set $B$.  
%Moreover, sets of \mtt\ include small propositions $\phi\  prop_s$ as set of their proofs.

Moreover, sets of \emtt\  are distinguished from those of \mtt, because they are closed under effective quotients
$A/R$ on a set $A$, provided that $R $ is a small equivalence relation $R(x,y) \
prop_s\ [x\in A,y\in A]$. 

In addition,   both the sets of \mtt\ and  those of \emtt\ include also  their small propositions $\phi\  prop_s$ thought as sets of their proofs.

Moving now to describe collections of \mtt\ and \emtt, we recall that  they both  include their sets and the indexed sum  $\Sigma_{x\in A} B(x)$  of the family of collections $B(x)\  coll \ [x\in A]$ indexed on a
collection $A$.
But,  whilst  \mtt-collections  include the proper collection of small propositions $\mathsf{prop_s}$ and  
the collection of small propositional functions  $A\rightarrow  \mathsf{prop_s}$ over a set $A$ (which are definitely not sets predicatively when $A$ is not empty!),
 the collections of  \emtt\  include the power-collection of the singleton  $ {\cal P}(1)$, that is the quotient of the collection of small propositions under the relation of equiprovability,  and  
the power-collection $A\rightarrow  {\cal P}(1)$ of a set $A$, that can be written simply as ${\cal P}(A)$. 

In addition, both collections of  \mtt\ and  those of \emtt\ include propositions $\phi\  prop$ viewed as collections of their proofs.

Both  propositions of \mtt\  and of \emtt\  include  propositional connectives and  quantifiers according to the following grammar: for $\phi$ and $\psi$ generic propositions,  $ \phi\land\psi$ denotes the conjunction, 
$\phi\lor\psi$ the disjunction,  $\ \phi\rightarrow \psi$ the implication,  $\forall x\in A. \phi$ the universal quantification and
 $\exists x\in A. \phi $ the existential quantification,
 for any collection $A$. Finally,   \mtt-propositions include a propositional equality type between terms of a type $A$,  called \lq\lq intensional
 propositional equality'',  that is  denoted with the type
 $$\mathsf{Id}(A,a,b)
$$
%
%\begin{center}
%    \begin{tabular}{lc}
%$\phi\ prop\  \, \equiv\,$&   
% $\phi\ \vee\ \psi\enspace \mid$ $\ \phi\rightarrow \psi\enspace\mid $ $\forall\, x\in B\ \phi(x)\enspace\mid$  $\exists\, x\in B\ \phi(x)\enspace\mid$  $\mathsf{Id}(A,a,b) $
%\end{tabular}
%\end{center}
%\vspace{1.0em}
%
%{\tt forse meglio esplicitare il significa di ogni connettivo per uniformita' con il resto..}
 since it has the same rules as   Martin-L{\"o}f's intensional identity type  in \cite{PMTT}
except that its elimination rule is  {\it restricted} to act towards propositions only (see \cite{m09}).
Instead,  \emtt-propositions include an extensional propositional equality between terms of a type $A$ that is denoted with the type
 $$\mathsf{Eq}(A,a,b)
$$
since it has the same rules as the  propositional equality type  in \cite{ML84}  and thus its elimination rule is given by the so-called {\it reflection rule}.

Furthermore,  propositions of \emtt\ are assumed to be  {\it proof-irrelevant} by imposing that if  a proof of a proposition exists, this is  {\it unique} and equal to a canonical proof term called
$\mathsf{true}$. These facts are represented by the following rules

\begin{center}
    $$
\mbox{{\bf prop-mono}) }  \,\displaystyle{ \frac{ \displaystyle \phi\ prop\ [\Gamma]\qquad p\in \phi\  [\Gamma] \quad q\in \phi\  [\Gamma]}{\displaystyle p=q\in \phi\ \ [\Gamma]}}  
\qquad
\begin{array}{l}
\mbox{{\bf prop-true}) }  \,\displaystyle{ \frac{ \displaystyle \phi\ prop\ \qquad p\in \phi\  }
{\displaystyle \mathsf{true}\in \phi}} 
\end{array}
$$
\end{center}
\vspace{1.0em}

Finally, both in \mtt\ and in \emtt \ {\it small propositions} are defined as  those propositions closed under propositional connectives, quantifications over sets
and propositional equality over a set. For example, in \mtt\ (resp. in \emtt)   the propositional equality $\mathsf{Id}(A,a,b)$ (resp. $\mathsf{Eq}(A,a,b) $) and the quantifications
  $\forall\, x\in A. \phi$
or $\exists\, x\in A. \phi$ are all small propositions if $A$ is a set and $\phi$ is a small proposition, too.

\begin{remark}\label{ac}
It is important to stress that elimination of propositions in \mtt\ as well as in \emtt\ acts only toward propositions and {\it not} toward proper sets and collections. In this way, \mtt\ and \emtt\
do not generally validate  choice principles, including unique choice, thanks to the results in \cite{M17, MR16, MR13}, and  similarly to what happens in the Calculus of Constructions, as first shown in
\cite{st92}.
\end{remark}

Observe that in \mtt\  term congruence rules are replaced by an explicit substitution rule for terms: 

\begin{center}
    \noindent
$$
\begin{array}{l}
      \mbox{repl)} \ \
\displaystyle{ \frac
         { \displaystyle 
\begin{array}{l}
 c(x_1,\dots, x_n)\in C(x_1,\dots,x_n)\ \
 [\, x_1\in A_1,\,  \dots,\,  x_n\in A_n(x_1,\dots,x_{n-1})\, ]   \\[2pt]
a_1=b_1\in A_1\ \dots \ a_n=b_n\in A_n(a_1,\dots,a_{n-1})
\end{array}}
         {\displaystyle c(a_1,\dots,a_n)=c(b_1,\dots, b_n)\in
 C(a_1,\dots,a_{n})  }}
\end{array}     
$$
\end{center}
\vspace{1.0em}

As a consequence, the $\xi$-rule for dependent products is no more available. This modification is crucial in order to obtain a sound Kleene-realizability interpretation for \mtt\
as required in \cite{mtt} and shown  in \cite{IMMS, mmr21,cind}\footnote{The issue of the relation between the $\xi$-rule and Kleene-style realizability was first spotted in \cite{ML75} and is discussed also in \cite{IMMS}.}.

Finally,  in order to make the interpretation of \mtt\ into \hott\ smoother,  differently from the version of \mtt\ presented in \cite{m09}, we encode small propositions into the collection of small propositions via an operator as follows:

\begin{center}
    {\small
$\begin{array}{lll} 
\mbox{Pr$_1$) }\ \widehat{\bot} \in \mathsf{prop_s} &
\mbox{Pr$_2$) }\displaystyle{ \frac
         { \displaystyle  p\in  \mathsf{prop_s} \hspace{.3cm} q \in
         \ \mathsf{prop_s}}
         {\displaystyle p \,\widehat{\vee}\, q\in \mathsf{prop_s} }} &
\mbox{Pr$_3$) }\displaystyle{ \frac
       {\displaystyle p\in
        \mathsf{prop_s} \qquad q\in
        \mathsf{prop_s} }
{p\,\widehat{\rightarrow}\,q \in \mathsf{prop_s} }}\\[15pt]
\mbox{Pr$_4$) }\displaystyle{ \frac
         { \displaystyle  p   \in \mathsf{prop_s}\qquad  q\in
           \mathsf{prop_s} }
         {\displaystyle  p\,\widehat{\wedge}\, q\in  \mathsf{prop_s} }}
&
\mbox{Pr$_5$) }\displaystyle{ \frac
{\displaystyle p \in \mathsf{prop_s}\ [x\in A]\qquad A\ set  }
{\displaystyle \widehat{(\forall x\in A)}\, p \in \mathsf{prop_s} }} 
&
\mbox{Pr$_6$) }\displaystyle{ \frac
         { \displaystyle   p\in
         \mathsf{prop_s}\ [x\in A]\qquad A\  set}
         {\displaystyle  \widehat{(\exists x\in A)}\, p\in
\mathsf{prop_s}}} \\[15pt]
&\mbox{Pr$_7$) }\displaystyle{ \frac
{\displaystyle  A \hspace{.1cm} set \hspace{.3cm}  a\in A \hspace{.3cm} b\in A}
         {\displaystyle \widehat{\mathsf{Id}}(A, a, b) \in \mathsf{prop_s} }}
&
\end{array}
$}
\end{center}
\vspace{1.0em}

Therefore, elements of the collection of small propositions can be decoded 
as small propositions by means of a decoding operator as follows

\begin{center}
    {\small
\noindent
$
\mbox{$\tau$-Pr) }\displaystyle{ \frac
       {\displaystyle p\in
        \ \mathsf{prop_s} }
{ \tau(p) \ prop_s }}
$
\\
}
\end{center}
\vspace{1.0em}

\noindent
and this operator satisfies the following definitional equalities:

\begin{center}
    {\small
$\begin{array}{ll} 
\mbox{eq-Pr$_1$) }\ \tau(\widehat{\bot})= \bot\, prop_s &\mbox{eq-Pr$_2$) }\displaystyle{ \frac
         { \displaystyle  p\in  \mathsf{prop_s} \,\,\, q \in
          \mathsf{prop_s}}
         {\displaystyle \tau( p\, \widehat{\vee}\, q)= \tau(p) \vee \tau(q)\,prop_s}}\\[15pt]
 \mbox{eq-Pr$_3$) }\displaystyle{ \frac {\displaystyle p\in
      \mathsf{prop_s} \qquad q\in
         \mathsf{prop_s} }
 {\tau(p\,\widehat{\rightarrow}\, q)= \tau(p)\rightarrow \tau(q)\,prop_s }} &\mbox{eq-Pr$_4$) }\displaystyle{ \frac
         { \displaystyle  p   \in \mathsf{prop_s}\,\,\,  q\in
           \mathsf{prop_s} }
         {\displaystyle  \tau(p\,\widehat{\wedge}\, q)=\tau(p) \wedge\tau(q)\, prop_s}}
\\[12pt]
\mbox{eq-Pr$_5$) }\displaystyle{ \frac
{\displaystyle p \in \mathsf{prop_s}\ [x\in A]\qquad A\ set  }
{\displaystyle \tau(\widehat{(\forall x\in A)} \,p)= (\forall x\in A) \,\tau(p) \, prop_s}} 
  &
\mbox{eq-Pr$_6$) }\displaystyle{ \frac
         { \displaystyle   p
         \in \mathsf{prop_s}\ [x\in A]\qquad A\  set}
         {\displaystyle \tau( \widehat{(\exists x\in A)}\, p)=
(\exists x\in A)\, \tau(p)\, prop_s }} \\[15pt]
\mbox{eq-Pr$_7$) }\displaystyle{ \frac
{\displaystyle  A \hspace{.1cm} set \hspace{.3cm}  a\in A \hspace{.3cm} b\in A}
         {\displaystyle \tau(\, \widehat{\mathsf{Id}}(A, a, b)\,)=
\mathsf{Id}(A,a,b)\,prop_s    }} 
\end{array}
$}
\end{center}
\vspace{1.2em}

A link between \mtt\ and \emtt\ is shown in \cite{m09} by interpreting  \emtt\ within a quotient model over \mtt. Such a quotient model was related to a free quotient completion construction in \cite{MR13}. Roughly speaking, thanks to the interpretation in \cite{m09},
\emtt\ types are seen as quotients of the corresponding intensional \mtt-types and thus \emtt\ can be regarded as a fragment of a quotient completion of the intensional level.  

More specifically, the interpretation of \emtt\ in \mtt\ relies upon the definition of a particular class of isomorphisms called {\it \can\ isomorphisms}, between dependent quotient types over \mtt, similar to  so called  {\it dependent setoids}.  It must be underlined that the idea of using canonical isomorphisms to interpret extensional aspects of type theory into intensional type theory  in \cite{m09}  was predated by M. Hofmann's  work  in \cite{Hof95} with the main difference that the target theory in \cite{Hof95} is not a pure intensional type theory as  in \cite{m09} where a setoid model is used.  Moreover, Hofmann's interpretation is not effective because of the use of the axiom of choice in the meta-theory.
 The interpretation in \cite{m09} is closer to the effective translation  presented in \cite{Oury05,WinST19} which refined Hofmann's one by employing a notion of heterogeneous equality.

\iffalse

which are closed under compositions, include identities and are unique up to propositional equality. For the definition of \can\ isomorphisms in \mtt\ we refer to \cite{m09}. 
\fi
Through this class of isomorphisms it is possible
to define a category of quotients over \mtt\  up to canonical isomorphisms within which to interpret \emtt\ correctly.

We underline that the interpretation of \emtt\ within
\mtt\ for some relevant constructors has been implemented and verified in \cite{FioriC18}.

Our main task  in this paper is to show the compatibility of each level of \mf\  with Homotopy Type Theory in \cite{hottbook}. For this purpose we make explicit the notion of compatibility between theories implicit in \cite{mtt} 
by stating that {\it a theory  $\bf T_1$ is said to be  {\bf  compatible } with another theory $\bf T_2$}  if and only if {\it  there exists a translation from  $\bf T_1$ to  $\bf T_2$ 
preserving the meaning of logical and set-theoretical constructors}
so that proofs of mathematical theorems in one theory are understood
as proofs of  mathematical theorems in the target  theory with the same meaning.

\subsection{Useful properties of \hott}
In 2013,  with the appearance of the book \cite{hottbook},  a completely new foundation for constructive mathematics  showed up under the name of 
Homotopy Type Theory, for  short \hott.  It was  introduced as an example of Voevodsky's Univalent Foundation with the remarkable property of combining intensional features of type theory with extensional ones.
Indeed, it  extends Martin-L{\"o}f's intensional type theory, for short \mltt, in \cite{PMTT} with  Voevodsky's $\textit{Univalence Axiom}$
and higher inductive types, including quotients of  homotopy sets and propositional truncation.

As a consequence, the  first order types of  \hott\ are the same as those of \mltt\ and therefore of the intensional level \mtt\ of \mf.
For the sake of clarity, we denote these types in \hott\  following \cite{hottbook}:  the empty type 
is denoted with $\textbf{0}$, the unit type with  $\textbf{1}$, the list type  constructor  with $\mathsf{List}$, the dependent product type  constructor with $\Pi$, the dependent sum constructor with $\Sigma$ and the sum type constructor with $+$. Further, we recall the notation of the following higher inductive types: propositional truncation is denoted with $\vert\vert A\vert\vert$ for any type $A$ and quotients  with $A/R$ for any homotopy set $A$ and an equivalence relation $R$. As usual,  the special cases of the type constructors $\Pi$ and $\Sigma$, when $B$ does not depend on $A$, are respectively denoted by $\rightarrow$ and $\times$. 

Voevodsky's $\textit{Univalence Axiom}$ states that  

\begin{center}
    $\mathrm{(UA)}\qquad \mbox{the map}\ \mathsf{idtoeqv}:(A=_{\mathcal{U}_i}B)\rightarrow (A\simeq B) $ is an equivalence
\end{center}
where \lq $\simeq$' denotes the type of $\textit{equivalences}$  and $\mathsf{idtoeqv}$ is the function which from a proof of equality of two types in the same universe $\mathcal{U}_i$,  for some index $i$, produces an equivalence, all as defined  in \cite{hottbook}.

 This in turn implies 

\begin{equation*}
(A=_{\mathcal{U}} B)\simeq (A\simeq B)
\end{equation*}

\iffalse
When $A$ and $B$ are homotopy sets, also called h-sets in the definition \ref{setdef} below), we can restrict ourselves to consider \lq $\simeq$' as the type of {\it isomorphisms}, where an isomorphism in \hott\ is defined as follows:  
\fi

\iffalse
Note that since $C$ is an h-set it is  validated   also the following propositional equality rule part of the higher order inductive schemata:
for $e,q : \vert\vert A\vert\vert $
$$ \mathsf{ap}_{ ind_{\vert\vert A\vert\vert} (c,s)  }(\mathsf{sq}_{A} (e,q)) =_C s( ind_{\vert\vert A\vert\vert} ( c,s) (e) \, \, ind_{\vert\vert A\vert\vert} (c,s)(q))$$
where  $\mathsf{ap}_f$ is defined as in \cite{hottbook}.
\fi

\noindent
We recall from \cite{hottbook} the following notations and definitions which characterize
h-sets and h-propositions by singling out some proof-terms (it does not matter which they are, it only matters that we can single out some of them!) proving the statements which will be used in the next sections:

$$\mathsf{isSet}(A) \ourdef \Pi_{x,y:A} \ \Pi_{p,q:x=_{A}y} \ p=_{Id_{A}}q  \qquad \qquad \mathsf{isProp}(A) \ourdef \Pi_{x,y:A}\  x=_{A}y $$

\begin{definition}
 A type $A$  is an h-proposition  if $\isprop(A) $ is provable in \hott.
\end{definition}

\begin{definition}\label{setdef}
A type $A$ is an h-set if $\mathsf{isSet}(A)$ is provable in \hott.
\end{definition}

\begin{lemma}\label{id}
If $A$ is an h-set, then $\mathrm{Id}_{A}$ is an h-proposition, i.e. there exists a proof-term

\begin{center}
    $\mathfrak{p}_{Id}: \Pi_{A:\guni}\ \Pi_{s: \isset(A)}\ \Pi_{a,b:A}\  \isprop(\mathrm{Id}_{A}(a,b))$
\end{center}

\end{lemma}

Since h-levels are cumulative (see Thm 7.1.7 \cite{hottbook}), in particular the following holds: 

\begin{lemma}\label{propintoset}
Every h-proposition is an h-set: i.e. there exists a proof-term

\begin{center}
    $\mathfrak{s}_{coe} : \Pi_{A:\mathcal{U}_{i}}\ \isprop(A)\rightarrow\isset(A)$.
\end{center}
\end{lemma}

Now we recall the notion of isomorphism between two h-sets:

\begin{definition}[Isomorphism between h-sets]\label{isomor}
Given two h-sets $A$ and $B$, a function $f:A\rightarrow B$ in \hott\  is an isomorphism if there exists $g: B\rightarrow A$ such that we can prove
\begin{equation*}
     \Pi_{x:A}\ \mathrm{Id}_{A}(g(f(x)), x)\ \times \ \Pi_{y:B}\ \mathrm{Id}_{B}(f(g(y)), y)
\end{equation*}

\end{definition}

We also recall the rules of the  {\it propositional truncation $\vert\vert A\vert\vert$  of a type $A$} 
\iffalse
which is isomorphic to
the higher inductive quotient $A/{\bf 1}$ on the trivial equivalence relation $R(a,b) \ourdef{\bf 1}$ for $a,b: A$,
\fi
 given in \cite{proptrunc}:
$\vert\vert A\vert\vert$ is a higher inductive type generated from the 
the  following two introductory constants
$$| -|: A\rightarrow \vert\vert A\vert\vert\qquad\qquad \mathsf{sq}_{A}:\Pi_{x,y:\vert\vert A\vert\vert }\  x=_{\vert\vert A\vert\vert}y  $$
by means of the elimination constructor:
$$
\mbox{E-$\vert\vert\ \vert\vert$ }\displaystyle{ \frac
         { \displaystyle  C: {\cal U}_i \ type \qquad e :\vert\vert A\vert\vert \qquad   c: C \ [x: A]  \qquad s: \Pi_{x,y:C}\  x=_{C}y  }
         {\displaystyle  ind_{\vert\vert A\vert\vert} (e,c,s) : \vert\vert A\vert\vert \rightarrow C}}
$$
satisfying the definitional equality rule
$$
\mbox{C-$\vert\vert\ \vert\vert$ }\displaystyle{ \frac
         { \displaystyle  C: {\cal U}_i \ type \qquad a : A \qquad   c: C \ [x: A]  \qquad s: \Pi_{x,y:C}\  x=_{C}y  }
         {\displaystyle  ind_{\vert\vert A\vert\vert} (|a|, c,s) \equiv c(a): C}}
$$

The presence of propositional truncation makes possible to represent logical notions  in a way alternative to the propositions-as-types paradigm by using h-propositions in a way similar to what happens in the internal dependent type theory  of a topos or of a regular theory as described in \cite{M05}.

In more detail,  in \hott\  the constant falsum $\perp$ is identified with {\bf 0}, the propositional  conjunction symbol   $\land$ with $\times$,  the universal quantifier symbol $\forall$ with $\Pi$, thanks to the following lemma derived from  \cite{hottbook}:

\begin{lemma}\label{firstcon}
The empty type {\bf 0} and the unit type {\bf 1} are h-propositions. Further, h-propositions are closed under $\times$ and $\Pi$ (and thus also $\rightarrow$), i.e.
 there exists the following proof-terms 
 
 \vspace{1.0em}
\begin{tabular}{l}
    $\mathfrak{p}_1: \isprop(\textbf{1})$\qquad $\mathfrak{p}_0: \isprop(\textbf{0})$\\[5pt]

    $\mathfrak{p}_\rightarrow: \Pi_{A,B: \guni}\ \Pi_{q :\isprop(B)}\ \isprop(A\rightarrow B)$\\[5pt]
    
    $\mathfrak{p}_\times: \Pi_{A,B: \guni}\ \Pi_{p:\isprop(A), q:\isprop(B)}\ \isprop(A\times B)$\\[5pt]
  
    $\mathfrak{p}_\Pi: \Pi_{A:\guni}\ \Pi_{B: A\rightarrow \guni}\ \Pi_{p: \Pi_{x:A} \isprop(B(x))} \isprop(\Pi_{x:A}B(x))$\\[5pt]

     $\mathfrak{p}_{\vert\vert\ \vert\vert}: \Pi_{A:\guni} \isprop (\vert\vert A \vert\vert)$\\[5pt]
     \end{tabular}
\end{lemma}
\begin{proof} See Chapter III in \cite{hottbook}. 
\end{proof}

Thanks to the notation introduced above we can define
$$\mathfrak{p}_{\vert\vert\ \vert\vert}\ourdef \lambda A. \mathsf{sq}(A)$$

Moreover, since h-propositions are not closed under $\Sigma$ and $\mathsf{+}$ (e.g. $\textbf{1}+\textbf{1}$ is not a h-proposition), we need to apply propositional truncation to define disjunction and existential quantification exactly as it happens in the internal dependent type theory of a topos \cite{M05}:  $P \vee Q$ is identified with $\vert\vert P + Q \vert\vert$ and $\exists_{x\in A}\ P(x)$ with $\vert\vert \Sigma_{x:A}\ P(x)\vert\vert$.

We recall introduction and elimination rules of disjunction
and existential quantifiers as defined in \hott\ to fix the notation and recall some properties:

\begin{definition}\label{disj}
The disjunction of h-propositions $P$ and $Q$ is defined as
$$P\vee Q\ourdef \vert\vert P + Q \vert\vert$$ 
Its canonical  introductory constructors are defined as follows: for $p: P$ and $q: Q$
$$\mathsf{inl}_\vee (p)\ourdef \vert \mathsf{inl}(p)\vert : P\vee Q  \qquad \mathsf{inr}_\vee (q)\ourdef \vert \mathsf{inr}(q)\vert: P\vee Q $$
and its  eliminator constructor is defined as follows:
for any $C$ such that $s:\isprop(C)$ , any $e: P\vee Q $ and any
$l_1(x) :  C\ [x: P]$ and $l_2(y) : C\ [y:Q]$ 
    $$\mathsf{ind}_{\vee }(e,x. l_1(x) ,\ y. l_2(y), s\ ) \ourdef \mathsf{ind}_{\vert\vert\  \vert\vert}   (e\, ,\,  z.  \mathsf{ind}_{+} (z,\, x. l_1(x) \, ,\, y. l_2(y) \ ), s\ ):C $$
\end{definition}

The disjunction as defined above satisfies the usual $\beta$-definitional equalities:
\begin{lemma}\label{betadisj}
The disjunction defined in def. \ref{disj} satisfies the  following  $\beta$ definitional equalities:
      for any $C$ such that $s: \isprop(C)$, any $p: P$ and $q:Q$,  and any
$l_1(x) :  C\ [x: P]$ and $l_2(y) : C\ [y:Q]$  it holds in \hott\ 
    $$\mathsf{ind}_{\vee }(\mathsf{inl}_\vee (p),x.l_1(x)\, , y.l_2(y), s)\equiv l_1(p): C\qquad
    \mathsf{ind}_{\vee}(\mathsf{inr}_\vee (q),x. l_1(x)\, , y.l_2(y), s)\equiv l_2(q): C$$
    \end{lemma}

    \begin{definition}\label{exist}
    For any h-set $A$ and any predicate or family of h-propositions
    $P(x) \ [x: A]$, the existential quantification  is defined as
    $$\exists_{x: A}\  P(x) \ \ourdef \vert\vert  \Sigma_{x:A}\ P(x)\vert \vert$$ 
Its canonical introductory constructor is defined as follows: for $a:A$ and $p: P(a)$
$$(a,_\exists p)\ourdef \vert (a,p)\vert: \exists_{x: A}\  P(x) $$
and its elimination constructor in turn as follows: for any $C$ such that $s: \isprop(C)$, any $ e: \exists_{x: A}\  P(x)$ and any
$l(x,y) :  C\ [x: A, y: P(x)]$ 
    $$\mathsf{ind}_{\exists}(e,\, x.y.l(x,y), s\ ) \ourdef \mathsf{ind}_{\vert\vert\ \vert\vert}   (e\, ,\,  z.  \mathsf{ind}_{\Sigma} (z, \, x.y.l(x,y) \, ), s\ ):C $$
    \end{definition}
    
    The existential quantification as defined above satisfies the usual $\beta$-definitional equality:
\begin{lemma}\label{betaex}
The existential quantifier defined in def. \ref{exist} satisfies the  following  $\beta$ definitional equality:
      for any $C$ such that $s:\isprop(C)$, any $a: A$ and $p: P(a)$ and any $q:Q$ and $l(x,y) :  C\ [x: A, y: P(x)]$   it holds in \hott\   
    $$\mathsf{ind}_{\exists} ((a,_\exists p),\, x.y.l(x,y), s\ )\equiv l(a,p): C$$
    \end{lemma}

We also encode  the fact that the disjunction $\lor$ and the existential quantifier $\exists$ are h-propositions by means of the following proof-terms: 

\begin{center}
    \begin{tabular}{l}
      $\mathfrak{p}_\lor\, \ourdef\,  \lambda A, B.\ \mathfrak{p}_{\vert\vert\ \vert\vert}(A+
      B) :  \Pi_{A, B: \guni}\  \isprop(A\lor 
      B)$\\[5pt]
       
      $\mathfrak{p}_\exists\, \ourdef\,  \lambda A, B.\ \mathfrak{p}_{\vert\vert\ 
     \vert\vert} (\Sigma_{x : A} \ B(x)): \Pi_{A:\guni}\ \Pi_{B:A\rightarrow\guni}\  \isprop(\exists_{x:A}B(x))$\\[5pt]
      
\end{tabular}
\end{center}

\vspace{1.0em}

It is worth to recall from \cite{hottbook} that  the notion of type equivalence of h-propositions coincides with that of logical equivalence:
\begin{lemma}
Two h-propositions $P$ and $Q$ are equivalent as types, namely $P\simeq Q$ holds, if and only if they are logically equivalent, namely $P\leftrightarrow Q$,
and by Univalence,  also $P=_{\guni} Q$ holds for $P, Q $ in $\guni$.
\end{lemma}

Further, we can state the following basic lemma: 

\begin{lemma}\label{invprop}
If $P: \guni$ and $s: \isprop(P)$, then   $\vert -\vert : P\rightarrow \vert\vert P\vert\vert$ is an isomorphism, i.e. there is an inverse $\vert -\vert ^{-1}:\vert\vert P\vert\vert \rightarrow P$ which satisfies $\vert- \vert\circ \vert -\vert^{-1} =_{\vert\vert P\vert\vert} \mathsf{id}_{\vert\vert P\vert\vert} $ and  $\vert -\vert^{-1} \circ \vert-\vert =_P \mathsf{id}_P $.
Therefore $P=_{\guni}\vert\vert P\vert\vert$ holds.
\end{lemma}
\begin{proof}
We can simply define $\vert z\vert^{-1}\ \ourdef\ \mathsf{ind}_{\vert \vert\ \vert\vert} (z, (x).x, s)$ since $P$ is a
h-proposition. Note that for any $z:\vert\vert P\vert\vert $ it is validated $\vert (\vert z\vert^{-1})\vert =_{\vert \vert P\vert\vert }z $
 only propositionally while $\vert (\vert p\vert )\vert^{-1} \equiv p: P $ holds for any $p: P$.
The rest follows by Univalence and because $P$ is an h-proposition.
\end{proof}

\begin{remark}\label{ptrunc}
Lemma~\ref{invprop} is crucial to provide a \lq\lq canonical presentation'' of all h-propositions up to propositional equality in terms of $\vert\vert A\vert\vert$ for
some type $A$ thanks to the fact that  the operator $\vert\vert -\vert\vert$ is {\it extensionally idempotent} as follows
from proposition~\ref{invprop}.
Therefore we could interpret also the conjunction, implication and universal quantifiers  as follows
\begin{center}
\begin{tabular}{lll}
$  P\wedge Q$ &  $\ourdef$  & $  \vert\vert P \times Q\vert\vert$ \\[5pt]
 $  P\rightarrow Q$ &  $\ourdef $ &  $ \vert\vert P \rightarrow Q\vert\vert$ \\[5pt]
 $\forall_{x:A} \ P(x)$ & $\ourdef$ & $ \vert\vert \Pi_{x:A} \ P(x) \vert\vert$ \\
\end{tabular}
\end{center}

Accordingly, the following proof-terms witness that they are h-propositions: 

\begin{center}
    \begin{tabular}{lll}
    $\mathfrak{p}_{\vert\vert \times\vert\vert}$ & $\ourdef$ & $\lambda A, B.\ \mathfrak{p}_{\vert\vert\ \vert\vert}(A\times B): \Pi_{A,B: \guni} \isprop(\vert\vert A\times B\vert\vert)$\\[5pt]
    
    $\mathfrak{p}_{\vert\vert \rightarrow\vert\vert}$ & $\ourdef$ & $\lambda A, B.\ \mathfrak{p}_{\vert\vert\ \vert\vert}(A\rightarrow B): \Pi_{A,B: \guni} \isprop(\vert\vert A\rightarrow B\vert\vert)$\\[5pt]
    
    $\mathfrak{p}_{\vert\vert \Pi\vert\vert}$ & $\ourdef$ & $\lambda A, B.\ \mathfrak{p}_{\vert\vert\ \vert\vert}(\Pi_{x:A}B(x)): \Pi_{A: \guni} \Pi_{B: A\rightarrow \guni} \isprop (\vert\vert \Pi_{x:A}B(x)\vert\vert)$\\[5pt]
    \end{tabular}
\end{center}
\end{remark}

\begin{definition}\label{cong}
Given $a: A$ and $b:B$ and $c: \vert\vert A\times B\vert\vert$ and $s:\isprop(A)$ and $t: \isprop(B)$ we define
$$\begin{array}{ll}(a,_\wedge b)\ \ourdef \ \vert (a,b)\vert: \vert\vert A\times B\vert\vert & \qquad {\fpr}_\wedge (c)\ \ourdef\ \mathsf{ind}_{\vert\vert \ \vert\vert}(\, c, x.\fpr(x), s\ )\\
& \qquad{\spr}_\wedge (c)\ \ourdef\ \mathsf{ind}_{\vert\vert \ \vert\vert}(\, c, x.\spr(x), t\ )
\end{array}$$
\end{definition}
\begin{definition}\label{imp}
Given $a: A$  and $b:B $ and $c: \vert\vert A\rightarrow B\vert\vert$ and $s: \isprop(B)$ we define
$$\begin{array}{ll} \lambda_\rightarrow x. b \ \ourdef \ \vert \lambda x. b \vert: \vert\vert A\rightarrow B\vert\vert & \qquad c_{\rightarrow} (a) \ourdef\ \mathsf{ind}_{\vert\vert \ \vert\vert}(\, c, x. x(a), s\ )\\
\end{array}$$
\end{definition}
\begin{definition}\label{forall}
Given $a: A$  and $b:B(x)\ [x:A] $ and $c: \vert\vert \Pi_{x : A}\ B(x)\vert\vert$ and $s: \isprop(B(a))$ we define
$$\begin{array}{ll} \lambda_\forall x. b \ \ourdef \ \vert \lambda x. b \vert: \vert\vert \Pi_{x : A}\ B(x)\vert\vert & \qquad c_{\forall} (a) \ourdef\ \mathsf{ind}_{\vert\vert \ \vert\vert}(\, c, x. x(a), s\ )\\
\end{array}$$
\end{definition}

\begin{lemma}\label{beta}
The usual $\beta$-definitional equalities for the projections of conjunctions in def.\ref{cong}
$${\fpr}_\wedge ( (a,_\wedge b)\  )\ \equiv\ a \qquad\qquad {\spr}_\wedge ( (a,_\wedge b)\  )\ \equiv\ b $$
for functions of implications  in def.\ref{imp} and universal quantifiers in def.\ref{forall}
$$(\lambda_\rightarrow x. b )_{\rightarrow}(a)\ \equiv\ b[a/x]\qquad \qquad(\lambda_\forall x. b )_{\forall}(a)\ \equiv\ b[a/x] $$
according to the notion of substitution in the appendix of \cite{hottbook}, all hold in \hott.
\end{lemma}
\begin{proof}
They follow by elimination of the truncation and usual  $\beta$-definitional equalities for the corresponding types under truncation.
\end{proof}

We will crucially use the fact  that h-sets are closed under the following type constructors:

\begin{lemma}\label{setclosure}
H-sets are closed under $\Pi$ (and hence $\rightarrow$), $\Sigma$ (and hence $\times$), and $\mathsf{+}$ and $\mathsf{List}$.
Furthermore, for any h-set $A$ and any equivalence relation $R$ defined as an h-proposition, then the higher quotient type $A/R$ is an h-set. Therefore, the following proof-terms exist: \\
\vspace{1.0em}
\begin{tabular}{l}
\\    
$\mathfrak{s}_1: \isset(\textbf{1})$\qquad $\mathfrak{s}_0: \isset(\textbf{0})$\qquad $\mathfrak{s}_{\mathbb{N}}: \isset(\mathbb{N})$\\[5pt]
$\mathfrak{s}_\Pi: \Pi_{A:\guni}\ \Pi_{B: A\rightarrow \guni}\ \Pi_{s: \Pi_{x:A} \isset(B(x))} \isset(\Pi_{x:A}B(x))$\\[5pt]
      $\mathfrak{s}_\Sigma: \Pi_{A:\guni}\ \Pi_{B: A\rightarrow \guni}\ \Pi_{s: \isset(A)}\  \Pi_{t: \Pi_{x:A} \isset(B(x))} \isset(\Sigma_{x:A}B(x))$\\[5pt]
       $\mathfrak{s}_+: \Pi_{A,B: \guni}\ \Pi_{s: \isset(A)}\ \Pi_{t: \isset(B)}\ \isset(A +
       B)$\\[5pt]
       $\mathfrak{s}_{\mathsf{List}}: \Pi_{A: \guni} \Pi_{s:\isset(A)}\ \isset(\mathsf{List}(A))$\\[5pt]
       $\mathfrak{s}_{Q}: \Pi_{A: \guni}\ \Pi_{R: A\rightarrow A\rightarrow\guni}\ \Pi_{s: \isset(A)}\ \Pi_{p: \isprop(R)}\ \Pi_{r:  \mathrm{equiv}(R)}\ \isset(A/R)$\\

\end{tabular}

where $\mathrm{equiv}(R)$ is an abbreviation for the fact that $R$ is an equivalence relation.
\end{lemma}
\vspace{1.0em}

For any natural number index $i$, the type of h-sets within $\mathcal{U}_{i}$ is defined as follows
$$\set_{\guni}\ourdef \Sigma_{(X: \guni)}\ \mathsf{IsSet}(X)$$

%Further, we can show that h-sets are closed also under the type constructor $\mathsf{List}$: 
%
%\begin{lemma}
%If $A$ is an h-set, then $\mathsf{List}(A)$ is an h-set. 
%\end{lemma}
%

\begin{remark}\label{eff-quot}

The lemma \ref{setclosure} follows from \cite{rijke2015sets} where more abstractly it is shown that 
the category of h-sets and  functions within \hott\ equated under propositional equality,  is   a locally cartesian closed pretopos with  well-founded trees, or W-types, as defined in \cite{wtyp}. 
In particular note that set-quotients satisfy {\em effectiveness} in the sense that,  given the quotient function $\mathsf{q}: A \rightarrow A/R$
sending an element $a$ of $A$  to its equivalence class $\mathsf{q}(a): A/R$, for any $a,b: A$ it follows
$\mathsf{q}(a)=_{A/R} \mathsf{q}(b)\ \leftrightarrow \ R(a,b)$ (see 10.1.3 in \cite{hottbook}).
\end{remark}

%a natural numbers object.
%a locally cartesian closed category with disjoint finite coproducts and effective equivalence relations as shown in
%   
%If we consider a version of \hott\ with $\mathcal{W}$-types, then {\sf Set} can be equipped with  initial algebras for polynomial endofunctors.
%
%The following definition is taken from \cite{M05}: 
%
%\begin{definition}
%A pretopos is a category equipped with finite limits, stable finite disjoint coproducts and stable effective quotients of monic equivalence relations.
%\end{definition}
%
%Therefore, we can show that
%
%\begin{theorem}
%The category $\mathsf{Set}$ is a  locally cartesian closed pretopos with a natural numbers object.
%
%\end{theorem}
%
%A locally cartesian closed pretopos is also called $\Pi$-pretopos. If we include $\mathcal{W}$-types among our constructors, then we can prove that the category {\sf Set} is a $\Pi\mathsf{W}$-pretopos (see \cite{rijke2015sets}). 

Another key property of \hott, missing in \mltt,  which we will crucially employ to interpret \mtt-collections of small propositions  and \emtt-power-collections of subsets of a set,
is that h-sets are closed under a  {\it sub-universe classifier} $ \mathsf{Prop}_{\mathcal{U}_{0}}$ of those h-propositions living in the universe $\funi$
\begin{equation*}
    \mathsf{Prop}_{\mathcal{U}_{0}}:= \Sigma_{(X:\mathcal{U}_{0})}\ \isprop (X)
\end{equation*}

Indeed, from section 2 of \cite{rijke2015sets} it follows:

\begin{lemma}\label{classifier}
$\mathsf{Prop}_{\mathcal{U}_{0}}$ is an h-set.
\end{lemma}

The proof-term inhabiting $\isset ( \mathsf{Prop}_{\mathcal{U}_{0}})$ is denoted by $\mathfrak{s}_{\mathsf{Prop_0}}$. 

\begin{remark}
However, $\mathsf{Prop}_{\mathcal{U}_{0}}$ is not \lq small', since it is not a type in $\mathcal{U}_{0}$, but it lives in a higher universe (see section 10.1 in \cite{hottbook}).
This is compulsory to keep \hott\ predicative.
\end{remark}

Further, we can  {\it  assume}  that {\it  if $A:\guni$, then $A/R:\guni$} 
motivated by  the cubical interpretation of higher inductive types given in \cite{CHM18}.

Moreover, h-sets within a universe $\mathcal{U}_i$ of \hott\ can be organized into a category $\set_{{\mathcal{U}_i}}$
as defined in \cite{hottbook}.

It is known that the principle of indiscernibility of identicals can be derived in type theory from the elimination rule for propositional equality. Such principle is called {\it transport} in \cite{hottbook} and says that, given a type family $P$ over $A$ and a proof $p: x=_A y$, there exists a map $\trasp(p,-): P(x) \rightarrow P(y)$. In particular, the following property holds for transport, that will turn out to be useful later: 

\begin{lemma}\label{apd-tr}
Suppose $f: \Pi_{(x:A)} B(x)$. Then there exists a map

\begin{equation*}
\mathsf{apd}_f: \Pi_{(p: x=_A y)} (\trasp(p,f(x))=_{B(y)} f(y))
\end{equation*}

\end{lemma}

\begin{proof}
The proof is a simple application of the elimination rule for propositional equality. 
\end{proof}

Finally, we recall two principles of \hott\ that we will crucially use
to meet our goals. One is the {\it propositional extensionality principle} which is an instance of the Univalence Axiom applied to h-propositions in the first universe $\funi$:

\begin{equation*}
\mathsf{propext}: \Pi_{P,Q:\mathsf{Prop}_{\mathcal{U}_{0}}}(P\leftrightarrow Q) \rightarrow (P=_{\funi}Q).
\end{equation*}
The other is 
the principle of {\it function extensionality}  for h-sets:
\begin{equation*}
\mathsf{funext}: (\Pi_{x:A}(f(x)=_{B(x)}g(x)))\ \rightarrow\  f=_{\Pi_{x:A}B(x)}g.
\end{equation*}

More precisely,  we will use function extensionality applied to h-sets up to those within the second universe $\suni$.
The reason is that, while sets of both \mtt\ and \emtt\  will be interpreted as h-sets in the first universe
$\funi$, collections of both \mtt\ and \emtt\ will be interpreted as  h-sets
at most in the second universe $\suni $.

\section{The compatibility of \mtt\ with \hott}
The main aim of the present section is to show that the intensional level  \mtt\ of \mf\ is compatible with \hott, according to the definition of compatibility given in section 2. In order to achieve this result, we need to make use of many new tools introduced in the context of \hott\ and not available in \mltt.

Indeed, the resulting interpretation  must be contrasted with the interpretation of \mtt\ in \mltt\ outlined in \cite{m09}: there the notion of proposition is identified with the notion of
set, while here we are going to interpret \mtt-propositions as h-propositions.

It is well known that the interpretation of  dependent type theories \`a  la Martin-L{\"o}f must be done by induction on the raw syntax of \mtt-judgements since types and terms
are recursively defined in a mutual way together with their definitional equalities.

Then, we can define a partial interpretation $(J)^\sqbullet$
 by induction on the associated raw syntax of \mtt-types and terms  in the raw syntax of types and terms in that of \hott\ as follows: we interpret all  types of \mtt\  including proper \mtt-collections  as h-sets, where the \lq\lq smallness'' character of \mtt-sets is captured by h-sets living in
 the  first universe $\funi$. Hence, \mtt-sets  and \mtt-small propositions are interpreted as h-sets and h-propositions in $\mathcal{U}_{0}$. On the other hand, \mtt-collections and \mtt-propositions are interpreted, respectively, as h-sets and h-propositions in $\mathcal{U}_1$.
 
\begin{definition}[Interpretation of \mtt-syntax]\label{mtt-int}

We define this interpretation as an instantiation of a partial interpretation of the raw syntax of types and terms of \mtt\ in those of \hott\
$$(-)^\sqbullet:  \mbox{Raw-syntax }(\mtt) \ \longrightarrow \ \mbox{Raw-syntax }(\hott) $$
assuming to have defined two auxiliary partial functions:
one meant to associate  to some type symbols of \hott\   a proof-term expressing that they are   h-propositions
$$\prp(-):  \mbox{Raw-syntax }(\hott) \ \longrightarrow \ \mbox{Raw-syntax }(\hott) $$
and another meant to associate to some type symbols of \hott\   a proof-term expressing that they are h-sets
 $$\prs(-):  \mbox{Raw-syntax }(\hott) \ \longrightarrow \ \mbox{Raw-syntax }(\hott) $$
by relying on proofs given in lemmas~\ref{firstcon} and \ref{setclosure} taken from \cite{hottbook} and \cite{rijke2015sets}.  

We then extend $(-)^\sqbullet$ to contexts of \mtt\ in those of \hott\ as follows: $([\ ])^{\sqbullet}$ is defined as the empty context $\cdot$  in  \hott\ and $(\Gamma, x: A)^{\sqbullet}$ is defined as $\Gamma^{\sqbullet}, x: A^{\sqbullet}$. 
Also the assumption of variables is interpreted as the assumption of variables in \hott:  $(x\in A\ [\Gamma])^{\sqbullet}$ is interpreted as $x: A^{\sqbullet}\ [\Gamma^{\sqbullet}]$, provided that $x: A^{\sqbullet}$ is in $\Gamma^{\sqbullet}$.

Then, the  \mtt-judgements  are interpreted as follows:  

\begin{center}
\begin{tabular}{|lcl|}
\hline
&&\\[0.5pt]
    $(A\ set\ [\Gamma])^{\sqbullet} $  &is  defined as &  $A^\sqbullet :\funi \ \gam$  such that $\prs(A^\sqbullet): \isset(A^\sqbullet)$ is derivable \\[5pt]
    
     $(A\ col\ [\Gamma])^{\sqbullet} $  &is  defined as &  $A^\sqbullet :\suni \ \gam$  such that $\prs(A^\sqbullet): \isset(A^\sqbullet)$ is derivable \\[5pt]
    
    $(P\ \textit{prop}_{s}\ [\Gamma])^{\sqbullet
    } $ & is defined as & $ P^\sqbullet: \funi \ \gam$ such that $\prp(P^\sqbullet): \isprop(P^\sqbullet)$ is derivable \\[5pt]
    
    $(P\ \textit{prop}\ [\Gamma])^{\sqbullet} $ & is defined as & $ P^\sqbullet: \suni \ \gam$ such that $\prp(P^\sqbullet): \isprop(P^\sqbullet)$ is derivable  \\[5pt]
   
     $(A=B \ set \ [\Gamma])^{\sqbullet}$ & is defined as & $  (A^{\sqbullet}, \prs(A^\sqbullet))\ \equiv  (B^{\sqbullet}, \prs(B^\sqbullet)) : \set_{\funi} \ \gam$  \\[5pt]
     
     $(A=B \ col \ [\Gamma])^{\sqbullet}$ & is defined as & $  (A^{\sqbullet}, \prs(A^\sqbullet))\ \equiv  (B^{\sqbullet}, \prs(B^\sqbullet)) : \set_{\suni} \ \gam $   \\[5pt]
     
     $(P=Q\ \textit{prop}_{s}\ [\Gamma])^{\sqbullet} $ & is defined as & $ (P^\sqbullet, \prp (P^\sqbullet)) \equiv (Q^\sqbullet, \prp (Q^\sqbullet)):  \prop_{\funi} \ \gam$  \\[5pt]
     
    $(P=Q\ \textit{prop}\ [\Gamma])^{\sqbullet}$ & is defined as & $(P^\sqbullet, \prp (P^\sqbullet)) \equiv (Q^\sqbullet, \prp (Q^\sqbullet)): \prop_{\suni} \ \gam$  \\[5pt]
      $(a\in  A\ \ [\Gamma])^{\sqbullet}$ & is defined as & $a^{\sqbullet}: A^{\sqbullet}\ \gam $\\[5pt]
         
 $(a=b\in  A\ [\Gamma])^{\sqbullet}$ &is defined as & $a^\sqbullet \equiv b^\sqbullet : A^\sqbullet \ \gam$\\[5pt]
 \hline
\end{tabular}
\end{center}
%where $\fpr(z): A\ [z: \Sigma (x:A)B(x)]$ and  $\spr(z): B (\fpr(z) )\  [z: \Sigma (x:A) B(x)]$ are the two projections associated with the $\Sigma$-type.
%
%\begin{remark}
%After recalling that in \mtt, if  the judgement
%$(A=B \ type \ [\Gamma])$ is derivable where.  , then also $A\ \ type \ [\Gamma]$ and $B\ \ type \ [\Gamma]$   are derivable
%\end{remark}

 %Therefore, in the following, when we give the interpretation for propositions we assume that there is also a chosen proof-term witnessing their coercion into sets. This fact is crucial to interpret correctly in \hott\ the subtyping rules of \mtt. 
 
The interpretation of the raw types and terms of \mtt\ as raw types and terms of \hott\ is spelled out in the Appendix A.  
\end{definition}
\vspace{1.0em}

The following substitution lemmas state that substitution on types and terms in \mtt\ corresponds to substitution on types and terms in \hott: 

\begin{lemma}\label{sub-t}
If $A$ is a raw-type in \mtt, $b$ is a \mtt\ raw-term and $x$ is a variable occurring free in $A$, then $$(A [b/x])^\sqbullet\ \ourdef\ A^\sqbullet [b^\sqbullet / x^\sqbullet].$$ 

If $a$ and $b$ are \mtt\ raw-terms and $x$ is a variable occurring free in $a$, then $$( a[b/x])^\sqbullet\ \ourdef\ a^\sqbullet [b^\sqbullet/x^\sqbullet].$$
\end{lemma}

\begin{theorem}[Validity]\label{mtt-val}
If $\mathcal{J}$ is a derivable judgement in \mtt, then the interpretation of $\mathcal{J}$ holds in \hott.
 Moreover, if $P\ prop\ [\Gamma]$ and $P\ prop_s\ [\Gamma]$ are derivable judgements in \mtt, then $\prs(P^\sqbullet):\isset(P^{\sqbullet})\ [\Gamma^{\sqbullet}]$ is derivable in \hott\ and $\prs(P^\sqbullet)\ \ourdef\ \mathfrak{s}_{coe}((P)^\sqbullet, \prp(P^\sqbullet))$.
\end{theorem}
    
\begin{proof}
The proof is by induction over the derivation of $J$. 

\noindent
The validity of judgements forming \mtt-sets follows from the definitions given above, the lemmas \ref{firstcon}, \ref{setclosure} 
and the closure of the first universe $\funi$ under set-theoretic constructors as in \cite{PMTT}. 

\noindent
The subtyping rules 

\begin{center}
    \AxiomC{$P\ prop_s\ [\Gamma]$}\RightLabel{${\bf prop_s\mbox{-}into\mbox{-}set}$}\UnaryInfC{$P\ set\ [\Gamma]$}\DisplayProof 
    \hskip1.0em\qquad
    \AxiomC{$P\ prop\ [\Gamma]$}\RightLabel{${\bf prop\mbox{-}into\mbox{-}col}$}\UnaryInfC{$P\ col\ [\Gamma]$}\DisplayProof
\end{center}
 are interpreted as follows: by induction hypothesis, $P^\sqbullet: \funi\ \gam$ and $\prp(P^\sqbullet): \isprop(P^\sqbullet)\ [\Gamma^\sqbullet]$; moreover, we have also $\prs(P^\sqbullet): \isset(P^\sqbullet)\ [\Gamma^\sqbullet]$, which is given  by $\mathfrak{s}_{coe}(P^\sqbullet, \prp(P^\sqbullet))$, and thus the conclusion follows. The other subtyping rule is validated by a similar argument. 

\noindent
The rules ${\bf prop_s\mbox{-}into\mbox{-}prop}$ and ${\bf set\mbox{-}into\mbox{-}col}$ are trivially validated by cumulativity of universes  and by definition of the interpretation. 

\noindent
The definition of the interpretation for judgemental equalities trivially validates the conversion rules of \mtt. In particular, those for \mtt-disjunction and existential quantifier follow from lemmas~\ref{betadisj} and \ref{betaex}.

\noindent
The collection of small proposition $\mathsf{prop_s}$ is interpreted as $\mathsf{Prop}_{\funi}: \suni$ with $\mathfrak{s}_{\mathsf{prop_0}}: \isset(\mathsf{Prop}_{\funi})$.

\noindent
Note that {\it the  validity of the  encoding of \mtt-small propositions
satisfies the usual compatibility rules} like
\begin{center}
    \AxiomC{$p_1=p_2\in \mathsf{prop_s} \ [\Gamma]$} \AxiomC{$q_1=q_2\in \mathsf{prop_s} \ [\Gamma]$}
\BinaryInfC{$p_1\land q_1= p_2\land q_2 \in \mathsf{prop_s}\ [\Gamma] $}
\DisplayProof
\vspace{0.5em}
\end{center}
since the interpretation of the encoding of small propositions into $\mathsf{Prop_s}$  is carried out by using the partial function $\prp(-)$  associating to the \hott-type $  \fpr(p^\sqbullet)\times \fpr(q^\sqbullet)$ 
the  proof-term  $\mathfrak{p}_{\times}(\fpr(p^\sqbullet),\fpr (q^\sqbullet ) , \spr(p^\sqbullet), \spr(q^\sqbullet)\, ):
\isprop( \fpr(p^\sqbullet)\times \fpr(q^\sqbullet)  )$.

In this sense {\it  the interpretation $(-)^\sqbullet$ depends on the chosen proof-terms of lemmas} \ref{id}, \ref{propintoset}, \ref{firstcon}, \ref{setclosure}, \ref{classifier} and {\it definitions} \ref{disj}, \ref{exist}. 

Moreover, the rule for the decoding operator $\tau$

\begin{center}
    \AxiomC{$p\in \mathsf{prop_s}\ [\Gamma]$}
    \RightLabel{\mbox{$\tau$-Pr} }
    \UnaryInfC{$\tau(p)\ prop_s\ [\Gamma]$}
    \DisplayProof
\end{center}
 is validated by our interpretation, since the premise is interpreted as $p^\sqbullet: \mathsf{Prop}_{\funi}\ [\Gamma^\sqbullet]$ and thus it follows that $\fpr(p^\sqbullet): \funi\ [\Gamma^\sqbullet]$ and $\spr(p^\sqbullet): \isprop(\fpr(p^\sqbullet))\ [\Gamma^\sqbullet]$, which is the interpretation of the conclusion by our definition. 

\noindent
Then, observe that the encoding rules are validated by construction. We just spell out the validity of the rule
\begin{center}
    \AxiomC{$p\in \mathsf{prop_s} \ [\Gamma]$}
\AxiomC{$q\in \mathsf{prop_s}\ [\Gamma] $}
\RightLabel{\mbox{Pr$_4$}}
\BinaryInfC{$p\widehat{\land} q\in \mathsf{prop_s}\ [\Gamma]$}
\DisplayProof
\vspace{0.5em}
\end{center}
  We know that $(p\in \mathsf{prop_s} \ [\Gamma] )^\sqbullet\ \ourdef \ p^\sqbullet: \mathsf{Prop}_{\funi}\ \gam$ and  that
   $ (q\in \mathsf{prop_s} \ [\Gamma] )^\sqbullet\ \ourdef \ 
 q^\sqbullet: \mathsf{Prop}_{\funi}\ \gam$ by inductive hypothesis.  Hence, $\fpr(p^\sqbullet): \funi\ \gam$ and $\fpr(q^\sqbullet): \funi\ \gam$ with $\spr(p^\sqbullet): \isprop(\fpr(p^\sqbullet))\ \gam$ and $\spr(q^\sqbullet): \isprop (\fpr(p^\sqbullet))\ \gam$, from which we can derive $\fpr(p^\sqbullet)\times \fpr(q^\sqbullet): \funi\ \gam$ with $\mathfrak{p}_{\times}(\fpr(p^\sqbullet), \fpr(q^\sqbullet), \spr(p^\sqbullet), \spr(q^\sqbullet)): \isprop( \fpr(p^\sqbullet)\times \fpr(q^\sqbullet))\ \gam$. This lets us conclude that  $(p\widehat{\land} q\in \mathsf{prop_s}\ [\Gamma])$  is well defined.

 \noindent
 Finally, the conversion rules associated to the decoding operator are all easily validated by construction as well. We just spell out the validity of the rule
 
\begin{center}
    \AxiomC{$p\in \mathsf{prop_s}\ [\Gamma]$}
    \AxiomC{$q\in \mathsf{prop_s}\ [\Gamma]$}
    \RightLabel{eq-\mbox{Pr$_4$ }}
    \BinaryInfC{$\tau(p \widehat{\land} q)= \tau(p) \land \tau (q)\ prop_s \ [\Gamma]$}
    \DisplayProof
\end{center}
 Indeed, assuming the premises as valid, we get that, since $(\tau(p \widehat{\land} q) [\Gamma])^\sqbullet\  \ourdef\ \fpr (\, (p \widehat{\land} q)^\sqbullet \, ): \funi\ [\Gamma^\sqbullet]$ with $\spr(\, (p \widehat{\land} q)^\sqbullet \, ): \isprop(\fpr (\, (p \widehat{\land} q)^\sqbullet \, ))$, but 
 $\fpr (\,  (p \widehat{\land} q)^\sqbullet \, ) \equiv (\fpr(p^\sqbullet) \times \fpr (q^\sqbullet)): \funi \ \gam $ 
and $ (\tau(p) \land \tau (q)\ [\Gamma])^\sqbullet\ \ourdef \ (\fpr(p^\sqbullet) \times \fpr (q^\sqbullet)): \funi\ [\Gamma^\sqbullet]$ with $\spr((\tau(p) \land \tau (q)\ [\Gamma])^\sqbullet ): \isprop(\fpr(p^\sqbullet) \times \fpr(q^\sqbullet))$,
 then the validity of the definitional equality $\tau(p \widehat{\land} q)= \tau(p) \land \tau (q)\ prop_s \ [\Gamma]$   trivially follows.

\end{proof}

 \begin{remark}
 The interpretation of two definitionally equal \mtt-types results in definitionally equal pairs in \hott- that is, not only the corresponding types in \hott\ are definitionally equal, but also the associated proof-terms witnessing that such types are h-sets or h-propositions. The fact that the interpretation depends on chosen proof-terms as observed above is fundamental to achieve this result. Also the validity of the coercion of propositions into sets relies on this fact. 
 \end{remark}
 
\begin{remark}[\it Alternative interpretation of \mtt\ in \hott] Observe that it is possible to define an alternative interpretation of \mtt\ in \hott\, which also implies the compatibility of the first with the latter. In this interpretation, we take the truncated version of {\it all} h-propositions as interpretations of \mtt-propositions. This choice will be compulsory later when we will define the interpretation for \emtt, since there we shall take into account \can\ isomorphisms between h-propositions. We refer to the Appendix B for the definition of this alternative interpretation. 
 \end{remark}

\section{Canonical isomorphisms  and the category  ${\mathsf{Set}_{mf}}/{\cong_c}$}
In this section, we inductively define a set of \can\ isomorphisms over \hott\  in order to be able to define a category, 
called $\setis$, of h-sets and functions up to canonical isomorphisms.
 This category could be  formalized within \hott\  as a H-category in the sense of \cite{Pal17}  provided that we extend \hott\ with the inductive type of canonical isomorphisms,
 or, alternatively, it could be simply defined in the meta-theory as done in \cite{m09}.
The category $\setis$ will be  used to interpret the extensional level of \emtt\ in \hott: its role will be the same as that of the category \cqis\  built over \mtt\ in \cite{m09} to interpret
\emtt\ within \mtt.

\begin{definition}\label{iso}
An indexed isomorphism $\mu_{A}^{B}:A\rightarrow B \  [\Gamma]$ is an isomorphism from the h-set $A$ to the h-set $B$ under the context $\Gamma$ with an inverse $(\mu_A^B)^{-1}: B\rightarrow A\ [\Gamma]$ which satisfies

\begin{equation*}
     \Pi_{x:A}\ \mathrm{Id}_{A}(  \ ( \mu_A^B)^{-1}(\mu_{A}^{B}(x)), x)\ \times \ \Pi_{y:B}\ \mathrm{Id}_{B}(\mu_{A}^{B}(\mu_A^B)^{-1}(y)), y)
\end{equation*}

\end{definition}

\begin{definition}
Given a dependent type $B \ [\Gamma]$ let us define the notion
of {\it transport} by induction on the number of assumptions in $\Gamma$:

\begin{enumerate}
\item 
If $\Gamma$ is empty, there are no  transports;
\item
If $\Gamma\ \ourdef \Delta, x: E$  and $B\ \ourdef\  C(x)$ then
a  {\it  transport operation } it is simply
$$ \trasp^1(p, -): C(x)\ \rightarrow C(x') \ [\Delta, x: E, x': E, p: x=_E x'  ]$$
where $\trasp^1(p, -)\ \ourdef \trasp(p, -)$ and $\trasp(p, -)$ is the usual transport
map as given in Section 2.

\item
If $\Gamma\ \ourdef \Delta, x: E, y_1:D_1, \dots y_n: D_n$  with $n\geq 1$ and $B\ \ourdef\  C(x, y_1,\dots, y_n)$ then
a  {\it  transport  operation } it is simply
$$ \begin{array}{l}
\trasp^{n+1}(p, -): C(x,y_1,\dots, y_n) \ \rightarrow\  C(x', \trasp^1(p,y_1), \dots
\trasp^{n} (p, y_n)) \\
\qquad [\Delta, x: E, y_1:D_1, \dots y_n: D_n,  x': E, p: x=_E x'  ]\end{array}$$
where $\trasp^{n+1}(p, -)\ \ourdef \ \mathsf{ind}_{Id} (p, x. (\lambda w. w))$ is defined by
eliminating toward 
$$C(x,y_1,\dots, y_n) )\ \rightarrow\  C(x', \trasp^1(p, y_1), \dots,
\trasp^{n}(p, y_n)) $$
\end{enumerate}
\end{definition}

To avoid an heavy notation in the following we simply write  $\trasp (p,-)$
instead of $\trasp^n (p,-)$ when it is clear from the context which is the transport map.

\begin{remark}
Since we are concerned with h-sets $A: \guni\ [\Gamma]$, the transport operations $\trasp(p,-)$ do not depend on the proof-term $p$. 
\end{remark}

\iffalse
\begin{enumerate}

\item
If $\Gamma\ \ourdef \ y: D,  x: A(y)$  and $B\ \ourdef \ C(y, x)$ 
a  {\it dependent transport} it is simply
$$ \begin{array}{l}
\trasp(p,q, -): C(y, x)\ \rightarrow\ C( y',  x') \\
\qquad [y:D,   x: A(y), y':D,   x': A (y') , p: y=_{D}y', q: \trasp(p, x)=_A(y') x'  ]\end{array}$$
where $\trasp(p,q) \ \ourdef\   \trap (q,-)\circ \trasp (p, -)
$ recalling that $\trasp (p, -): C(y,x)\  \rightarrow \ C (y', \trasp(p,x))$
and $\trap (q,-): C( y', \trasp(p,x))\ \rightarrow C(y', x') $.
\item
If $\Gamma\ \ourdef \ \Delta,  x: A$  and $B\ \ourdef C(y_1,\dots, y_n, x)$ for $\Delta\ \ourdef\  y_1:D_1, \dots,  y_n:D_n$, supposing $n\geq 1$ then
a  {\it dependent transport } it is simply
$$ \begin{array}{l}
\trasp(w, -): C(y_1, \dots,y_n, x)\ \rightarrow\ C( y'_1, \dots,y'_n,  x') \\
\qquad [y_1:D_1,\dots, y_n: D_n,\  x: A,\ 
y'_1:D_1,\dots, y'_n:D_n,\\
\qquad\qquad \qquad \ x':A,\ p_1: y_1=_{D_1}y'_1, \dots, p_n : y_n=_{D_n (y_1,\dots y_{n_1} }y'_n,  p: x=_A x'  ]\end{array}$$
\item
\end{enumerate}

\fi

\begin{lemma}\label{isoext}
If $\mu_A^{B}: A\rightarrow B\ [\Gamma]$ and $\nu_A^{B}: A\rightarrow B\ [\Gamma]$ are indexed isomorphisms, and for any $x:A$, $\mu_A^{B}(x) =_{B} \nu_A^{B}(x)$, then $\mu_A^{B} =_{A\rightarrow B} \nu_A^{B}$ holds.
\end{lemma}

\begin{proof}
The statement follows immediately from function extensionality.
\end{proof}

In the following we give a definition of canonical isomorphisms between
dependent h-sets. This definition is meant to  {\it generalize} the notion
of transport between dependent types on equal elements, by enlarging the notion
of equality to include that among arbitrary truncated h-propositions which are equivalent.

\begin{remark}
It must be stressed that \can\ isomorphisms do not coincide with all the isomorphisms, because they need to preserve their canonical elements. On the other hand, assuming that for any type in the universe $\mathcal{U}_0$ and $\mathcal{U}_1$ the associated identity map is a canonical isomorphism yields a contradiction in presence of the Univalence Axiom. We thank one of the anonymous referee for this last observation.
\end{remark}

To this purpose we first introduce an inductive universe of h-sets  (within  $\mathcal{U}_1$)  equipped with an inductive elimination, formally given as an inductive-recursive definition added to \hott,  but  it would be enough to define it in the meta-theory~\footnote{This approach is taken because $\suni$ lacks an inductive elimination which would be contradictory with the Univalence Axiom.}. This universe will be used to interpret sets of \emtt:

 \begin{definition}\label{setem}
 Let $\setem$ be the type inductively generated from the following inductive clauses: 
 \begin{itemize}

\item[-] 
 If  $A\ \ourdef \ \propo$,  or  $A\ \ourdef \ \bf 0 $,  or
$A\ \ourdef \ \bf 1 $,  or $A\ \ourdef \ \bf \mathbb{N} $ then $A: \setem\ [ \Gamma]$ for any context $\Gamma$.

\item[-]  $ \vert\vert  B \vert\vert : \setem\ [\Gamma]$  for  any type  $B:\suni\ [\Gamma]$.

\item[-]  $\ \Sigma_{x:B}\ C(x): \setem \ [\Gamma] $   for any $B : \setem\  [\Gamma]$  and  $C(x) : \setem\  [\Gamma, x:B]$.

\item[-] $\ \Pi_{x:B}\ C(x): \setem \ [\Gamma] $   for any $B : \setem\  [\Gamma]$  and  $C(x) : \setem\  [\Gamma, x:B]$.

\item[-] $B+C :\setem \ [\Gamma] $   for any $B: \setem\  [\Gamma]$  and  $C:\setem\  [\Gamma]$.

 \item[-] $\mathsf{List}(B): \setem \ [\Gamma] $   for any $B: \setem\  [\Gamma]$.

\item[-] $B/R:\setem\ [\Gamma]$ for any h-set $B:\funi\ [\Gamma]$  and  an equivalence relation $R: \propo\ [\Gamma, x: B,y:B] $  such that  $B: \setem\  [\Gamma]$ and
 $R(x,y): \setem\ [\Gamma, x: B,y:B] $.

\end{itemize}

 \end{definition}

Then, we are ready to define by recursion on $\setem$ the type $\mathsf{Ciso}(A,B)$ of canonical isomorphisms between h-sets A and B in $\setem$ as a subtype of $A \to B$. Formally, it is given again as an inductive-recursive definition, where each $\mathsf{Ciso}(A,B)$ is thought as a set of codes, together with a decoding function from $ \mathsf{Ciso}(A,B)$ to $ A \to B$.

 \begin{definition}\label{caniso}
The type of indexed canonical isomorphisms $\mu_{A_1}^{A_2}:A_1\rightarrow A_2 \  [\Gamma]$  is the type  inductively generated from the following inductive clauses: 

\begin{itemize}

\item[-]  If  $A\ \ourdef \ \propo$,  or  $A\ \ourdef \ \bf 0 $,  or
$A\ \ourdef \ \bf 1 $,  or $A\ \ourdef \ \bf \mathbb{N} $, then
the identity morphism $\mathsf{id}_A^A\ \ourdef \lambda x.x: A\ \rightarrow \  A \ [\Gamma]$ is a \can\ isomorphism, which is trivially an isomorphism whose inverse
${\mu_{A_1}^{A_2}}^{-1}$  is the identity.
    
\item[-] If $A_1\ \ourdef\ \vert\vert  B_1 \vert\vert : \guni$ and $A_2\ \ourdef\ \vert\vert  B_2 \vert\vert : \guni$,  then any isomorphism (with a chosen inverse)  $\mu_{\vert\vert B_1\vert\vert}^{\vert\vert B_2\vert\vert}: \vert\vert B_1\vert\vert \rightarrow \vert\vert B_2\vert\vert\ [\Gamma] $ is canonical and we denote the chosen inverse with  ${\mu_{A_1}^{A_2}}^{-1}$.

    \item[-] If $A_1\ \ourdef\ \Sigma_{x:B_1}\ C_1(x)\ [\Gamma]$ and $A_2\ \ourdef\ \Sigma_{x':B_2}\ C_2(x')\ [\Gamma]$
    and  $\mu_{B_1}^{B_2}: B_1 \rightarrow  B_2\ [\Gamma]$ and $\mu_{C_1(x)}^{C_2(\mu_{B_1}^{B_2}(x))} :  C_1(x)\ \rightarrow  \  C_2(\mu_{B_1}^{B_2}(x)) \ [\Gamma, x:B_1] $ 
    are \can\ isomorphisms, then any function
    $$\mu_{\Sigma_{x:B_1}\ {C_1}(x)}^{\Sigma_{x':B_2}\ C_2(x')  } : \Sigma_{x:B_1} \  C_1(x)\ \rightarrow \ \Sigma_{x':B_2}\ C_2 (x')\ [\Gamma]$$ such that
    $$\mu_{\Sigma_{x:B_1}\ {C_1}(x)}^{\Sigma_{x':B_2}\ C_2(x')  } (z)\ =\ ( \mu_{B_1}^{B_2}(\fpr(z))\, ,
    \mu_{C_1(\fpr(z))}^{C_2(\mu_{B_1}^{B_2} (\fpr(z) ))}(\spr(z) )\, )$$ for any $z: \Sigma_{x:B_1}\ C_1(x)$, is a \can\ isomorphism with inverse 
    
    \begin{equation*}
     {\mu_{A_1}^{A_2}}^{-1}\ \ourdef\  \lambda z. (\ {\mu_{B_1}^{B_2}}^{-1}(\fpr (z))\ , \ (
    { \mu_{C_1({\mu_{B_1}^{B_2}}^{-1}(\fpr(z)))}^{C_2(\mu_{B_1}^{B_2}({{\mu_{B_1}^{B_2}}^{-1} (\fpr(z)))}  }})^{-1} \circ \trasp({p_{\mu}}^{-1}, -) ( \spr(z)) \, )
    \end{equation*}
    where $p_{\mu}: \mu_{B_1}^{B_2}({\mu_{B_1}^{B_2}}^{-1}(\fpr(z)) =_{B_2} \fpr(z) $ and
    provided that $\mu_{B_1}^{B_2}$ and $\mu_{C_1(x)}^{C_2(\mu_{B_1}^{B_2}(x))}$ come equipped with inverses  ${\mu_{B_1}^{B_2}}^{-1}$ and
   $ {\mu_{C_1(x)}^{C_2(\mu_{B_1}^{B_2}(x))}}^{-1}$ respectively.
   
    \item[-] If $A_1\ \ourdef\ \Pi_{x:B_1}\ C_1(x)\ [\Gamma]$ and $A_2\ \ourdef\ \Pi_{x':B_2}\ C_2(x')\ [\Gamma]$
    and  $\mu_{B_1}^{B_2}: B_1 \rightarrow  B_2\ [\Gamma]$
    and  $\mu_{C_1(x)}^{C_2(\mu_{B_1}^{B_2}(x))} :  C_1(x)\ \rightarrow  \ C_2(\mu_{B_1}^{B_2}(x)) \ [\Gamma, x:B_1] $ are \can\ isomorphisms, then any function $$\mu_{\Pi_{x:B_1}\ C_1(x)}^{\Pi_{x': B_2}\ C_2(x')} : \Pi_{x:B_1}\ C_1(x)\rightarrow \Pi_{x':B_2}\ C_2(x')\ [\Gamma]$$  such that
    $$\mu_{\Pi_{x:B_1}\ C_1(x)}^{\Pi_{x': B_2}\ C_2(x')}(f)\ =\ \lambda x': B_2. (\ \trasp (p_\mu\ ,-)\circ  
    \mu_{C_1(\mu_{B_1}^{B_2^{-1}}(x'))}^{C_2(  \mu_{B_1}^{B_2} (\mu_{B_1}^{B_2^{-1}}(x') )\,)} \ ) (f (\mu_{B_1}^{B_2^{-1}}(x'))\ )$$  is a \can\ isomorphism, for any $f: \Pi_{x:B_1}\ C_1(x)$ and for any $p_\mu \ :\mu_{B_1}^{B_2}(\mu_{B_1}^{B_2^{-1}}(x') )=_{B_2}x'$
%    and we call $$\mathsf{bd}_\mu\ \ourdef  \trasp (p_\mu\ ,-)\circ  
%    \mu_{C_1(\mu_{B_1}^{B_2^{-1}}(x'))}^{C_2(  \mu_{B_1}^{B_2} (\mu_{B_1}^{B_2^{-1}}(x') )\,)} \ $$
where the body after the lambda is the arrow
$$\xymatrix@+4pt{
C_1(  \mu_{B_1}^{B_2^{-1}}(x') )  \ar[rrr]^{ \qquad \mu_{C_1(\mu_{B_1}^{B_2^{-1}}(x') )}^{C_2(\mu_{B_1}^{B_2}((\mu_{B_1}^{B_2^{-1}}(x') ) ))} \qquad} & &&C_2(  \mu_{B_1}^{B_2} (\mu_{B_1}^{B_2^{-1}}(x') )\,) \ar[rr]^{\qquad \trasp( p_\mu, -)} &&C_2(  x') 
  }$$
applied to the value $f (\mu_{B_1}^{B_2^{-1}}(x'))$. The associated inverse is given by

\begin{equation*}
  (\mu_{A_1}^{A_2})^{-1} =\ \lambda f'. \lambda x: B_1.  (\ ( \mu_{C_1 \ ( x)}^{C_2(  {\mu_{B_1}^{B_2}}(x))})^{-1} 
    (f' ({\mu_{B_1}^{B_2}}(x))\ ))
\end{equation*}
  
provided that  $\mu_{B_1}^{B_2}$ and $\mu_{C_1(x)}^{C_2(\mu_{B_1}^{B_2}(x))}$ come equipped with inverses  ${\mu_{B_1}^{B_2}}^{-1}$ and
   $ {\mu_{C_1(x)}^{C_2(\mu_{B_1}^{B_2}(x))}}^{-1}$ respectively.
    
    \item[-] If $A_1\ \ourdef\ B_1+C_1$ and $A_2\ \ourdef\ B_2 +C_2$ and $\mu_{B_1}^{B_2}: B_1\rightarrow B_2\ [\Gamma]$ and $\mu_{C_1}^{C_2}: C_1\rightarrow C_2\ [\Gamma]$ are \can\ isomorphisms, then any function $$ \mu_{B_1 + C_1}^{B_2 + C_2}: B_1 + C_1 \rightarrow B_2 + C_2\ [\Gamma]$$ such that
    $$\mu_{B_1 + C_1}^{B_2 + C_2}(z)\ =\ \mathsf{ind}_+ (z, z_0. \mathsf{inl}(\mu_{B_1}^{B_2}(z_0)), z_1.\mathsf{inr}(\mu_{C_1}^{C_2}(z_1)))$$ for any $z: B_1 + C_1$, is a \can\ isomorphism with inverse
    
    \begin{equation*}
    {\mu_{A_1}^{A_2}}^{-1}\ =\ \lambda z.\mathsf{ind}_+ (z, z_0. \mathsf{inl}({\mu_{B_1}^{B_2}}^{-1}(z_0)), z_1.\mathsf{inr}({\mu_{C_1}^{C_2}}^{-1}(z_1)))
    \end{equation*}
    
    provided that $\mu_{B_1}^{B_2}$ and $\mu_{C_1}^{C_2}$ come equipped with inverses ${\mu_{B_1}^{B_2}}^{-1}$ and ${\mu_{C_1}^{C_2}}^{-1}$.

    \item[-] If $A_1\ \ourdef\ \mathsf{List}(B_1)$ and $A_2\ \ourdef\ \mathsf{List}(B_2)$
    and  $\mu_{B_1}^{B_2}: B_1 \rightarrow B_2\ [\Gamma]$ is a \can\ isomorphism,  then any function $$\mu_{\mathsf{List}(B_1)}^{\mathsf{List}(B_2)}: \mathsf{List}(B_1)\rightarrow \mathsf{List}(B_2)\ [\Gamma]$$ such that
    $$\mu_{\mathsf{List}(B_1)}^{\mathsf{List}(B_2)}(z)\ =\ \mathsf{ind}_{\mathsf{List}}(z, \epsilon, (x,y,z).\mathsf{cons}(z, \mu_{B_1}^{B_2}(y)) $$ for any $z : \mathsf{List}(B_1)$, is a \can\ isomorphism with inverse
    
    \begin{equation*}
   {\mu_{A_1}^{A_2}}^{-1}\ =\ \lambda z.  \mathsf{ind}_{\mathsf{List}}(z, \epsilon, (x,y,z).\mathsf{cons}(z, {\mu_{B_1}^{B_2}}^{-1}(y))
\end{equation*}     
    
    provided that $\mu_{B_1}^{B_2}$ comes equipped with inverse ${\mu_{B_1}^{B_2}}^{-1}$.

    \item[-] If $A_1\ \ourdef\ B_1/R_1$ and  $A_2\ \ourdef\ B_2/R_2$, for $R_1, R_2$ equivalence relations,  and  $\mu_{B_1}^{B_2}: B_1\ \rightarrow \ B_2\ [\Gamma]$ is a \can\ isomorphism and $R_1(x,y)\leftrightarrow R_2( \mu_{B_1}^{B_2} (x)\, , \mu_{B_1}^{B_2} (y)) \ [\Gamma , x: B_1, y: B_1] $ holds,  then any function $$\mu_{B_1/R_1}^{B_2/R_2}\, :\, B_1/R_1\ \rightarrow\ B_2/R_2 \ [\Gamma]$$ such that
    $$\mu_{B_1/R_1}^{B_2/R_2}(z)\ =\ \mathsf{ind}_{Q}(z,  x.\mu_{B_1}^{B_2}(x))$$ for any $z: B_1/R_1$, is a \can\ isomorphism with inverse
    
    \begin{equation*}
    {\mu_{A_1}^{A_2}}^{-1}\ =\ \lambda z.  \mathsf{ind}_{Q}(z,  x.{\mu_{B_1}^{B_2}}^{-1}(x))
    \end{equation*}
    
    provided that $\mu_{B_1}^{B_2}$ comes equipped with an inverse ${\mu_{B_1}^{B_2}}^{-1}$.
    
\end{itemize}

\end{definition}

\begin{lemma}
Canonical isomorphisms are closed under substitution: if $ {\mu}_{A}^{B}: A \  \rightarrow B\ [\Gamma]$ is a \can\ isomorphism and 
$\Gamma\ \ourdef\ \Delta, x : E, y_1: C_1,\ldots, y_n: C_n$
then the result $$\begin{array}{l}{\mu}_{A}^{B}[e/x] [y'_i/y_i]_{i=1,\dots,n} :  \
 A[e/x] [y'_i/y_i]_{i=1,\dots,n}\  \rightarrow\ B[e/x] [y'_i/y_i]_{i=1,\dots,n}\\[5pt]
\qquad \qquad\qquad\qquad\qquad\qquad [\Delta, y'_1: C_1[e/x], \dots, y'_n: C_n[ e/x] [y'_i/y_i]_{i=1,\dots,n-1} ]\end{array}$$
of the substitution in $\mu_{A}^{B}$ of the variable $x$ with
$e: E\ [\Delta]$  is a \can\ isomorphism.
\end{lemma}

\begin{proof}
The proof is by structural induction over the definition of \can\ isomorphism.
\end{proof}

\begin{lemma}\label{caninv}
Any h-set $A \ [\Gamma]$  of \hott\  in $\setem$ has canonical transport operations.
\end{lemma}
\begin{proof}
By induction  on the formation of the type. Here, we just show that the transport operations
of the form $\trasp^1(p,-)$ are canonical for some type constructors
 since the canonicity of those of the form $\trasp^n(p,-)$ follows analogously for 
 all the types.
 
\begin{itemize}

\item[-] Non dependent ground types have just the identities as transport operations and these are canonical by definition \ref{caniso}.

\item[-] If $A\ \ourdef\ \vert\vert B\vert\vert$ and $\Gamma\ \ourdef\ \Delta, x:E$, then transport operations are \can\ by definition \ref{caniso}, since they are isomorphisms.

\item[-] If  $A\ \ourdef\ \Sigma_{y:B} \ C(y)$ and $\Gamma\ \ourdef\ \Delta, x:E$, then $$\trasp (p, z): A\rightarrow A[x'/x]\ [\Delta, x:E, x': E, p: x=_E x']$$ satisfies $$\trasp (p,z) = (\ \trasp(p, \fpr(z))\ ,\  \trasp (p, \spr(z)) \ )$$ which follows by $\mathrm{Id}$-elimination and is \can\ by definition \ref{caniso}, since by inductive hypothesis the  transport operations of $B$ and $C(y)$  are canonical.

\item[-] If $A\ \ourdef\ \Pi_{y:B} \ C(y)$ and $\Gamma \ \ourdef\ \Delta,x:E$
then $$\trasp (p,-) : A\ \rightarrow\  A[x'/x]\ [\Delta, x: E, x': E, p: x=_E x'  ]$$
%for $z': B[x'/x]$
%is defined as 
for any $f: \Pi_{y:B} \ C(y)$ 
satisfies
$$\trasp(p,f) =  \lambda z.   \ \trasp^C(\,p,\,  f(\, \trasp^B(p^{-1}, z)\, )\, ) )$$
where $p^{-1}$ is the reverse path in \cite{hottbook}.

This is canonical by definition \ref{caniso}, since the transport operations of $B$ and $C(y)$ along $p$ and $p^{-1}$ are all canonical by inductive hypothesis.

\item[-] If $A\ \ourdef\ B + C$ and $\Gamma\ \ourdef\ \Delta, x:E$, then $$\trasp(p, z):   A \rightarrow A[x'/x]\ [\Delta, x: E, x': E, p: x=_E x'  ]$$ 
satisfies $$\trasp(p, z) =\mathsf{ind}_+ (z, z_1. \mathsf{inl}(\trasp (p,z_1))\, ,
z_2. \mathsf{inr}(\trasp (p,z_2)))$$ which is canonical since the transport
operations of $B$ and $C$ are canonical by inductive hypothesis.

\item[-] If $A\ \ourdef\ B/R$ and $\Gamma\ \ourdef\ \Delta, x:E$, then $$\trasp(p, z):   A \rightarrow A[x'/x]\ [\Delta, x: E, x': E, p: x=_E x'  ]$$ satisfies 
$$\trasp (p, z) =\mathsf{ind}_{Q} (z, w. [\trasp (p, w)]\, )$$ which is \can\ by definition \ref{caniso} since the transport operations of $B$ are canonical by inductive hypothesis.

\item[-] If $A\ \ourdef\ \mathsf{List}(B)$, then it follows in a similar manner that it has canonical transport operations. 
\end{itemize}
\end{proof}

\begin{corollary}
For any transport operation its inverse
is a canonical isomorphism as well.
\end{corollary}
\begin{proof}
Note that $\trasp^{i} (p^{-1}, -)$
is the inverse of $\trasp^{i} (p, -)$
as shown in Example 2.4.9 \cite{hottbook}.
\end{proof}

Canonical isomorphisms are unique, are closed under composition and they have canonical inverses:

\begin{proposition}\label{isocanprop}
The following properties of canonical isomorphisms hold:
\begin{itemize}
    
    \item[-]{\bf identities are canonical} : For any h-set $A: \suni\ [\Gamma]$ in $\setem$,  the map $\mathsf{id}_{A}: A\rightarrow A\ [\Gamma]$ is a \can\ isomorphism;
    
    \item[-] {\bf uniqueness of  canonical isomorphisms}: For any  h-sets $A_1, A_2 : \suni\ [\Gamma]$   in $\setem$, if $\mu_{A_1}^{A_2}: A_1\rightarrow A_2\ [\Gamma]$ and $\nu_{A_1}^{A_2}: A_1\rightarrow A_2\ [\Gamma]$ are \can\ isomorphisms, then  $\mu_{A_1}^{A_2}(z) =_{A_2} {\nu}_{A_1}^{A_2}(z)\ [\Gamma, z: A_1]$;
    
    \item[-]{\bf closure under composition}: For any  h-sets $A_1, A_2 : \suni\ [\Gamma]$   in $\setem$,, if $\mu_{A_1}^{A_2}: A_1\rightarrow A_2\ [\Gamma]$ and $\mu_{A_2}^{A_3}: A_2\rightarrow A_3\ [\Gamma]$ are \can\ isomorphisms, then $\mu_{A_2}^{A_3}\circ\mu_{A_1}^{A_2}: A_1\rightarrow A_3\ [\Gamma]$ is a \can\ isomorphism.
    
    \item[-] {\bf  closure under canonical inverse}:  For any  h-sets $A_1, A_2 : \suni\ [\Gamma]$   in $\setem$,  each canonical isomorphism $$\mu_{A_1}^{A_2}:A_1\rightarrow A_2 \  [\Gamma]$$
is an isomorphism in the sense of definition \ref{iso} with a canonical inverse.
\end{itemize}
\end{proposition}

\begin{proof}
All the statements are proved simultaneously  by structural induction over the definition of \can\ isomorphisms. For each point we just show  some cases since the others follow analogously.

\begin{enumerate}
\item First point.
 
\noindent   
If $A\ \ourdef\ \vert\vert C\vert\vert$, then that $\mathsf{id}_{A}$ is a \can\ isomorphism trivially follows, since the identity map is an isomorphism. 

If $A\ \ourdef\ \Sigma_{x:B}\ C(x)$, then by induction hypothesis $\mathsf{id}_{B}$ and $\mathsf{id}_{C(x)}$ are \can\ isomorphisms, hence

$$\nu_{\Sigma_{x:B}C(x)}(z)\ =\ (\mathsf{id}_{B}(\fpr(z)), \mathsf{id}_{C(\fpr(z))}^{C(\mathsf{id}_{B}(\fpr(z))}(\spr(z))) \equiv\   (\fpr(z),\spr(z)).$$
But we know that $(\fpr(z),\spr(z)) = z$, 
hence  $\nu_{\Sigma_{x:B}C(x)}(z)\ = \mathsf{id}_{\Sigma_{x:B}C(x)}(z)\ $  which means that $\mathsf{id}_{\Sigma_{x:B}C(x)} $ is a canonical isomorphism since by hypothesis its transports are \can.

\noindent
If $A\ \ourdef\ B + C$ with canonical transport operations, then by induction hypothesis $\mathsf{id}_{B}$ and $\mathsf{id}_{C}$ are \can\ isomorphisms, therefore $$ \nu_{B+C}(z)\ \ourdef\ \mathsf{ind}_{+}( z, z_0. \mathsf{inl}(\mathsf{id}_B(z_0)), z_1.\mathsf{inr}(\mathsf{id}_C(z_1)))\equiv\ \mathsf{ind}_{+}( z, z_0. \mathsf{inl}(z_0), z_1.\mathsf{inr}(z_1)) $$ for any $z: B+C $, but $\mathsf{ind}_{+}( z, z_0. \mathsf{inl}(z_0), z_1.\mathsf{inr}(z_1)) = z$ and hence $\nu_{B+C}(z) = \mathsf{id}_{B+C}(z) $, which implies that the latter is a \can\ isomorphism.

 The other cases are similar. 
\vspace{1.0em}

\item Second point.

\noindent
For non-dependent ground types, the result is immediate since \can\ isomorphisms are the identities. 

\noindent
Suppose  $A_1\ \ourdef\ \vert\vert B_1\vert\vert$ and $A_2\ \ourdef\ \vert\vert B_2\vert\vert$. 
Then   $\mu_{\vert\vert B_1\vert\vert}^{\vert\vert B_2\vert\vert}(x)=_{\vert\vert B_2\vert\vert} \nu_{\vert\vert B_1\vert\vert}^{\vert\vert B_2\vert\vert}(x)$\ for any $x: \vert\vert B_1\vert\vert$\ , since $\vert\vert B_2\vert\vert$ is a h-proposition. 

\noindent
If $A_1\ \ourdef\ \Sigma_{x:B_1}\ C_1(x)$ and $A_2\ \ourdef\ \Sigma_{x':B_2}\ C_2(x')$ then both $\mu_{A_1}^{A_2}$  and $\nu_{A_1}^{A_2}$ are defined componentwise as in definition~\ref{caniso}. Let us assume that $\mu_{B_1}^{B_2}$ and $\mu_{C_1(x)}^{C_2(\mu_{B_1}^{B_2}(x))}$ are the components of the first and  $\nu_{B_1}^{B_2}$ and $\nu_{C_1(x)}^{C_2(\nu_{B_1}^{B_2}(x))}$ are those of the latter.

\noindent
Then, by inductive hypothesis $$\mu_{B_1}^{B_2}(x)=_{B_2}\nu_{B_1}^{B_2}(x)\ [\Gamma, x:B_1]$$ and

$$\trasp (p,\  \mu_{C_1(x)}^{C_2(\mu_{B_1}^{B_2}(x))}(y)) =_{C_2(\nu_{B_1}^{B_2}(x))} \nu_{C_1(x)}^{C_2(\nu_{B_1}^{B_2}(x))}(y)\ [\Gamma, x:B_1, y:C_1(x), p: \mu_{B_1}^{B_2}(x)=_{B_2}\nu_{B_1}^{B_2}(x)\ ]$$
therefore  $$\mu_{\Sigma_{x:B_1}\ C_1(x)}^{\Sigma_{x':B_2}\ C_2(x')}(z) =_{\Sigma_{x':B_2}C_2(x')} \nu_{\Sigma_{x:B_1}\ C_1(x)}^{\Sigma_{x':B_2}\ C_2(x')}(z)\ [\Gamma, z: \Sigma_{x:B_1}\ C_1(x)]$$  and, by lemma \ref{isoext}, we conclude $\mu_{\Sigma_{x:B_1}\ C_1(x)}^{\Sigma_{x':B_2}\ C_2(x')} = \nu_{\Sigma_{x:B_1}\ C_1(x)}^{\Sigma_{x':B_2}\ C_2(x')}$.

\noindent
If $A_1\ \ourdef\ \Pi_{x:B_1}\ C_1(x)$ and $A_2\ \ourdef\ \Pi_{x':B_2}\ C_2(x')$, let us consider any two canonical isomorphisms which we denote as $$\mu_{A_1}^{A_2}=\lambda f. \lambda z.   (\ \trasp (p_\mu\ ,-)\circ  
    \mu_{C_1(\mu_{B_1}^{B_2^{-1}}(x'))}^{C_2(  \mu_{B_1}^{B_2} (\mu_{B_1}^{B_2^{-1}}(x') )\,)} \ ) (f (\mu_{B_1}^{B_2^{-1}}(z))\ )$$ and
$$\nu_{A_1}^{A_2} =\lambda f. \lambda z.   (\ \trasp (p_\nu\ ,-)\circ  
    \nu_{C_1(\nu_{B_1}^{B_2^{-1}}(x'))}^{C_2(  \nu_{B_1}^{B_2} (\nu_{B_1}^{B_2^{-1}}(x') )\,)} \ ) (f (\nu_{B_1}^{B_2^{-1}}(z))\ )
 $$
where $p_\mu \ :\mu_{B_1}^{B_2}(\mu_{B_1}^{B_2^{-1}}(z) )=_{B_2}z$ and $p_\nu \ :\nu_{B_1}^{B_2}(\nu_{B_1}^{B_2^{-1}}(z) )=_{B_2}z$.
Now, since
for any $f:A_1$ and any $x: B_1$
by inductive hypothesis there exists a proof $q$ of type 
$$\mu^{B_2}_{B_1}(x)=\nu^{B_2}_{B_1}(x) $$
and the same holds for their
inverses, which are canonical  by inductive hypothesis.
Therefore there exists  a proof $q':
{\mu^{B_2}_{B_1}}^{-1}(z)={\nu_{B_1}^{B_2}}^{-1}(z)
$ for $z: B_2$ 
and by lemma \ref{apd-tr} we get  a proof of the equality
%$\mathsf{apd_f} (q'):  \trasp(q', f({\mu^{B_2}_{B_1}}^{-1}(z))\ ) )=f({\nu_{B_1}^{B_2}}^{-1}(z) )$ and hence

%{\tt direi di inserirlo nei preliminari}

$$\trasp(q',-) ( f({\mu^{B_2}_{B_1}}^{-1}(z))\ )  = \trasp(q', f({\mu^{B_2}_{B_1}}^{-1}(z))\ ) )=f({\nu_{B_1}^{B_2}}^{-1}(z) ) $$
Moreover, we have also a proof
$$q'': \mu_{B_1}^{B_2}(\mu_{B_1}^{B_2^{-1}}(z)) =  \nu_{B_1}^{B_2}({\nu_{B_1}^{B_2}}^{-1}(z))$$
being each member equal to $z:B_2$.

Furthermore,  by uniqueness of canonical morphisms from
$C_1({\mu^{B_2}_{B_1}}^{-1}(z))$ to $C_2 ( \nu^{B_2}_{B_1} ({\nu^{B_2}_{B_1}}^{-1}(z))   )$ which follows
by inductive hypothesis we have a proof of the following equality
$$\nu_{C_1({\nu^{B_2}_{B_1}}^{-1}(z))}^{C_2( \nu^{B_2}_{B_1} ({\nu^{B_2}_{B_1}}^{-1}(z))   ) }
\circ \trasp (q', -)
=   \trasp (q'', -)\circ \mu_{C_1({\mu^{B_2}_{B_1}}^{-1}(z))}^{C_2( \mu^{B_2}_{B_1} ({\mu^{B_2}_{B_1}}^{-1}(z))   ) }  $$

\begin{center}
$$\xymatrix{
    C_1({\mu_{B_1}^{B_2}}^{-1}(z)) \ar[rr]^{\mu_{C_1(--)}^{C_2(\mu_{B_1}^{B_2}(--))}}       \ar[dd]_{{\trasp (q',-)}}                          &  &
     {C_2 (\mu_{B_1}^{B_2}({\mu_{B_1}^{B_2}}^{_1}(z)) ) }   \ar[dd]^{{\trasp (q'',-)}}                           \\
                                                                                                                    &  &                                                                               \\
{C_1({\nu_{B_1}^{B_2}}^{-1}(z))}\ar[rr]_{\nu_{C_1(--)}^{C_2(\nu_{B_1}^{B_2}(--))}} &  & C_2 (\nu_{B_1}^{B_2}({\nu_{B_1}^{B_2}}^{_1}(z)))}
$$
\end{center}
and hence

$$\trasp(p_\nu,-) \circ (\ \nu_{C_1({\nu^{B_2}_{B_1}}^{-1}(z))}^{C_2( \nu^{B_2}_{B_1} ({\nu^{B_2}_{B_1}}^{-1}(z))   ) }
\circ \trasp (q', -)\ )
=  \trasp (p_\nu, -)  \circ ( \trasp (q'', -)\circ \mu_{C_1({\mu^{B_2}_{B_1}}^{-1}(z))}^{C_2( \mu^{B_2}_{B_1} ({\mu^{B_2}_{B_1}}^{-1}(z))   ) } \ ) $$

Moreover, knowing that  transports commute because they are uniquely determined up to propositional equality we get 
$$\trasp (p_\nu, -)  \circ  \trasp (q'', -)= \trasp(p_\mu,-)$$

\begin{center}
  $$\xymatrix{
  C_2 (\mu_{B_1}^{B_2}({\mu_{B_1}^{B_2}}^{-1}(z))) \ar[rrd]^{\trasp(p_\mu,-)}      \ar[dd]_{\trasp(q'',-)}                      &  &        \\
                                                                                                   &  & C_2(z) \\
C_2 (\nu_{B_1}^{B_2}({\nu_{B_1}^{B_2}}^{-1}(z)))\ar[rru]_{\trasp(p_\nu,-)} &  &       
}$$
\end{center}
and we conclude
$$
\trasp(p_\mu,-) \circ \ \mu_{C_1({\mu^{B_2}_{B_1}}^{-1}(z))}^{C_2( \mu^{B_2}_{B_1} ({\mu^{B_2}_{B_1}}^{-1}(z))   ) }
=  \trasp (p_\nu, -) \circ  (\  \nu_{C_1({\nu^{B_2}_{B_1}}^{-1}(z))}^{C_2( \nu^{B_2}_{B_1} ({\nu^{B_2}_{B_1}}^{-1}(z))   ) } \circ \trasp (q', -)\ ) $$
which applied to $f (\mu_{B_1}^{B_2^{-1}}(z))  $ and recalling that $ \trasp (q', -) (f (\mu_{B_1}^{B_2^{-1}}(z)))= f({\nu_{B_1}^{B_2}}^{-1}(z) )$  immediately gives

$$\begin{array}{rl} \mu_{A_1}^{A_2} (f,z)= & (\ \trasp(p_\mu,-) \circ \ \mu_{C_1({\mu^{B_2}_{B_1}}^{-1}(z))}^{C_2( \mu^{B_2}_{B_1} ({\mu^{B_2}_{B_1}}^{-1}(z))   ) } \ )
(f (\mu_{B_1}^{B_2^{-1}}(z)) )\\[10pt]
 = & \trasp (p_\nu, -) \circ (\  \nu_{C_1({\nu^{B_2}_{B_1}}^{-1}(z))}^{C_2( \nu^{B_2}_{B_1} ({\nu^{B_2}_{B_1}}^{-1}(z))   ) } \circ
 \trasp (q', -) \ ) (f (\mu_{B_1}^{B_2^{-1}}(z)) )\\[10pt]
 =&  \trasp (p_\nu, -) \circ   \nu_{C_1({\nu^{B_2}_{B_1}}^{-1}(z))}^{C_2( \nu^{B_2}_{B_1} ({\nu^{B_2}_{B_1}}^{-1}(z))   ) }
 f({\nu_{B_1}^{B_2}}^{-1}(z) )\\[10pt]
 =& \nu_{A_1}^{A_2} (f,z)
 \end{array}$$

and hence
 $$\mu_{A_1}^{A_2}=\nu_{A_1}^{A_2}$$

\noindent
If $A_1\ \ourdef\ B_1 + C_1 $ and $A_2\ \ourdef\ B_2 + C_2$ then both $\mu_{A_1}^{A_2}$ and $\nu_{A_1}^{A_2}$ are defined as in definition \ref{caniso}: in particular, $\mu_{B_1}^{B_2}$ and $\mu_{C_1}^{C_2}$ are \can\ isomorphisms as well as $\nu_{B_1}^{B_2}$ and $\nu_{C_1}^{C_2}$. 

Then,  by inductive hypothesis $$ \mu_{B_1}^{B_2}(x) =_{B_2} \nu_{B_1}^{B_2}(x)\ [\Gamma, x:B_1]$$ and $$ \mu_{C_1}^{C_2}(y) =_{C_2} \nu_{C_1}^{C_2}(y)\ [\Gamma, y:C_1]$$ therefore it trivially follows that  $$ \mu_{B_1 + B_2}^{C_1 + C_2}(z) =_{C_1 + C_2} \nu_{B_1 + B_2}^{C_1 + C_2}(z)\ [\Gamma, z: B_1 + B_2]$$

\noindent
and by lemma \ref{isoext} $\mu_{B_1 + B_2}^{C_1 + C_2} = \nu_{B_1 + B_2}^{C_1 + C_2}$. 

\noindent
If $A_1\ \ourdef\ B_1/R_1 $ and $A_2\ \ourdef\ B_2/R_2$ then $\mu_{A_1}^{A_2}$ and $\nu_{A_1}^{A_2}$ are defined as in definition \ref{caniso}: hence we can assume that $\mu_{B_1}^{B_2}$ and $\nu_{B_1}^{B_2}$ are \can\ isomorphisms and that the following propositions  $R_1 (x,y) \leftrightarrow R_2 (\mu_{B_1}^{B_2}(x),\mu_{B_1}^{B_2}(y)) $ and $R_1 (x,y) \leftrightarrow R_2 (\nu_{B_1}^{B_2}(x), \nu_{B_1}^{B_2}(y))$ hold. 

\noindent
Then,  by inductive hypothesis $$\mu_{B_1}^{B_2}(x) =_{B_2}  \nu_{B_1}^{B_2}(x)\ [\Gamma, x:B_1]$$ and  hence $$ R_2 (\mu_{B_1}^{B_2}(x), \mu_{B_1}^{B_2}(y)) \leftrightarrow R_2 (\nu_{B_1}^{B_2}(x), \nu_{B_1}^{B_2}(y))\ [\Gamma, x:B_1, y:B_1].$$ Therefore it trivially follows that  $$ \mu_{B_1/R_1}^{B_2/R_2}(z) =_{B_2/R_2} \nu_{B_1/R_1}^{B_2/R_2}(z)\ [\Gamma, z: B_1/R_1]$$ and by lemma \ref{isoext} $\mu_{B_1/R_1}^{B_2/R_2} = \nu_{B_1/R_1}^{B_2/R_2}$.
\vspace{1.0em}

\item Third point.

For non-dependent ground types, the composition is the identity and hence is canonical by definition. 

For truncated types, since isomorphisms are closed under composition and any isomorphism between truncated types is canonical by definition, then it immediately follows that the composition of canonical isomorphisms between truncated types is canonical too.

\noindent
If $A_1\ \ourdef\  \Sigma_{x:{B_1}}{C_1 (x)}$ and $A_2\ \ourdef\ \Sigma_{x':B_2}C_2 (x') $ and $A_3\ \ourdef\ \Sigma_{x'':B_3}C_3 (x'')$, then by definition of canonical isomorphisms
$$\mu_{A_1}^{A_2}\ =\ \lambda z.(\mu_{B_1}^{B_2}(\fpr(z)), \mu_{C_1(\fpr(z)) }^{C_2(\mu_{B_1}^{B_2}(\fpr(z))}( \spr(z)))\ )$$ 
and $$\mu_{A_2}^{A_3}\ =\ \lambda z.(\mu_{B_2}^{B_3}(\fpr(z)), \mu_{C_2 (\fpr(z))}^{C_3 (\mu_{B_2}^{B_3}(\fpr(z)) ) }(\spr(z))\ )$$
Now the composition of 
$\mu_{A_2}^{A_3}\circ \mu_{A_1}^{A_2}$ applied to $z: \Sigma_{x:B_1} C_1 (x) $
amounts to 

\iffalse
for which we get
$\mu_{A_2}^{A_3}(\mu_{A_1}^{A_2}(z))$ and 
let us call $\nu_{B_1}^{B_3}\ \ourdef \mu_{B_2}^{B_3}\circ \mu_{B_1}^{B_2}$ and $\nu_{C_1 (x)}^{C_3(  ) }\ \ourdef\ \mu_{C_2}^{C_3(\mu_{B_2}^{B_3})}\circ \mu_{C_1}^{C_2} $. Therefore by definition of canonical isomorphisms
of indexed sums we get 
\fi

$$\begin{array}{rcl} 
\mu_{A_2}^{A_3}\circ \mu_{A_1}^{A_2} (z)\ & = & \mu_{\Sigma_{x':B_2}C_2(x')}^{\Sigma_{x'':B_3}C_3 (x'')} \circ \mu_{\Sigma_{x:B_1}C_1(x)}^{\Sigma_{x':B_2}C_2(x')}\, (z)\\[5pt]
 & = & (\mu_{B_2}^{B_3}( \mu_{B_1}^{B_2}(\fpr(z))\, )\, ,\  \mu_{C_2(\mu_{B_1}^{B_2}(\fpr(z)) }^{C_3(\mu_{B_2}^{B_3}(\mu_{B_1}^{B_2}(\fpr(z))     )} (\, \mu_{C_1(\fpr(z))}^{C_2( \mu_{B_1}^{B_2}(\fpr(z))) }(\spr(z)) \,  )\ )
 \end{array}$$
 which is a canonical isomorphism by definition \ref{caniso}
 since $\mu_{B_2}^{B_3}\circ \mu_{B_1}^{B_2} $ and
 $\mu_{C_2}^{C_3}\circ \mu_{C_1}^{C_2}$ are \can\ isomorphisms by inductive hypothesis.

\noindent
If $A_1\ \ourdef\  \Pi_{x:{B_1}}{C_1}(x)$ and $A_2\ \ourdef\ \Pi_{x':B_2}C_2(x') $ and $A_3\ \ourdef\ \Pi_{x'':B_3}C_3(x'')$, then, by definition of canonical isomorphisms
$$\mu_{A_1}^{A_2}\ =\ \lambda f. \lambda x': B_2. (\ \trasp (p_{\mu_{A_1}^{A_2}},-) \circ  
    \mu_{C_1(\mu_{B_1}^{B_2^{-1}}(x'))}^{C_2(  \mu_{B_1}^{B_2} (\mu_{B_1}^{B_2^{-1}}(x') )\,)} \ ) (f (\mu_{B_1}^{B_2^{-1}}(x'))\ )$$ 
    for any $p_{\mu_{A_1}^{A_2}} \ :\mu_{B_1}^{B_2}(\mu_{B_1}^{B_2^{-1}}(x') )=_{B_2}x'$ and
$$\mu_{A_2}^{A_3}\ =\ \lambda f. \lambda x'': B_3. (\trasp (p_{\mu_{A_2}^{A_3}}, -) \circ 
    \mu_{C_2({\mu_{B_2}^{B_3}}^{-1}(x''))}^{C_3(  \mu_{B_2}^{B_3} (\mu_{B_2}^{B_3^{-1}}(x'') )\,)} )(f (({\mu_{B_2}^{B_3}}^{-1}(x'')))\ )     $$
 for any $p_{\mu_{A_2}^{A_3}} \ : \mu_{B_2}^{B_3}({\mu_{B_2}^{B_3}}^{-1}(x''))=x''$.
 
Hence,   for any $x'': B_3$ and $f:  \Pi_{x'':B_3}C_3(x'')$ their composition  becomes
$$\begin{array}{l} 
\mu_{A_2}^{A_3}\circ \mu_{A_1}^{A_2} (f,x'') =  \ \mu_{\Pi_{x':B_2}C_2(x')}^{\Pi_{x'':B_3}C_3 (x'')} \circ \mu_{\Pi_{x:B_1}C_1(x)}^{\Pi_{x':B_2}C_2(x')}\, (f,x'')\\[10pt]
=  
(\trasp (p_{\mu_{A_2}^{A_3}}, -) \circ 
    \mu_{C_2(--)}^{C_3(  \mu_{B_2}^{B_3} (--) \,)} ) \circ 
( \ \trasp (p'_{\mu_{A_1}^{A_2}},-) \circ  
    \mu_{C_1(--)}^{C_2(  \mu_{B_1}^{B_2} (--) )} (   f ( { \mu_{B_1}^{B_2}}^{-1} ({\mu_{B_2}^{B_3}}^{-1}(x'') ) )
 )\\[10pt]
 =
 (\trasp(p_{\mu_{A_2}^{A_3}}, -) \circ   \trasp (p''_{\mu_{A_1}^{A_2}},-) ) \circ (\    \mu_{C_2 (--)}^{C_3(  \mu_{B_2}^{B_3} (--) \,)} )  \circ
    \mu_{C_1(--)}^{C_2(  \mu_{B_1}^{B_2} (--) )} (   f ( { \mu_{B_1}^{B_2}}^{-1} ({\mu_{B_2}^{B_3}}^{-1}(x'') ) ))
   \end{array}$$
   where $p'_{{\mu_{A_1}^{A_2}} }\ \ourdef\  p_{{\mu_{A_1}^{A_2}}} [  {\mu_{B_2}^{B_3}}^{-1}(x'')/x']$ 
and $p''_{\mu_{A_1}^{A_2}}\ \ourdef\  p_{{\mu_{A_1}^{A_2}}} [  {  \mu_{B_1}^{B_2}}^{-1} \circ{\mu_{B_2}^{B_3}}^{-1}  (x'')/x']$ .
In particular, the last equality
 follows  by uniqueness of canonical isomorphisms
from $C_2(\mu_{B_1}^{B_2} ({\mu_{B_1}^{B_2}}^{-1}({\mu_{B_2}^{B_3}}^{-1}(z)))) $ to $ C_3(\mu_{B_2}^{B_3}({\mu_{B_2}^{B_3}}^{-1}(z)))   $
from this other equality
\begin{center}
    $$   
    \xymatrix{
C_2(\mu_{B_1}^{B_2} ({\mu_{B_1}^{B_2}}^{-1}({\mu_{B_2}^{B_3}}^{-1}(x'')))) \ar[dd]_{\mu_{C_2(--)}^{C_3(\mu_{B_2}^{B_3}(--))}} \ar [rrr]^{\trasp(p'_{\mu_{A_1}^{A_2}},-)} &  &  &  C_2({\mu_{B_2}^{B_3}}^{-1}(x''))  \ar[dd]^{\mu_{C_2(--)}^{C_3(\mu_{B_2}^{B_3}(--))}}  \\
                                                                                                                                                                &  &  &                                                                                                        \\
C_3(\mu_{B_2}^{B_3}(\mu_{B_1}^{B_2} ({\mu_{B_1}^{B_2}}^{-1}({\mu_{B_2}^{B_3}}^{-1}(x'')) ))) \ar[rrr]_{\qquad \trasp(p''_{\mu_{A_1}^{A_2}},-)}                                        &  &  &         C_3(\mu_{B_2}^{B_3}({\mu_{B_2}^{B_3}}^{-1}(x'')) )               \\
                                                                                                                                                                  }$$       
\end{center}
Hence, $\mu_{A_2}^{A_3}\circ \mu_{A_1}^{A_2}$ is  a canonical isomorphism because consists of compositions of canonical isomorphisms by inductive hypothesis beside the fact that transport operations compose.

% where we used the abbreviation $b\ \ourdef \ {\mu_{B_1}^{B_2}}^{-1} ({\mu_{B_2}^{B_3}}^{-1}(x'') )$  and $\beta(x'')\ \ourdef\ {\mu_{B_2}^{B_3}} \circ\mu_{B_1}^{B_2}\circ {\mu_{B_1}^{B_2}}^{-1} \circ {\mu_{B_2}^{B_3}}^{-1} (x'') $
% and $q'_1$ is a proof that
% $\mu_{B_1}^{B_2} {\mu_{B_1}^{B_2}}^{-1} ({\mu_{B_2}^{B_3}}^{-1}(x'')) ={\mu_{B_2}^{B_3}}^{-1} (x'') $
% and $q_{3}$ is a proof that $ \mu_{B_2}^{B_3}\circ \mu_{B_1}^{B_2}({\mu_{B_1}^{B_2}}^{-1} ({\mu_{B_2}^{B_3}}^{-1}(x'')))= x''$.
% Note that the equality $=^{(\ast\ast )}$ holds because
% the body of the lambda function of each member is the result of composing canonical isomorphisms toward the same type on the same input
% $$f ({\mu_{B_1}^{B_2}}^{-1} ({\mu_{B_2}^{B_3}}^{-1}(x'')) ))$$
% and by uniqueness of canonical
% value we conclude the equality.
 
% Moreover, $ {\mu_{B_1}^{B_2}}^{-1}\circ {\mu_{B_2}^{B_3}}^{-1}$ is canonical because by inductive hypothesis they 
% are canonical, too. 
%

\noindent
If $A_1\ \ourdef\ B_1 + C_1$ and $A_2\ \ourdef\ B_2 + C_2$ and $A_3\ \ourdef\ B_3 + C_3$, then by definition of canonical isomorphisms $$\mu_{A_1}^{A_2}\ =\ \lambda z. \mathsf{ind}_{+} (z, z_0.\mathsf{inl}(\mu_{B_1}^{B_2}(z_0)), z_1.\mathsf{inr}(\mu_{C_1}^{C_2}(z_1)))$$ and $$ \mu_{A_2}^{A_3}\ =\ \lambda z. \mathsf{ind}_{+} (z, z_0.\mathsf{inl}(\mu_{B_2}^{B_3}(z_0)), z_1.\mathsf{inr}(\mu_{C_2}^{C_3}(z_1)))$$ 
Let us consider the composition $\mu_{A_2}^{A_3}\circ \mu_{A_1}^{A_2}$ applied to $z: B_1 + C_1$, for which we get $\mu_{A_2}^{A_3}(\mu_{A_1}^{A_2}(z))$, then $$ \mu_{B_1 + C_1}^{B_3 + C_3}(z)\ =\ \mathsf{ind}_{+}(z, z_0.\mathsf{inl}(\mu_{B_2}^{B_3}(\mu_{B_1}^{B_2}(z_0)) ), z_1.\mathsf{inr}(\mu_{C_2}^{C_3}(\mu_{C_1}^{C_2}(z_1)) )$$ which amounts to $\mu_{A_2}^{A_3}(\mu_{A_1}^{A_2}(z))$ and is a \can\ isomorphism by definition \ref{caniso}.

\noindent
If $A_1\ \ourdef\ B/R_1$ and $A_2\ \ourdef\ B_2/R_2$ and $A_3\ \ourdef\ B_3/R_3$, then by inductive hypothesis $$\mu_{A_1}^{A_2}\ =\ \lambda z. \mathsf{ind}_Q (z, x. \mu_{B_1}^{B_2}(x))$$ and  $$\mu_{A_2}^{A_3}\ =\ \lambda z. \mathsf{ind}_Q (z, x. \mu_{B_2}^{B_3}(x))$$ are canonical isomorphisms. Then,  let us consider the composition $\mu_{A_2}^{A_3}\circ \mu_{A_1}^{A_2}$ applied to $z: B_1/R_1$, so that we get $\mu_{A_2}^{A_3}(\mu_{A_1}^{A_2}(z))$, then $$ \mu_{B_1/R_1}^{B_3/R_3}(z)\ \ourdef\ \mathsf{ind}_{Q} (z, x. \mu_{B_2}^{B_3}(\mu_{B_1}^{B_2}(x)))$$ which amounts to $\mu_{A_2}^{A_3}(\mu_{A_1}^{A_2}(z))$ and is a \can\ isomorphism by definition \ref{caniso}.

\item Fourth point.

\noindent
For non-dependent ground types
the inverse is the identity which is canonical by definition.

Canonical isomorphisms between truncated types have canonical inverse
 by definition \ref{caniso}.

 If $A_1\ \ourdef\ \Sigma_{x:B_1}\ C_1(x)\ [\Gamma]$ and $A_2\ \ourdef\ \Sigma_{x':B_2}\ C_2(x')\ [\Gamma]$ and  $\mu_{B_1}^{B_2}: B_1 \rightarrow  B_2\ [\Gamma]$ and $\mu_{C_1(x)}^{C_2(\mu_{B_1}^{B_2}(x))} :  C_1(x)\ \rightarrow  \  C_2(\mu_{B_1}^{B_2}(x)) \ [\Gamma, x:B_1] $ 
 are \can\ isomorphisms, then the inverse of $\mu_{A_1}^{A_2}$ given as in definition \ref{caniso},
 %  the chosen inverse of ${\mu_{A_1}^{A_2}}$  defined for 
%$p_{\mu} $ proof of $ {\mu_{B_1}^{B_2}}({\mu_{B_1}^{B_2}}^{-1}(\fpr(z)))=\fpr(z)$ as
 $${\mu_{A_1}^{A_2}}^{-1}\ \ourdef\  \lambda z. (\ {\mu_{B_1}^{B_2}}^{-1}(\fpr (z))\ , \ (
    { \mu_{C_1({\mu_{B_1}^{B_2}}^{-1}(\fpr(z)))}^{C_2(\mu_{B_1}^{B_2}({{\mu_{B_1}^{B_2}}^{-1} (\fpr(z)))}  }})^{-1} \circ \trasp({p_{\mu}}^{-1}, -) ( \spr(z)) \, ) $$

%\begin{center}
%    $$\xymatrix@+4pt{
%C_1(  \mu_{B_1}^{B_2^{-1}}(x) )  \ar[rrr]^{ \qquad \mu_{C_1({\mu_{B_1}^{B_2}}^{-1}(x) )}^{C_2(\mu_{B_1}^{B_2}((\mu_{B_1}^{B_2^{-1}}(x') ) ))} \qquad} & &&C_2(  \mu_{B_1}^{B_2} (\mu_{B_1}^{B_2^{-1}}(x') )\,) \ar[rr]^{\qquad \trasp( q, -)} &&C_2(  x) 
%  }$$
%\end{center}
 is canonical by construction: it is composed of inverses of canonical isomorphisms, which are canonical by inductive hypothesis, and  transports, which are canonical by lemma~\ref{caninv}.
It amounts to be an inverse since  
%$$\trasp ( p_\mu,-)\circ  ({ \mu_{C_1({\mu_{B_1}^{B_2}}^{-1}(x'))}^{C_2(\mu_{B_1}^{B_2}({{\mu_{B_1}^{B_2}}^{-1} (x')} ) }})^{-1} \circ \trasp(p^{-1}_{\mu}, -)
%\circ
% ( \mu_{C_1(\fpr(z) )}^{C_2(\mu_{B_1}^{B_2}(\fpr(z) )) })= \mathsf{id}_{C_1( \fpr(z))}$$
% for  $p_{\mu} $ proof of $ \mu_{B_1}^{B_2}(\mu_{B_1}^{B_2^{-1}}(\fpr(z)))=\fpr(z)$. This in turn follows since
 the following equality holds
 by uniqueness of canonical isomorphisms 
 
  \begin{center}
    $$   
    \xymatrix{
    C_2(\fpr(z)) \ar[rrr]^{ \trasp({p_\mu}^{-1} , -)  }& && C_2({\mu_{B_1}^{B_2}} ({\mu_{B_1}^{B_2}}^{-1} (\fpr(z)))) \ar[dd]^{(\mu_{C_1(--)}^{C_2(--)})^{-1}}\\
    &&\\
 C_2(   \mu_{B_1}^{B_2}( {\mu_{B_1}^{B_2} }^{-1}   (\fpr(z) )))\ar[uu]^{ \trasp( p_\mu, -) }&&& C_1( {\mu_{B_1}^{B_2}}^{-1} (\fpr(z) ))\ar[lll]^{  { \mu_{C_1(--)}^{C_2( --)}} }}
 %\ar[lll]^{  ({ \mu_{C_1({\mu_{B_1}^{B_2}}^{-1}(--))}^{C_2( (\mu_{B_1}^{B_2} \circ ( {{\mu_{B_1}^{B_2}}^{-1}  \circ \mu_{B_1}^{B_2}))(\fpr(z) ))}})^{-1} }}}
 $$      
\end{center}

If $A_1\ \ourdef\ \Pi_{x:B_1}\ C_1(x)\ [\Gamma]$ and $A_2\ \ourdef\ \Pi_{x':B_2}\ C_2(x')\ [\Gamma]$
    and  $\mu_{B_1}^{B_2}: B_1 \rightarrow  B_2\ [\Gamma]$
%     (with  inverse
%    $(\mu_{B_1}^{B_2})^{-1}: B_2 \rightarrow  B_1\ [\Gamma]$)
    and  $\mu_{C_1(x)}^{C_2(\mu_{B_1}^{B_2}(x))} :  C_1(x)\ \rightarrow  \ C_2(\mu_{B_1}^{B_2}(x)) \ [\Gamma, x:B_1] $  are \can\ isomorphisms, 
    then the inverse of $ \mu_{A_1}^{A_2}$ given as in definition \ref{caniso}
    % for $f:  \Pi_{x:B_2}\ C_2(x)$ 
    $$(\mu_{A_1}^{A_2})^{-1} =\ \lambda f'. \lambda x: B_1.  (\ ( \mu_{C_1 \ ( x)}^{C_2(  {\mu_{B_1}^{B_2}}(x))})^{-1} 
    (f' ({\mu_{B_1}^{B_2}}(x))\ )) $$

%        $$(\mu_{A_1}^{A_2})^{-1} =\ \lambda f'. \lambda x: B_1. (\  \trasp(p_{\mu_{B_1}^{B_2}}\ , -) \circ  
%  (\ \mu_{C_1 \ ( \mu_{B_1}^{B_2^{-1}}(x))}^{C_2(  ({\mu_{B_1}^{B_2}})^{-1}\circ  (\  \mu_{B_1}^{B_2} ) (x)}) ^{-1} \ )
%    (f' ({\mu_{B_1}^{B_2}}(x))\ ) $$
%
%for $p_{\mu_{B_1}^{B_2}}$ proof of $(\ (\mu_{B_1}^{B_2})^{-1} \circ \mu_{B_1}^{B_2}\ )(x)=x$.
    
     is a  canonical  since 
    we can show that:     for   $q_{\mu} $ proof of $(\mu_{B_1}^{B_2})^{-1}(\mu_{B_1}^{B_2}(x))=x$ and  for any  $f': \Pi_{x':B_2}\ C_2(x')$ and $x: B_1$ 
     $$(\mu_{A_1}^{A_2})^{-1} (f')( x)= (\  \trasp(q_{\mu}\ , -) \circ  (\   (\ \mu_{C_1 \ ( \mu_{B_1}^{B_2^{-1}}(-))}^{C_2(  ({\mu_{B_1}^{B_2}}) \circ  (\ { \mu_{B_1}^{B_2}}^{-1} ) (-) )}) ^{-1} \ ) \circ  \trasp({q_\mu}^{-1}, -) \ )
   (f' ({\mu_{B_1}^{B_2}}(x))\ ) $$
   where the right member is the application of a composition of  isomorphisms which are canoni\-cal  by inductive hypothesis,
       because by uniqueness of canonical isomorphisms 
        $$   \trasp(q_{\mu}\ , -) \circ (\  (\mu_{C_1 \ ( \mu_{B_1}^{B_2^{-1}}(-))}^{C_2(  ({\mu_{B_1}^{B_2}}) \circ  (\ { \mu_{B_1}^{B_2}}^{-1} ) (-))}) ^{-1} \circ  \trasp({q_\mu}^{-1}, -)  \ )
           = ( \mu_{C_1 \ ( x)}^{C_2(  {\mu_{B_1}^{B_2}}(x))})^{-1}$$
      and diagrammatically
       $$
           \xymatrix{
C_2(\mu_{B_1}^{B_2}(x))\ar[dd]_{\trasp({q_{\mu}}^{-1}\ , -)}   \ar[rrr]^{ (\mu_{C_1 \ ( x)}^{C_2(  {\mu_{B_1}^{B_2}}(x))})^{-1}}
% \ar[dd]^{(\mu_{C_1 \ ( \mu_{B_1}^{B_2^{-1}}(-))}^{C_2(  ({\mu_{B_1}^{B_2}} \circ   { \mu_{B_1}^{B_2}}^{-1} ) (-))})^{-1}}
&  & &C_1(x) \\
&&&\\
              C_2(\mu_{B_1}^{B_2}({\mu_{B_1}^{B_2}}^{-1}(\mu_{B_1}^{B_2}(x))))      \ar[rrr]_{(\mu_{C_1 \ ( \mu_{B_1}^{B_2^{-1}}(-))}^{C_2(  ({\mu_{B_1}^{B_2}} \circ   { \mu_{B_1}^{B_2}}^{-1} ) (-))})^{-1}}                                                                                             
  &&& C_1({\mu_{B_1}^{B_2}}^{-1}(\mu_{B_1}^{B_2}(x)))\ar[uu]_{\trasp(q_{\mu}\ , -)}                                                                                       
}$$

    \noindent
    The other  canonical isomorphisms obtained by different clauses can be easily shown to be equipped with canonical inverses by applying the inductive hypothesis to the canonical isomorphisms of lower type complexity.
\end{enumerate}
\end{proof}

\noindent
In \cite{Pal17} Palmgren discussed the issue of equality on objects in categories as formalized in type theory and he defined {\it E-categories } and {\it H-categories}. In this approach a fundamental role is played by the notion of setoid and proof-irrelevant dependent setoid as defined in \cite{m09}. 

\begin{definition}
An $E$-category consists of the following data: a type $C$ of objects, a dependent setoid of morphisms $\mathsf{Hom}(a,b)$ for any $a,b:C$ and a composition operation $\circ: \mathsf{Hom}(b,c)\times \mathsf{Hom}(a,b)\rightarrow \mathsf{Hom}(a,b)$, that is an extensional function in the sense that it preserves the relevant equivalence relations and that satisfy the usual associativity and identity conditions.
\end{definition}

\noindent
We can impose equality on objects in a $E$-category in a way compatible with composition. This leads to the following definition:

\begin{definition}\label{pal-h}
An $H$-category is an $E$-category where the type of objects $C$ is equipped with an equivalence relation $\sim_C$ and there exists a family of isomorphisms $\tau_{a,b,p}\in \mathsf{Hom}(a,b)$ for each $p:a\sim_C b$ such that 

\begin{itemize}
    \item[H1]: $\tau_{a,a,p} = 1_a$ for any $p: a\sim_C a$; 
    
    \item[H2]: $\tau_{a,b,p} = \tau_{a,b,q}$ for any $p,q: a\sim_C b$; 
    
    \item[H3]: $\tau_{b,c,q}\circ \tau_{a,b,p} = \tau_{a,c,r}$ for any $p: a\sim_C b$, $q: b\sim_C c$ and $r: a \sim_C c$.
     
\end{itemize}
\end{definition}

\begin{definition}\label{hcat}
Let $\setis$ be the category of h-sets in $\setem$
up to canonical isomorphisms and functions as morphisms defined
\iffalse as  a ''setoid category'' in \cite{Pal17} \fi
as follows:
the objects of $\setis$ are equivalent classes of  h-sets $A: \setem$ equated under canonical isomorphisms,  i.e.
an object of $\setis$ is an equivalence class $[A]$ of h-sets $A$ in $\setem$  where two objects $A$ and $B$ of $\setis$ are declared equal, by writing $[A]=_{\setis}[B]$, if there exists  a canonical isomorphism
$\tau_A^B: A\rightarrow B$. (Note that by Univalence, the equality  $[A]=_{\setis}[B]$ implies that 
$A=_{\suni}B$ holds in \hott\ as well.)
 
\noindent
Morphisms of $\setis$ from an object $[A]$ to an object $[B]$, indicated with
$\setis ([A],[B])$, 
are determined by  functions $f:A'\rightarrow B'$ between h-sets $A'$ and $B'$ such that $[A']=_{\setis}[A]$ and $[B']=_{\setis}[B]$ and 
 given two functions $f:A'\rightarrow B'$ and $g:A''\rightarrow B''$ 
with $[A']=_{\setis}[A'']$ and $[B']=_{\setis}[B'']$,  we define $f=_{\setis}g$  when 
$\mu_{B'}^{B''} \circ f=_{ A'\rightarrow B''} g\circ \mu_{A'}^{A''} $ holds for
canonical isomorphisms $\mu_{A'}^{A''}\ :\  A' \rightarrow \ A''$ and $\mu_{B'}^{B''}\ :\ B'\ \rightarrow \ B''$.  We denote such  morphisms
with $[f]:[A]\ \rightarrow\  [B]$ and when  there is no loss of generality  we implicitly mean that $f:A\rightarrow B$.
(Note that the morphism equality $[f]=[g]$  for arrows $f,g:A\rightarrow B$
 implies the {\em propositional equality} $f=_{A\rightarrow B}g$.)

\noindent
Composition of morphisms of $[f]:[A]\rightarrow [B]$ and $[g]:[B]\rightarrow [C]$ 
is defined as $[g\circ f] $ for representatives $f: A'\rightarrow B'$  and
$g: B'\rightarrow C'$.

\noindent
The identity morphism from $[A]$ to $[A]$ is the equivalence class $[\mathsf{id}_A]: [A]\rightarrow [A]$
of the identity morphism in \hott.
\end{definition} 

\begin{remark}
The category $\setis$ is a small H-category in the sense of definition \ref{pal-h} by taking  as objects of $\it C$ the setoid whose support is
$\set_{mf}$
and whose equality $A'=_{\it C} B'$ is defined as the truncation of the assumed inductive type
$\vert\vert \mathsf{Ciso}(A',B')\vert\vert$ and the hom-set between two objects
$\mathsf{Hom}(A',B')$ is the setoid having as support the set of arrows  
$A'\ \rightarrow\ B'$, and whose equality for $f,g :A'\ \rightarrow\  B'$ 
 is the propositional equality $f = g$.
Moreover,  for any $p:\vert\vert \mathsf{Ciso}(A',B')\vert\vert$ we define $\tau_{A',B',p}\ \ourdef\ \mathsf{ind}_{\vert \vert \ \vert\vert} (p, z.z)$, which is well defined since any canonical isomorphism between two h-sets is unique up to propositional equality and satisfy the required properties of an H-category as shown in proposition~\ref{isocanprop}.
\end{remark}

\section{ The compatibility of \emtt\ with \hott}
In this section, we show that also the extensional level \emtt\ of \mf\ is compatible with \hott. We are going to define a direct interpretation $\fintd: \emtt \longrightarrow\ \setis$, that is based on a {\it multi-functional} partial interpretation from \emtt\ raw-syntax to \hott\ raw-syntax. As in the case of definition \ref{mtt-int}, we assume to have defined two auxiliary partial maps $\prp$ and $\prs$, both from \hott\ raw-syntax to \hott\ raw-syntax, where the first is meant to associate to a type symbol of \hott\ a (chosen)  proof that it is a h-proposition, while the second associates to a type symbol of \hott\ a (chosen)  proof that it is a h-set. 

We stress the fact that the interpretation crucially relies upon \can\ isomorphisms as defined in definition \ref{caniso}. Indeed, it is only by means of \can\ isomorphisms that we can interpret correctly the definitional equalities and the conversions of \emtt. This means that when we are defining the interpretation for a raw type or a raw term depending on some other raw terms, we assume that the type of this term has been corrected by means of \can\ isomorphisms.

In this sense, the interpretation bears some resemblance to the interpretation of \emtt\ in \mtt\ given in \cite{m09}, but it has a more direct flavor, since we can avoid any setoid model construction thanks to the availability of set quotients as higher inductive types within \hott. 

Further, another important difference with the interpretation presented in \cite{m09} is due to the assumption of the Univalence Axiom. Indeed, the axiom plays a fundamental role in showing the compatibility of \emtt\ with \hott\, since it allows to convert the canonical isomorphism interpreting two definitionally equal \emtt-types  into {\it propositional} equal \hott-types.  The lack of a similar principle in \mtt\ prevents the interpretation in \cite{m09} from achieving a full compatibility result of \emtt\ with \mtt.

We will indicate the interpretation multi-function with $(-)^\blacktriangledown$ and the case when \can\ isomorphisms are required with $(-)^{\widetilde{\blacktriangledown}}$. The notation $(-)^{\widetilde{\blacktriangledown}}$ is similar to that used in \cite{m09}. Given an expression $a$ of \emtt\ raw-syntax, we write $a^{\widetilde{\blacktriangledown}}$ instead of $\widetilde{a}^{\blacktriangledown}$. Moreover, we introduce the following definitions: 

\begin{definition}
Given $A\ type\ [\Gamma]$ and $B\ type\ [\Gamma]$, the judgement $A =_{ext} B$ means that there exists a \can\ isomorphism $\mu_{A}^{B}$ relating $A$ and $B$. 
\end{definition}

\begin{definition}
If $C\ type\ [\Gamma]$ and $D\ type\ [\Delta]$, the judgement $C\ [\Gamma] =_{ext} D\   [\Delta]$ means the following: given $\Gamma\ \ourdef\  x_{1} : A_1, \ldots, x_n : A_n$ and $\Delta\ \ourdef\ y_1 : B_1, \ldots, y_n : B_n$, then we can derive $A_1 =_{ext} B_1, \ldots, A_n =_{ext} B_n[\mu_{A_1}^{B_1}(x_1)/ y_1, \ldots, \mu_{A_{n-1}}^{B_{n-1}}(x_{n-1})/ y_{n-1}] $ and also $C=_{ext} \widetilde{D}\ [\Gamma]$, where \ $\widetilde{D}\ \ourdef\ D [\mu_{A_1}^{B_1}(x_1)/ y_1, \ldots, \mu_{A_{n}}^{B_{n}}(x_{n})/ y_{n}]$ for some \can\  isomorphisms $\mu_{A_i}^{B^i}$ for $i=1,\dots, n$ and $\mu_{C}^{D}$.
\end{definition}

\begin{definition}
Given $c: C\ [\Gamma]$ and $D\ [\Delta]$ such that $C\  [\Gamma] =_{ext} D\ [\Delta]$, where $\Gamma\ \ourdef\ x_1 :A_1, \ldots, x_n :A_n$ and $\Delta\ \ourdef\ y_1:B_1, \ldots, y_n:B_n$, the judgement $c:_{ext} D\ [\Delta]$ means that we can derive $\widetilde{c}: \widetilde{D}\ [\Gamma]$, where $\widetilde{c}\ \ourdef\ \mu_{C}^{D}(c ( \mu_{A_1}^{B_1}(x_1), \ldots, \mu_{A_n}^{B_n}(x_n))) $ for some \can\  isomorphisms $\mu_{A_i}^{B^i}$ for $i=1,\dots, n$ and $\mu_{C}^{D}$.
\end{definition}

\begin{definition}
The judgement $a=_{ext} b :_{ext} A\ [\Gamma]$ means that we can derive $p: \widetilde{a} =_{\widetilde{A}} \widetilde{b}$.
\end{definition}

The definitions given above specify the meaning of the notation $\widetilde{a}$ for any raw-expression $a$ of \emtt\ and thus the notation $(-)^{\widetilde{\blacktriangledown}}$, which we will adopt in the next definition.

\begin{definition}[interpretation of \emtt\ raw-syntax]

We define a partial multifunctional interpretation of raw terms and types of \emtt\ into those of \hott\ 

$$(-)^\blacktriangledown:  \mbox{Raw-syntax }(\emtt) \ \longrightarrow \ \mbox{Raw-syntax }(\hott) $$
assuming to have defined two auxiliary partial functions

$$\prp(-):  \mbox{Raw-syntax }(\hott) \ \longrightarrow \ \mbox{Raw-syntax }(\hott) $$

and 

 $$\prs(-):  \mbox{Rawsyntax }(\hott) \ \longrightarrow \ \mbox{Rawsyntax }(\hott) $$

The definition of $(-)^\blacktriangledown$ for contexts of \emtt\ is the following: $([\ ])^\blacktriangledown$ is defined as $\textbf{1}$ and $(\Gamma, x\in A)^\blacktriangledown$ is defined as $\Gamma^\blacktriangledown, x: A^{\blacktriangledown}$. Furthermore, $(x\in A\ [\Gamma])^\blacktriangledown$ is defined as $x: A^\blacktriangledown\ [\Gamma^\blacktriangledown]$, provided that $x: A^\blacktriangledown$ is in $\Gamma^\blacktriangledown$. 

The interpretation of \emtt-judgements is defined as follows:

\begin{center}
\begin{tabular}{|lcl|}
\hline
&&\\[0.5pt]
    $(A\ set\ [\Gamma])^{\blacktriangledown} $  &is  defined as &  $A^\blacktriangledown :\funi \ [\Gamma^{\blacktriangledown}]$  such that $\prs(A^\blacktriangledown): \isset(A^\blacktriangledown)$ 
    
    is derivable \\[5pt]
    
     $(A\ col\ [\Gamma])^{\blacktriangledown} $  &is  defined as &  $A^\blacktriangledown :\suni \ [\Gamma^{\blacktriangledown}]$  such that $\prs(A^\blacktriangledown): \isset(A^\blacktriangledown)$ is derivable \\[5pt]
    
    $(P\ \textit{prop}_{s}\ [\Gamma])^{\blacktriangledown
    } $ & is defined as & $ \vert\vert P^\blacktriangledown\vert\vert : \funi \ [\Gamma^{\blacktriangledown}]$ such that $\prp(\vert\vert P^\blacktriangledown\vert\vert): \isprop(\vert\vert P^\blacktriangledown\vert\vert)$ is derivable \\[5pt]
    
    $(P\ \textit{prop}\ [\Gamma])^{\blacktriangledown} $ & is defined as & $ \vert\vert P^\blacktriangledown\vert\vert: \suni \ [\Gamma^{\blacktriangledown}]$ such that $\prp(\vert\vert P^\blacktriangledown\vert\vert): \isprop(\vert\vert P^\blacktriangledown\vert\vert)$ is derivable  \\[5pt]
   
     $(A=B \ set \ [\Gamma])^{\blacktriangledown}$ & is defined as & $  (A^{\blacktriangledown}, \prs(A^\blacktriangledown))\ =_{ext} (B^{\blacktriangledown}, \prs(B^\blacktriangledown)) : \set_{\funi} \ [\Gamma^\blacktriangledown]$  \\[5pt]
     
     $(A=B \ col \ [\Gamma])^{\blacktriangledown}$ & is defined as & $  (A^{\blacktriangledown}, \prs(A^\blacktriangledown))\ =_{ext} (B^{\blacktriangledown}, \prs(B^\blacktriangledown)) : \set_{\suni} \ [\Gamma^\blacktriangledown] $   \\[5pt]
     
     $(P=Q\ \textit{prop}_{s}\ [\Gamma])^{\blacktriangledown} $ & is defined as & $ (\vert\vert P^\blacktriangledown\vert\vert, \prp (\vert\vert P^\blacktriangledown\vert\vert)) =_{ext} (\vert\vert Q^\blacktriangledown\vert\vert, \prp (\vert\vert Q^\blacktriangledown\vert\vert)):  \prop_{\funi} \ [\Gamma^\blacktriangledown]$  \\[5pt]
     
    $(P=Q\ \textit{prop}\ [\Gamma])^{\blacktriangledown}$ & is defined as & $(\vert\vert P^\blacktriangledown\vert\vert, \prp (\vert\vert P^\blacktriangledown\vert\vert)) =_{ext} (\vert\vert Q^\blacktriangledown\vert\vert, \prp (\vert\vert Q^\blacktriangledown\vert\vert)): \prop_{\suni} \ [\Gamma^\blacktriangledown]$  \\[5pt]
      $(a\in  A\ \ [\Gamma])^{\blacktriangledown}$ & is defined as & $a^{\widetilde{\blacktriangledown}}:_{ext} A^{\blacktriangledown}\ [\Gamma^\blacktriangledown] $\\[5pt]
         
 $(a=b\in  A\ [\Gamma])^{\blacktriangledown}$ &is defined as & $a^{\widetilde{\blacktriangledown}} =_{ext} b^{\widetilde{\blacktriangledown}} : A^\blacktriangledown \ [\Gamma^\blacktriangledown]$\\[5pt]
 \hline
\end{tabular}
\end{center}

The interpretation of \emtt-constructors is defined as follows: 

\vspace{1.0em}
\begin{tabular}{|l|}
\hline
\\
$(\, \Sigma_{x\in A}B(x)\ [\Gamma])^{\blacktriangledown}\ \ourdef \Sigma_{x: A^{\blacktriangledown}}\, B(x)^{\blacktriangledown}\ [\Gamma^\blacktriangledown]$  \\[5pt]
$(\langle a,b\rangle)^{\blacktriangledown}\ \ourdef\ (a^{\widetilde{\blacktriangledown}}, b^{\widetilde{\blacktriangledown}})$\  \  \\[5pt]

$(\mathrm{El}_{\Sigma}(d, c))^{\blacktriangledown}\ \ourdef\ \mathsf{ind}_{\Sigma}(d^{\widetilde{\blacktriangledown}}, x. y. c(x,y)^{\widetilde{\blacktriangledown}})$\\[5pt]

$\prs ((\Sigma_{x\in A}B(x)\,)^{\blacktriangledown})\ \ourdef\ \mathfrak{s}_{\Sigma}(A^{\blacktriangledown}, \lambda x: A^{\blacktriangledown}. B(x)^{\blacktriangledown}, \prs (A^\blacktriangledown), \lambda x: A^{\blacktriangledown} .\prs(B(x)^{\blacktriangledown}))\ $ \\[5pt]

\hline
\end{tabular}

\begin{tabular}{|ll|}
\hline
&\\
$ (\Pi_{x\in A}B(x)\ [\Gamma])^{\blacktriangledown}\ \ourdef\ \Pi_{x:A^{\blacktriangledown}}\, B(x)^{\blacktriangledown}\ [\Gamma^\blacktriangledown]$
& $(\lambda x.b(x))^{\blacktriangledown}\ \ourdef\ \lambda x. b(x)^{\widetilde{\blacktriangledown}}$\ \\[5pt]

$\prs( (\Pi_{x\in A}B(x))^{\blacktriangledown})\ \ourdef\
\mathfrak{s}_{\Pi}(A^{\blacktriangledown}, \lambda x: A^{\blacktriangledown}. B(x)^{\blacktriangledown}, \lambda x: A^{\blacktriangledown} .\prs(B(x)^{\blacktriangledown}))\  $  

& $(\mathsf{Ap}(f, a) )^{\blacktriangledown}\ \ourdef\  f^{\widetilde{\blacktriangledown}}(a^{\widetilde{\blacktriangledown}})$\\[5pt]
\hline
\end{tabular}
 
 \begin{tabular}{|ll|}
 \hline
 &\\[5pt]
$(\mathsf{N_0}\ [\Gamma])^{\blacktriangledown}\ \ourdef\ \textbf{0}\ [\Gamma^\blacktriangledown]$ &
$(\mathsf{emp}_0(c))^{\blacktriangledown}\ \ourdef\ \mathsf{ind}_0(c^{\widetilde{\blacktriangledown}})$
\\[5pt]
$\prs((\mathsf{N_0})^{\blacktriangledown})\ \ourdef\ \mathfrak{s}_0\ $ & \\[5pt]
\hline
\end{tabular}

\begin{tabular}{|ll|}
\hline
&\\
$(\mathsf{N_1}\ [\Gamma])^{\blacktriangledown}\ \ourdef\ \textbf{1}\ [\Gamma^\blacktriangledown]$ & $(\star)^{\blacktriangledown}\ \ourdef\ \star$\\[5pt]
$\prs((\mathsf{N_1})^{\blacktriangledown})\ \ourdef\  \mathfrak{s}_1\ $ & 
$(\mathrm{El}_{\mathsf{N_1}}(t, c))^{\blacktriangledown}\ \ourdef\ \mathsf{ind_1}(t^{\widetilde{\blacktriangledown}},  c^{\widetilde{\blacktriangledown}})$\\[5pt]
\hline
\end{tabular}
 
 \begin{tabular}{|l|}
 \hline
\\[5pt]
$(A+B\ [\Gamma])^{\blacktriangledown}\ \ourdef\  A^{\blacktriangledown}+B^{\blacktriangledown}\ [\Gamma^\blacktriangledown]$\\[5pt]

$(\mathsf{inl}(a))^{\blacktriangledown}\ \ourdef\ \mathsf{inl}(a^{\widetilde{\blacktriangledown}})$\qquad
$(\mathsf{inr}(b))^{\blacktriangledown}\ \ourdef\ \mathsf{inr}(b^{\widetilde{\blacktriangledown}})$\\[5pt]

$(\mathrm{El}_+(c, d_A, d_B))^{\blacktriangledown}\ \ourdef\ \mathsf{ind}_+(c^{\widetilde{\blacktriangledown}}, x.d_{A}(x)^{\widetilde{\blacktriangledown}}, y. d_B(y)^{\widetilde{\blacktriangledown}})$\\[5pt]

$\prs((A+B)^{\blacktriangledown})\ \ourdef\  \mathfrak{s}_{+}(A^{\blacktriangledown}, B^{\blacktriangledown}, \prs(A^{\blacktriangledown}), \prs(B^{\blacktriangledown}))\ $\\[5pt]
 
\hline
\end{tabular}
 
 \begin{tabular}{|ll|}
 \hline
&\\
$(\mathsf{List}(A)\ [\Gamma])^{\blacktriangledown}\ \ourdef\ \mathsf{List}(A^{\blacktriangledown})\ [\Gamma^\blacktriangledown]$\ & $(\epsilon )^{\blacktriangledown}\ \ourdef\ \mathsf{nil}$\qquad
$(\mathsf{cons}(\ell, a))^{\blacktriangledown}\ \ourdef\ \mathsf{cons}(\ell^{\widetilde{\blacktriangledown}}, a^{\widetilde{\blacktriangledown}})$\\[5pt]

$\prs((\mathsf{List}(A))^{\blacktriangledown})\ \ourdef\ 
\mathfrak{s}_{\mathsf{List}}(A^{\blacktriangledown}, \prs(A^{\blacktriangledown})) $\ 
& 
$(\mathrm{El}_{\mathsf{List}}(c,d, l))^{\blacktriangledown}\ \ourdef\ \mathsf{ind}_{\mathsf{List}}( c^{\widetilde{\blacktriangledown}}, d^{\widetilde{\blacktriangledown}}, x.y.z. l(x,y,z)^{\widetilde{\blacktriangledown}})$\\[5pt]
\hline
\end{tabular}

\begin{tabular}{|ll|}
\hline
     &\\
  $(A/R\ [\Gamma])^{\blacktriangledown}\ \ourdef\ A^{\blacktriangledown}/ R^{\blacktriangledown}\ [\Gamma^\blacktriangledown] $\ & $([ a ])^{\blacktriangledown}\ \ourdef\ \mathsf{q}(a^{\widetilde{\blacktriangledown}})$ \\[5pt]
  $\prs((A/R)^{\blacktriangledown})\ \ourdef\ \mathfrak{s}_{Q}(A^{\blacktriangledown},  R^{\blacktriangledown}, \prs(A^{\blacktriangledown}), \prp(R^{\blacktriangledown}), r^{\blacktriangledown}) \mbox{ for some term }  r$ &
       $(\mathrm{El}_{Q}(p,c))^{\blacktriangledown}\ \ourdef\ \mathsf{ind_Q}(p^{\widetilde{\blacktriangledown}}, c^{\widetilde{\blacktriangledown}})$\\[5pt]
       
       $(\textbf{true}\in R(a,b)\ [\Gamma])^{\blacktriangledown}\ \ourdef\ p: R(a,b)^{\blacktriangledown}\ [\Gamma^{\blacktriangledown}]\ \mbox{for some term}\ p$& \\[5pt]
  \hline
\end{tabular}

\begin{tabular}{|l|}
\hline
     \\
     $(\mathcal{P}(1)\ [\Gamma])^{\blacktriangledown}\ \ourdef\ \mathsf{Prop}_{\funi}\ [\Gamma^\blacktriangledown]$\ \\[5pt] $([A])^{\blacktriangledown}\ \ourdef\ (\vert\vert A^{\blacktriangledown}\vert\vert ,  \prp(\vert\vert A^{\blacktriangledown}\vert\vert))$
     \\[5pt]
     $\prs((\mathcal{P}(1))^{\blacktriangledown})\ \ourdef\ \mathfrak{s}_{\mathsf{Prop_0}}$\\[5pt]
     
     $(\textbf{true}\in A\leftrightarrow B\ [\Gamma])^{\blacktriangledown}\ \ourdef\ p: \vert\vert A^{\blacktriangledown}\vert\vert \leftrightarrow \vert\vert B^{\blacktriangledown}\vert\vert\ [\Gamma^{\blacktriangledown}]\ \mbox{for some term}\ p$
     \\[5pt]
     \hline
\end{tabular}

\begin{tabular}{|ll|}
\hline
&\\
$ ( A\ \rightarrow \ \mathcal{P}(1)\ [\Gamma] )^{\blacktriangledown} \ourdef\ A^{\blacktriangledown}\ \rightarrow \ \mathsf{Prop}_{\funi }\ [\Gamma^\blacktriangledown]$  & $(\lambda x.b(x))^\blacktriangledown\ \ourdef\ \lambda x.b(x)^{\widetilde{\blacktriangledown}}\ $ \\[5pt]
$\prs ( ( A\ \rightarrow \ \mathcal{P}(1) )^{\blacktriangledown}) \ourdef  \mathfrak{s}_{\Pi}(A^\blacktriangledown, \lambda:A^\blacktriangledown.\prop_{\funi}, \mathfrak{s}_{\mathsf{Prop_0}}\ )$
 &
$(\mathsf{Ap}(f,a))^{\blacktriangledown}\ourdef\ f^{\widetilde{\blacktriangledown}}(a^{\widetilde{\blacktriangledown}})$\\[5pt]
\hline
\end{tabular}

\begin{tabular}{|ll|}
 \hline
&\\
$(\bot\ [\Gamma])^{\blacktriangledown}\ \ourdef\ \vert\vert\textbf{0}\vert\vert\ [\Gamma^\blacktriangledown]$ & 

$(\mathsf{true}\in C\ [\Gamma])^{\blacktriangledown}\ \ourdef\ \mathsf{ind_{\bot^{\blacktriangledown}}}( c^{\widetilde{\blacktriangledown}}): C^\blacktriangledown\ [\Gamma^\blacktriangledown]\ \mbox{for some term}\ c$\ \\[5pt]

$\prp((\bot)^{\blacktriangledown})\ \ourdef\ \mathfrak{p}_{\vert\vert\ \vert\vert}(\textbf{0}) $ 

& $\prs((\bot)^{\blacktriangledown})\ \ourdef\ \mathfrak{s}_{coe}((\bot)^\blacktriangledown, \prp((\bot)^{\blacktriangledown}) )$ \\[5pt]
\hline
\end{tabular}

\begin{tabular}{|l|}
\hline     
 \\[5pt]
 $(A\lor B\ [\Gamma])^{\blacktriangledown}\ \ourdef\ A^{\blacktriangledown}\lor B^{\blacktriangledown}\ [\Gamma^\blacktriangledown]$\\[5pt] $(\mathsf{true}\in A\lor B\ [\Gamma])^{\blacktriangledown}\  \ourdef\ \mathsf{inl}_{\lor}(a^{\widetilde{\blacktriangledown}}): A^\blacktriangledown \lor B^\blacktriangledown\ [\Gamma^\blacktriangledown]\ \mbox{for some term}\ a$\\[5pt]
$(\mathsf{true}\in A\lor B\ [\Gamma])^{\blacktriangledown}\ \ourdef\ \mathsf{inr}_{\lor}(b^{\widetilde{\blacktriangledown}}): A^\blacktriangledown \lor B^\blacktriangledown\ [\Gamma^\blacktriangledown]\ \mbox{for some term}\ b$\\[5pt]
$(\mathsf{true}\in C\ [\Gamma])^{\blacktriangledown}\ \ourdef\  \mathsf{ind}_{\lor}( d^{\widetilde{\blacktriangledown}},  x.c_1(x)^{\widetilde{\blacktriangledown}},  y.c_2(y)^{\widetilde{\blacktriangledown}}): C^\blacktriangledown\ [\Gamma^\blacktriangledown]\ \mbox{for some terms}\ c_1, c_2, d$\\[5pt]

$\prp((A\lor B)^{\blacktriangledown})\ \ourdef\  \mathfrak{p}_{\lor}(A^{\blacktriangledown}, B^{\blacktriangledown})\ $\\[5pt]

$\prs((A \lor B)^{\blacktriangledown})\ \ourdef\ \mathfrak{s}_{coe}((A\lor B)^{\blacktriangledown}, \prp((A\lor B)^{\blacktriangledown})) $ \\[5pt]
\hline
\end{tabular}

\begin{tabular}{|l|}
\hline
\\
$(A\land B\ [\Gamma])^{\blacktriangledown} \ourdef\ \vert\vert A^{\blacktriangledown}\times B^{\blacktriangledown}\vert\vert\ [\Gamma^\blacktriangledown]$\\[5pt]
$(\mathsf{true}\in A\land B\ [\Gamma])^{\blacktriangledown}\ \ourdef\ (a^{\widetilde{\blacktriangledown}},_{\land} b^{\widetilde{\blacktriangledown}}): \vert\vert A^\blacktriangledown\times B^\blacktriangledown\vert\vert\ [\Gamma^\blacktriangledown]\ \mbox{for some terms}\ a,b$
 \\[5pt]
  $(\mathsf{true}\in A\ [\Gamma])^{\blacktriangledown}\ \ourdef\ {\fpr}_{\land}( c^{\widetilde{\blacktriangledown}}): A^\blacktriangledown\ [\Gamma^\blacktriangledown]\ \mbox{for some term}\ c$\\[5pt]
  
   $(\mathsf{true}\in B\ [\Gamma])^{\blacktriangledown}\ \ourdef\ {\spr}_{\land}( c^{\widetilde{\blacktriangledown}}): B^\blacktriangledown\ [\Gamma^\blacktriangledown]\ \mbox{for some term}\ c$\\[5pt]
  
 $\prp((A\land B)^{\blacktriangledown})\ \ourdef\ \mathfrak{p}_{\vert\vert \times\vert\vert}(A^{\blacktriangledown}, B^{\blacktriangledown})$\\[5pt]
 
 $\prs((A \land B)^{\blacktriangledown})\ \ourdef\ \mathfrak{s}_{coe}((A\land B)^\blacktriangledown, \prp((A\land B)^{\blacktriangledown}))$ \\[5pt]
\hline
\end{tabular}
 
\begin{tabular}{|l|}
\hline
\\
$(A\rightarrow B\ [\Gamma])^\blacktriangledown\ \ourdef\ \vert\vert A^\blacktriangledown\rightarrow B^\blacktriangledown\vert\vert\ [\Gamma^\blacktriangledown]$\\[5pt]
$(\mathsf{true}\in A\rightarrow B\ [\Gamma])^\blacktriangledown\ \ourdef\ \lambda_{\rightarrow} x.b^{\widetilde{\blacktriangledown}}:\vert\vert A^\blacktriangledown\rightarrow B^\blacktriangledown\vert\vert\ [\Gamma^\blacktriangledown]\ \mbox{for some term}\ b$\\[5pt]

$(\mathsf{true}\in B\ [\Gamma])^\blacktriangledown\ \ourdef\ f_{\rightarrow}^{\widetilde{\blacktriangledown}} (a^{\widetilde{\blacktriangledown}}): B^\blacktriangledown\ [\Gamma^\blacktriangledown]\ \mbox{for some terms}\ a,f$\\[5pt]

$\prp((A\rightarrow B)^\blacktriangledown)\ \ourdef\ \mathfrak{p}_{\vert\vert\rightarrow \vert\vert}(A^\blacktriangledown, B^\blacktriangledown)$\\[5pt]

$\prs((A\rightarrow B)^\blacktriangledown )\ \ourdef\ \mathfrak{s}_{coe}(( A \rightarrow B)^\blacktriangledown , \prp((A\rightarrow B)^\blacktriangledown))$ \\[5pt]
\hline
\end{tabular}

\begin{tabular}{|l|}
\hline  
\\
$(\exists_{x\in A}B(x)\ [\Gamma])^{\blacktriangledown}\ \ourdef\  \exists_{x:A^{\blacktriangledown}}\, B(x)^{\blacktriangledown}\ [\Gamma^\blacktriangledown]$\\[5pt]

$(\mathsf{true}\in \exists_{x\in A}B(x)\ [\Gamma] )^{\blacktriangledown}\ \ourdef\ (a^{\widetilde{\blacktriangledown}},_{\exists} b^{\widetilde{\blacktriangledown}}): \exists_{x:A^{\blacktriangledown}}\, B(x)^{\blacktriangledown}\ [\Gamma^\blacktriangledown]\ \mbox{for some terms}\ a,b $\\[5pt] 

$(\mathsf{true}\in C\ [\Gamma])^{\blacktriangledown}\ \ourdef\ \mathsf{ind}_{\exists}( d^{\widetilde{\blacktriangledown}},  x.  y. c(x,y)^{\widetilde{\blacktriangledown}}): C^\blacktriangledown\ [\Gamma^\blacktriangledown]\ \mbox{for some terms}\ c,d $\\[5pt]

$\prp((\exists_{x\in A}B(x))^{\blacktriangledown})\ \ourdef\ 
\mathfrak{p}_{\exists}(A^{\blacktriangledown}, \lambda x: A^{\blacktriangledown}. B(x)^{\blacktriangledown} \, )$\\[5pt]

$\prs((\exists_{x\in A}B(x))^{\blacktriangledown})\ \ourdef\ \mathfrak{s}_{coe}( (\exists_{x\in A}B(x))^{\blacktriangledown}
, \,\prp((\exists_{x\in A}B(x))^{\blacktriangledown}))$ \\[5pt]

\hline

\end{tabular}

\begin{tabular}{|l|}
\hline
    \\
$(\forall_{x\in A}B(x)\ [\Gamma])^{\blacktriangledown}\ \ourdef  \vert\vert \Pi_{x:A^{\blacktriangledown}}\ B(x)^{\blacktriangledown}\vert\vert\ [\Gamma^\blacktriangledown]$\\[5pt]

$(\mathsf{true}\in \forall_{x\in A}B(x)\ [\Gamma])^{\blacktriangledown}\ \ourdef\ \lambda_{\forall} x.b(x)^{{\widetilde{\blacktriangledown}}}:\vert\vert \Pi_{x:A^{\blacktriangledown}}\ B(x)^{\blacktriangledown}\vert\vert\ [\Gamma^\blacktriangledown]\ \mbox{for some term}\ b $\\[5pt]
 
 $(\mathsf{true}\in B(a)\ [\Gamma])^{\blacktriangledown}\ \ourdef\ (f^{\widetilde{\blacktriangledown}})_\forall(a^{{\widetilde{\blacktriangledown}}}): B(a)^{\blacktriangledown}\ [\Gamma^\blacktriangledown]\ \mbox{for some terms}\ a,f $\\[5pt]

$\prp((\forall_{x\in A}B(x))^{\blacktriangledown})\ \ourdef\ \mathfrak{p}_{\vert\vert\Pi\vert\vert}(A^{\blacktriangledown}, \lambda x: A^{\blacktriangledown} .B(x)^\blacktriangledown)\  $\\[5pt]

 $\prs((\forall_{x\in A}B(x))^{\blacktriangledown})\ \ourdef\ \mathfrak{s}_{coe}((\forall_{x\in A}B(x))^\blacktriangledown, \, \prp((\forall_{x\in A}B(x))^{\blacktriangledown} ))$ \\[5pt]
 \hline
\end{tabular}

\begin{tabular}{|l|}
\hline
     \\
$(\mathrm{Eq}(A,a,b)\ [\Gamma])^{\blacktriangledown} \ourdef\ \vert\vert\mathrm{Id}_{A^{\blacktriangledown}}(a^{\widetilde{\blacktriangledown}}, b^{\widetilde{\blacktriangledown}})\vert\vert\ [\Gamma^\blacktriangledown]$\\[5pt]

$(\mathsf{true}\in \mathrm{Eq}(A,a,a)\ [\Gamma])^{\blacktriangledown}\ \ourdef\ \vert\mathsf{refl}_{a^{\widetilde{\blacktriangledown}}}\vert: \vert\vert \mathrm{Id}_{A^{\blacktriangledown}}(a^{\widetilde{\blacktriangledown}}, a^{\widetilde{\blacktriangledown}})\vert\vert\ [\Gamma^\blacktriangledown]\ \mbox{for some term}\ a$\\[5pt]
$\prp((\mathrm{Eq}(A,a,b))^{\blacktriangledown})\ \ourdef\ \mathfrak{p}_{\vert\vert\ \vert\vert}(A^{\blacktriangledown}, a^{\widetilde{\blacktriangledown}}, b^{\widetilde{\blacktriangledown}}, \mathrm{Id}_{A^{\blacktriangledown}}(a^{\widetilde{\blacktriangledown}},b^{\widetilde{\blacktriangledown}}) ) $\\[5pt]
$\prs((\mathrm{Eq}(A,a,b))^{\blacktriangledown})\ \ourdef\ \mathfrak{s}_{coe}( (\mathrm{Eq}(A,a,b))^{\blacktriangledown}\, ,\, \prp((\mathrm{Eq}(A,a,b))^{\blacktriangledown}))$ \\[5pt]
\hline
\end{tabular}

\end{definition}
\vspace{1.0em}

\begin{remark}
We could alternatively give a single clause for judgements with the proof-term \lq {\bf true}', namely $(\textbf{true})^{\blacktriangledown}\ \ourdef\ p$ for some proof-term $p$ in \hott. This would allow us to avoid to specify the interpretation of $\textbf{true}$ for each term constructor, since all these cases would be particular instances of this generic clause, but then we should make explicit how to recover them in the validity theorem.
\end{remark}

\begin{definition}
Let $(-)^{\blacklozenge}$ be a multifunctional interpretation from the raw-syntax of \emtt-types and terms judgements to the raw-syntax of \hott-types and terms judgements defined as follows: 

$$ (\mathcal{J})^{\blacklozenge}\ \ourdef\ (\mathcal{J})^{\blacktriangledown}\ \mbox{if $\mathcal{J}$ is a {\it type} judgement} $$

$$ (\mathcal{J})^{\blacklozenge}\ \ourdef\ (\mathcal{J})^{\widetilde{\blacktriangledown}}\ \mbox{if $\mathcal{J}$ is a {\it term} judgement} $$

\end{definition}

In order to define the interpretation of \emtt-judgements into the category $\setis$, we need to allow the possibility of regarding dependent types as arrows into the category and the following definition is introduced for this purpose: 

\begin{definition}
Let $\Gamma$ be a context in \hott, then we define by induction over the length of $\Gamma$ the indexed closure $Sig (\Gamma)$, which comes equipped with projections $\pi_{i}^{n}(z)$ for $z: Sig (\Gamma)$ and $i = 1, \ldots, n$

\begin{tabular}{l}
 \\     
$\mbox{If}\ \Gamma\ \ourdef\ x:A, \mbox{then}\ Sig\ (\Gamma)\ \ourdef\ A\ \mbox{and}\ \pi_{1}^{1}(z)\ \ourdef\ z$\\[5pt]
 $\mbox{If}\ \Gamma\ \ourdef\ \Delta, x:A\ \mbox{of length }\ n+1, \mbox{then}\ Sig\ (\Gamma)\ \ourdef\ (\Sigma_{z: Sig (\Delta)}\  A[\pi_{1}^{n}(z)/x_1, \ldots, \pi_{n}^{n}(z)/x_n]) $\\[5pt]
\end{tabular}

where $\pi_{i}^{n+1}(w)\ \ourdef\ \pi_{i}^{n}(\pi_{1}(w))$ for $i=1,\ldots, n$ and $\pi_{n+1}^{n+1}(w)\ \ourdef\ \pi_{2}(w)$ for any 

$w: \Sigma_{z: Sig(\Delta)} A [\pi_{1}^{n}(z)/x_1, \ldots, \pi_{n}^{n}(z)/x_n]$.

Moreover, we denote $\overline{a}$ the result of the substitution  of the free variables $x_1, \ldots, x_n$ in a term $a$ with $\pi_{i}^{n}(z)$ for $i=1, \ldots, n$ and $z: Sig(\Gamma)$.
\end{definition}

The definition of the multi-function interpretation  $(-)^{\blacklozenge}$ from the raw-syntax of \emtt\ to the raw-syntax of \hott\ allows us to define a direct interpretation $\fintd$ : $\emtt\rightarrow \setis$ of \emtt-judgements into the category $\setis$ described in definition \ref{hcat}

\begin{definition}
The interpretation $\fintd$\ : $\emtt\rightarrow \setis$ is defined by using the partial multi-function $(-)^{\blacklozenge}$ in the following way: 

\begin{itemize}
    \item[-] An \emtt-type judgements is interpreted as a projection in $\setis$
    
    $$\fintd\ (A\ type\ [\Gamma])\ourdef\ [\pi_1]: [Sig (\Gamma^{\blacklozenge}, A^{\blacklozenge})]\ \rightarrow\ [Sig(\Gamma^{\blacklozenge})]$$
    
    which amounts to derive $A^{\blacklozenge}\ [\Gamma^{\blacklozenge}]$ in \hott\ with canonical transports.
    
    \item[-] An \emtt-type equality judgement is interpreted as the equality of type interpretations in $\setis$
    
    $$\fintd\ (A = B\ type\ [\Gamma])\ourdef\ \fintd\ (A\ type\ [\Gamma])=_{\setis} {\it \fintd}\ (B\ type\ [\Gamma])$$
    
    which amounts to derive $ A^{\blacklozenge}\ [\Gamma^{\blacklozenge}] =_{ext}  B^{\blacklozenge}\ [\Gamma^{\blacklozenge}] $ and hence $ A^{\blacklozenge} =_{\suni}  B^{\blacklozenge}\ [\Gamma^{\blacklozenge}] $. 
    
    \item[-] An \emtt-term judgement is interpreted as a section of the interpretation of the corresponding type
    
    $$\fintd\ (a\in A\ [\Gamma])\ourdef\ [\langle z, \overline{a}^{\blacklozenge}\rangle]: [Sig(\Gamma^{\blacklozenge})] \rightarrow [Sig (\Gamma^{\blacklozenge}, A^{\blacklozenge})]$$
    
    which amounts to derive $a^{\blacklozenge}: A^{\blacklozenge}\ [\Gamma^{\blacklozenge}]$ in \hott\ with $A^{\blacklozenge}\ [\Gamma^{\blacklozenge}]$ equipped with canonical transports. 
%    The notation $\overline{a}^{\blacklozenge}$ denotes the result of the substitution of the free variables $x_1, \ldots, x_n$ in $a^{\blacklozenge}$ with $\pi_{i}^{n}(z)$ for $i=1, \ldots, n$ and $z: Sig(\Gamma)$.
    
    \item[-] An \emtt-term equality judgement is interpreted as the equality of term interpretations in $\setis$
    
    $$\fintd\ (a = b\in A\ [\Gamma])\ \ourdef\ \fintd\ (a\in A\ [\Gamma])=_{\setis} \fintd\ (b\in A\ [\Gamma])$$
    
    which amounts to derive $a^{\blacklozenge} =_{A^{\blacklozenge}} b^{\blacklozenge}\ [\Gamma^{\blacklozenge}] $, for some $a^{\blacklozenge}: A^{\blacklozenge}\ [\Gamma^{\blacklozenge}]$ and $b^{\blacklozenge}: A^{\blacklozenge}\ [\Gamma^{\blacklozenge}]$.
\end{itemize}
\end{definition}

In the following, given $\Gamma\ \ourdef \Delta', x_n : A_n, \Delta''$ with $\Delta''\ \ourdef\ x_{n+1}: A_{n+1}, \ldots, x_{m}: A_{m}$, then for every $a\ : A_n\ [\Delta']$ and for any type $B\ type\ [\Gamma]$, we denote the substitution of $x_n$ with $a$ in $B$ as $$B[a/x_n]\ type\ [\Delta', \Delta''_{a}]$$ instead of the extended form $$B[a/x_n] [x'_i/x_i]_{i= n+1,\ldots, m}\ type\ [\Delta', \Delta''_{a}]$$ where $$\Delta''_{a}\ \ourdef\ x'_{n+1}: A'_{n+1}, \ldots, x'_{m}: A'_{m}$$ and $$A'_{j}\ \ourdef\ A_{j} [a_n/x_n] [x'_i/x_i]_{i=n+2, \ldots, m}$$ if $n+2\leq m$, otherwise $A'_{n+1}\ \ourdef\ A_{n+1}[a_n/x_n]$. Moreover, if $\Delta''$ is the empty context, then $\Delta''_{a}$ is the empty context as well. We use similar abbreviations also for terms. 

\begin{lemma}[Substitution]\label{sub2}

For any \emtt-judgement $B\ type\ [\Gamma]$ interpreted in $\setis$ as 
$$[\pi_1]: [Sig(\Gamma^{\blacklozenge}, y: B^{\blacklozenge})]\rightarrow [Sig(\Gamma^{\blacklozenge})]$$
if $\Gamma\ \ourdef\ \Delta', x_{n}\in A_{n}, \Delta''$,  then for every \emtt-judgement $a\in A_n\ [\Delta']$ interpreted as 
$[\langle  z, \overline{a}^{\blacklozenge}\rangle]: [Sig (\Delta'^{\blacklozenge})]\rightarrow [Sig( \Delta'^{\blacklozenge}, x_n \in A_n^{\blacklozenge} )],$
$$ \fintd (B [a/x_n]\ type\ [\Delta', \Delta_{a}''])\ =_{\setis} [\pi_1]: [Sig (\Delta'^{\blacklozenge}, \Delta_{a}''^{\blacklozenge}, y\in B^{\blacklozenge}[a^{\blacklozenge}/x_n])]\rightarrow [Sig(\Delta'^{\blacklozenge}, \Delta_{a}''^{\blacklozenge})] $$

Similarly, for any \emtt-judgement $b\in B\ [\Gamma]$, where $B$ and $\Gamma$ are exactly as specified above, and which is interpreted as $[\langle z, \overline{b}^{\blacklozenge}\rangle]: [Sig(\Gamma^{\blacklozenge})]\rightarrow [Sig(\Gamma^{\blacklozenge}, y: B^{\blacklozenge})]$, 
$$\begin{array}{l} \fintd ( b [a/x_n]\in B[a/x_n]\ [\Delta', \Delta''_{a}]) =_{\setis}\\ \qquad\qquad\qquad[\langle z, \overline{b^{\blacklozenge} [a^{\blacklozenge}/x_n]\ \rangle}\ ]:[Sig(\Delta'^{\blacklozenge}, \Delta_{a}''^{\blacklozenge})] \rightarrow  [Sig (\Delta'^{\blacklozenge}, \Delta_{a}''^{\blacklozenge}, y\in B^{\blacklozenge}[a^{\blacklozenge}/x_n])] . \end{array}$$

\end{lemma}

\begin{proof}
By induction over the interpretation of raw types and terms after noting that canonical isomorphisms are closed under substitution.
\end{proof}

\begin{theorem}
If $A\ type\ [\Gamma]$ is derivable in \emtt, then $\fintd\ ( A\ type\ [\Gamma])$ is well-defined. 

If $a\in A\ [\Gamma]$ is derivable in \emtt, then $\fintd\ ( a\in A\ [\Gamma])$ is well-defined. 

If $A\ type\ [\Gamma]$, $B\ type\ [\Gamma]$ and $A = B\ [\Gamma]$ are derivable in \emtt, then $\fintd\ (A = B\ [\Gamma])$ is well-defined. 

If $a\in A\ [\Gamma]$, $b\in A\ [\Gamma]$ and $a = b\ \in A\ [\Gamma]$ are derivable in \emtt, then $\fintd\ (a = b\in A\ [\Gamma])$ is well-defined.

Therefore, \emtt\ is valid with respect to the interpretation $\fintd$.
\end{theorem}

\begin{proof}
The proof is by induction over the derivation of judgements. Sets in $\set_{mf}$ form a $\Pi$-pretopos, therefore they possess enough structure to interpret \emtt-type and term constructors.  Note that conversion rules are interpreted correctly by \can\ isomorphisms, since it is possible to coerce a term along a \can\ isomorphism for the definitions given above. Indeed the rule 

\begin{center}
    \AxiomC{$a\in A\ [\Gamma]$}
    \AxiomC{$A = B\ type\ [\Gamma]$}
    \RightLabel{$\mathsf{conv}$}
    \BinaryInfC{$a\in B\ [\Gamma]$}
    \DisplayProof
    \end{center}
is interpreted as follows: by inductive hypothesis, $\fintd (a\in A\ [\Gamma])$ is well-defined and amounts to derive $a^{\blacklozenge}: A^{\blacklozenge}\ [\Gamma^{\blacklozenge}]$ for some $a^{\blacklozenge}$ and some $A^{\blacklozenge}\ type\ [\Gamma^{\blacklozenge}]$ in \hott; further, $\fintd (A = B\ type\ [\Gamma])$ is well-defined too and amounts to derive $A^{\blacklozenge} =_{ext} B^{\blacklozenge}\ [\Gamma^{\blacklozenge}]$ for some canonical isomorphism $\mu: A^{\blacklozenge}\rightarrow B^{\blacklozenge}\ [\Gamma^{\blacklozenge}]$ and for some $A^{\blacklozenge}\ type\ [\Gamma^{\blacklozenge}], B^{\blacklozenge}\ type\ [\Gamma^{\blacklozenge}]$ in \hott\ and thus, by Univalence, it boils down to $A^{\blacklozenge} =_{\guni} B^{\blacklozenge}\ [\Gamma^{\blacklozenge}]$. Therefore, $\fintd( a \in B\ [\Gamma])$ is well-defined, since $\mu(a): B^{\blacklozenge}\ [\Gamma^{\blacklozenge}, a:A^{\blacklozenge}]$ is derivable and, moreover, such an isomorphism is unique up to propositional equality.

The power collection of the singleton $\mathcal{P}(1)$ is interpreted as $\prop_{\funi} : \suni$ together with a proof $\prs ((\mathcal{P}(1)^{\blacklozenge})): \isset ( (\mathcal{P}(1)^{\blacklozenge}))$. The introduction rule

\begin{center}
    \AxiomC{$A\ prop_s\ [\Gamma]$}
    \RightLabel{I-P}
    \UnaryInfC{$[A]\in \mathcal{P}(1)\ [\Gamma]$}
\DisplayProof
\end{center}

\noindent
is validated as follows: by induction hypothesis,
$\fintd( A\ prop_s\ [\Gamma])$ is well-defined and amounts to derive $\vert\vert A^{\blacklozenge}\vert\vert: \funi\ [\Gamma^{\blacklozenge}]$ together with $\prp(\vert\vert A^{\blacklozenge}\vert\vert): \isprop (\vert\vert A^{\blacklozenge}\vert\vert)$. Therefore the conclusion is immediately valid, since it boils down to derive $(\vert\vert A^{\blacklozenge}\vert\vert,  \prp(\vert\vert A^{\blacklozenge}\vert\vert)): \prop_{\funi}\ [\Gamma^{\blacklozenge}]$.

Then there are the following two rules: 

\begin{center}
    \AxiomC{$\textbf{true}\in A\leftrightarrow B\ [\Gamma]$}
\RightLabel{$\mathsf{eq}\mbox{-}\mathcal{P}(1)$}
\UnaryInfC{$[A] = [B]\in \mathcal{P}(1)\ [\Gamma]$}
\DisplayProof
\hskip1.5em
\AxiomC{$[A] = [B]\in \mathcal{P}(1)\ [\Gamma]$}
\RightLabel{$\mathsf{eff}\mbox{-}\mathcal{P}(1)$}
\UnaryInfC{$\textbf{true}\in A\leftrightarrow B\ [\Gamma]$}
\DisplayProof
\end{center}

For the first: by induction hypothesis, $\fintd( \textbf{true} \in A\leftrightarrow B\ [\Gamma])$ is well-defined and hence there exists a proof-term $p$ such that $p: \vert\vert A^{\blacklozenge}\vert\vert \leftrightarrow \vert\vert B^{\blacklozenge}\vert\vert\ [\Gamma^{\blacklozenge}]$ is derivable and $(\textbf{true})^{\blacklozenge}\ \ourdef\ p$, but then by Propositional Extensionality we can infer $\vert\vert A^{\blacklozenge}\vert\vert =_{\funi} \vert\vert B^{\blacklozenge}\vert\vert\ [\Gamma^{\blacklozenge}]$ and then the conclusion is valid, because $(\vert\vert A^{\blacklozenge}\vert\vert, \prp(\vert\vert A^{\blacklozenge}\vert\vert)) =_{\prop_{\funi}} (\vert\vert B^{\blacklozenge}\vert\vert, \prp(\vert\vert B^{\blacklozenge}\vert\vert))$ holds. The latter instead trivially follows by definition of $(-)^{\blacklozenge}$. 

For \emtt-quotients we have the {\it effectiveness} rule: 

\begin{center}
\AxiomC{$a\in A\ [\Gamma]$}
\AxiomC{$b\in A\ [\Gamma]$}
\AxiomC{$[a] = [b]\in A/R\ [\Gamma]$}
\AxiomC{$A/R\ set\ [\Gamma]$}
\RightLabel{$\mathsf{eff}\mbox{-}Q$}
\QuaternaryInfC{$\textbf{true}\in R(a,b)\ [\Gamma]$}
\DisplayProof

\end{center}

which is interpreted as follows: by induction hypothesis, $\fintd$ applied to the premises is well-defined and this amounts to derive that there exist $a^{\blacklozenge}, b^{\blacklozenge}$ in \hott\ such that $a^{\blacklozenge}: A^{\blacklozenge}\ [\Gamma^{\blacklozenge}]$, $b^{\blacklozenge}: A^{\blacklozenge}\ [\Gamma^{\blacklozenge}]$, $ \mathsf{q}(a^{\blacklozenge}) =_{A^{\blacklozenge}/R^{\blacklozenge}} \mathsf{q}(b^{\blacklozenge})\ [\Gamma^{\blacklozenge}]$ are derivable and $A^{\blacklozenge}/R^{\blacklozenge}: \funi\ [\Gamma^{\blacklozenge}]$ together with a proof $\prs(A^{\blacklozenge}/R^{\blacklozenge}): \isset (A^{\blacklozenge}/R^{\blacklozenge})$ is derivable as well for some $A^{\blacklozenge}$ and $R^{\blacklozenge}$. Since {\it set} quotients in \hott\ are effective (see remark \ref{eff-quot}), then the interpretation of the conclusion is well-defined and the effectiveness rule is validated by our interpretation. Indeed, for some \hott-term $p$ such that $(\textbf{true})^{\blacklozenge}\ \ourdef\ p$, we can derive $p: R(a,b)^{\blacklozenge}\ [\Gamma^{\blacklozenge}]$.

The reflection rule for extensional propositional equality

\begin{center}
    \AxiomC{$\textbf{true}\in \mathrm{Eq}(A, a,b)\ [\Gamma]$}
    \RightLabel{E-Eq}
    \UnaryInfC{$a = b\in A\ [\Gamma]$}
    \DisplayProof
\end{center}

is trivially validated by our interpretation. Indeed, if we assume that $\fintd$ is well-defined for the premise, then this means that $p: \mathrm{Id}_{A^{\blacklozenge}}(a^{\blacklozenge}, b^{\blacklozenge})\ [\Gamma^{\blacklozenge}]$ is derivable for some $a^{\blacklozenge}, b^{\blacklozenge}: A^{\blacklozenge}$ and some proof-term $p$. But then the interpretation of the conclusion is well-defined as well, since it amounts to derive $p: \mathrm{Id}_{A^{\blacklozenge}}(a^{\blacklozenge}, b^{\blacklozenge})\ [\Gamma^{\blacklozenge}]$ for some $p$.

In general, the interpretation of the judgements with proof-term ${\bf true}$ works by restoring a corresponding proof-term in the intensional setting: $\fintd ({\bf true}\in A\ [\Gamma]) $  amounts to derive that there exists a term $p$ such that $p: A^{\blacklozenge}\ [\Gamma^{\blacklozenge}]$ is derivable in \hott\ and where $p$ corresponds to $(\textbf{true})^{\blacklozenge}$. By way of example, let us consider the following rule: 

\begin{center}
    \AxiomC{$a\in A\ [\Gamma]$}
    \RightLabel{I-Eq}
    \UnaryInfC{$\textbf{true}\in \mathrm{Eq}(A, a,a)\ [\Gamma]$}
    \DisplayProof
\end{center}

By induction hypothesis, $\fintd (a\in A\ [\Gamma])$ is well-defined and hence we can derive $a^{\blacklozenge}: A^{\blacklozenge}\ [\Gamma^{\blacklozenge}]$ for some term $a^{\blacklozenge}$ in \hott; then the interpretation of the conclusion is also well-defined, since $\vert \mathsf{refl}_{a^{\blacklozenge}}\vert : \vert\vert \mathrm{Id}_{A^{\blacklozenge}}(a^{\blacklozenge}, a^{\blacklozenge})\vert\vert\ [\Gamma^{\blacklozenge}]$ is derivable and $(\textbf{true})^{\blacklozenge}\ \ourdef\ \vert \mathsf{refl}_{a^{\blacklozenge}}\vert$. Therefore, the rule I-Eq is validated by our interpretation.

Finally, note that the validity of $\beta$-rules also depends on the substitution lemma \ref{sub2}.
\end{proof}

\begin{remark}\label{truejudg}

An important feature of the interpretation of \emtt\ is that it can be regarded as an extension of Martin-L\"of's interpretation of {\it true judgements} \cite{ML84, sienlec}. A judgement of the form $A\ true$ must be read intuitionistically as \lq there exists a proof of $A$'. In \emtt\ we know that there exists a unique canonical inhabitant for propositions denoted by {\bf true} and hence we have that $A\ true \ourdef \textbf{true} \in A$. By applying the interpretation defined above, we can recover a proof-term $p$ such that $p : A^{\blacklozenge}$. Such $p$ could be considered as a typed realizer. Indeed, as a result of the validity of the interpretation, true judgments are endowed with computational content. However, this result was already achieved in the interpretation of \emtt\ in \mtt\ given in \cite{m09}. We just remark that this applies also to the present interpretation.  

\end{remark}

\begin{remark}\label{alt}
We could have interpreted \emtt\ within \hott\ in another way by employing as an intermediate step the interpretation of \emtt\ within the quotient model
construction   $\cqis$\  done in \cite{m09}. The reason is that this quotient model construction could be functorially mapped into $\setis$ by employing set-quotients of
\hott\ and a variation of the   interpretation $(-)^\bullet$ of \mtt\ within \hott\
where all \emtt\ propositions are interpreted as truncated propositions (as in remark \ref{ptrunc})
in order to guarantee that the canonical isomorphisms defined
in \cite{m09}
betweeen extensional dependent types, which are actually dependent setoids (the word \lq\lq setoid''
was avoided in \cite{m09} because \mtt-types are not all called sets!),
are sent to canonical isomorphisms of \hott\ as defined in \ref{caniso}.

The  existence of such an alternative interpretation  in $\setis$  is also expected  for categorical reasons. First, $\cqis$
is an instance of a general categorical construction called {\it elementary quotient completion} in \cite{QCFF,EQC}. Second, such a completion satisfies
a universal property  with respect to suitable  Lawvere's elementary doctrines
closed under stable effective quotients including as an example
the elementary doctrine of h-propositions indexed over a suitable syntactic category of h-sets
of \hott\  thanks to the presence
of set-quotients in \hott.
However,  it is not guaranteed that the resulting translation from \emtt\ into \hott\  shows that \emtt\
is
compatible with \hott\ by construction. We think that  the best way to show this would be to check that this alternative interpretation is \lq\lq isomorphic'' to the one described in this section
according to a suitable notion of  isomorphism between interpretations of \emtt\  which would
be better described after shaping both interpretations in categorical terms
as functors from a suitable syntactic category of \emtt\ into
a category with families, in the sense of \cite{Dybjer95}, built out of $\setis$. The precise definition of this alternative
compatible translation of \emtt\ within \hott\ and the possible use of an heterogeneous equality as in 
 \cite{ABKT19,WinST19}  are left to future work.
\end{remark}

\begin{remark}
Note that   the interpretations of   \mtt\ and \emtt\ within \hott\ presented in the previous section,  interpret both the \mtt-universe of small propositions
$\mathsf{Prop_s}$ and the \emtt\ power-collection $\mathcal{P}(1)$ of  the singleton set as the set $\mathsf{Prop}_{\funi}$ of h-propositions in the first universe up to propositional equality.  
Indeed, we could have interpreted the equality judgements  of \mtt\  concerning the definitional equality of types and terms as done
for \emtt.  

However, we have chosen to interpret the definitional
equality of \mtt-types and terms as {\it definitional equality of types and terms of \hott}
to preserve not only the meaning of \mtt-sets and propositions but also the type-theoretic distinction between definitional and propositional equality which disappears in the extensional version of dependent type theories
as \emtt.
\end{remark}

\begin{remark}[{\bf Related Works}]
Of course, the already cited work by M. Hofmann in  \cite{Hof95, HofTh} is related to the one presented here, being related to the interpretation of \emtt\ into \mtt\
in \cite{m09} as said in the Introduction.
%A definition of a class of isomorphisms similar to our \can\ isomorphisms was introduced by Hofmann in \ 
Hofmann aimed to show the conservativity of extensional type theory over the intensional one extended with function extensionality and uniqueness of identity proofs axioms. His approach is  semantic since he employed a  category with families  quotiented under canonical isomorphisms.
% (e.g. contexts are equivalence classes of contexts in this model and two contexts are equivalent if they are canonically isomorphic). 
The drawback of using such a semantic approach is that  the whole development relies on the Axiom of Choice, which allows to pick out a representative from each equivalence class
involved in the construction. 

Hofmann's interpretation was made effective later in  \cite{Oury05, WinST19} by defining a syntactical translation  which
is closed to our interpretation of \emtt\ into \hott.
Both interpretations are actually multifunctional since they  associates to any judgment in the source extensional type theory  a set of possible judgements in the target intensional type theory linked by means of an heterogenous equality in \cite{Oury05, WinST19}  and by canonical isomorphisms in ours.

Note that  our interpretation does not achieve any conservativity result over \hott: first, \emtt\ is not an extension of \hott\ and moreover the derivability of the axiom  of unique choice in \hott\  prevents any conservativity result because it is not valid in \emtt\ (see Remark \ref{ac}). 

\end{remark}

\section{Conclusions}

We have shown how to interpret both levels of \mf\ within \hott\  in a compatible way by  preserving the meaning of logical and set-theoretical constructors. Higher inductive set-quotients, Univalence for h-propositions in the first universe $\funi$ and function extensionality for h-sets within the second universe $\suni$ are the additional principles on the top of Martin-L{\"of}'s type theory which are needed to interpret \emtt\  within \hott\ in a way that preserves compatibility. On the other hand, the interpretation also works thanks to the possibility of defining \can\ isomorphisms within \hott. 

In the future we hope to investigate the alternative translation of \emtt\ within \hott\ mentioned in remark
\ref{alt}.
Moreover, we would like to employ an extension of \hott\ with Palmgren's superuniverse to interpret both levels of \mf\ extended with inductive and coinductive definitions
as in \cite{cind,MS23}. 

As a relevant consequence of the results shown here, both levels of  \mf\ inherit a computable model where proofs are seen as programs 
    in \cite{SA21}  and a model witnessing its consistency with Formal Church's thesis 
  in  \cite{su22}.  We leave to future work to relate them with those already available for \mf\  extended with Church's thesis in
  \cite{peff}, \cite{IMMS},\cite{mmr21}, \cite{cind},  and in particular with the predicative variant of Hyland's Effective Topos
  in \cite{peff}. 
 It would also be very relevant from the computational point of view to relate \mf\ and its extensions in \cite{cind} with  Berger and Tsuiki's logic presented in \cite{BT21} as a  framework for program extraction from proofs.

\subsection*{Acknowledgments}  The second author  acknowledges very useful discussions with Thorsten Altenkirch, Pietro Sabelli, and Thomas Streicher on the topic of this paper.
Both  authors wish to thank Steve Awodey  for hosting them at CMU (in different periods)
and for the very helpful discussions with him  and his collaborators Jonas Frey and Andrew Swan.

Last but not least, we thank the referees for their very helpful suggestions.
\bibliography{bibliomfhott}
\bibliographystyle{alpha}

\subsection*{Appendix A: The translation of  \mtt-syntax in \hott}

Here we spell out the interpretation of the raw syntax of \mtt-types and terms as raw types and terms  of \hott.
First of all, all variables in \mtt\ are translated as variables of \hott\ without changing the name
$$x^\sqbullet\ourdef x$$
Then the interpretation of specific \mtt-types and terms
is defined in the following table:

\vspace{1.0em}
\begin{tabular}{|l|}
\hline
\\[5pt]
$(\mathsf{prop_s} )^{\sqbullet} \ourdef  \mathsf{Prop}_{\funi}$\\[5pt]
$\prs ((\mathsf{prop_s} )^{\sqbullet} ) \ourdef \mathfrak{s}_{\mathsf{Prop_0}} $ \\[8pt]

$(\tau(p))^{\sqbullet}\ourdef  \fpr (p^\sqbullet) $ \\[5pt]
$\prp ( (\tau(p))^{\sqbullet}) \ourdef \spr(p^\sqbullet)$  \\[5pt]
$\prs ( (\tau(p))^{\sqbullet}) \ourdef \mathfrak{s}_{coe} ((\tau(p))^{\sqbullet} , \prp (( \tau(p))^{\sqbullet})  )$  \\[5pt]

\begin{tabular}{l}
\\[5pt]
    $(\widehat{\bot})^{\sqbullet} \ourdef \ (\textbf{0}, \ \mathfrak{p}_0 )$\qquad $({\widehat{\top} })^\sqbullet \ \ourdef \ ( {\bf 1} , \ \mathfrak{p}_1 \, ) $\\[5pt]
  $(p\widehat{\lor} q)^\sqbullet  \ourdef (\fpr( p^\sqbullet)  \lor \fpr(q^\sqbullet) \ ,\ \mathfrak{p}_{\lor }(\fpr(p^\sqbullet),\fpr (q^\sqbullet )\, )) $\\[5pt]
$(p\widehat{\land} q)^\sqbullet \ourdef 
(\fpr(p^\sqbullet) \times \fpr(q^\sqbullet) \ ,\ \mathfrak{p}_{\times}(\fpr(p^\sqbullet),\fpr (q^\sqbullet ) , \spr(p^\sqbullet), \spr(q^\sqbullet)\, )) $\\[5pt]
$(p\widehat{\rightarrow}q)^\sqbullet\ \ourdef\ (\fpr(p^\sqbullet) \rightarrow \fpr(q^\sqbullet) \ ,
\ \mathfrak{p}_\rightarrow  (\fpr(p^\sqbullet),\fpr (q^\sqbullet ) , \spr(p^\sqbullet), \spr(q^\sqbullet)\, ))     $\\[5pt]
$(\widehat{\exists}_{x\in A}\, p(x))^\sqbullet\ \ourdef\  (\exists_{x:A^\sqbullet} \ p(x)^\sqbullet,\ \mathfrak{p}_\exists ( A^\sqbullet,\lambda x. \fpr(p(x)^\sqbullet))\  \, ) $\\[5pt]
$(\widehat{\forall}_{x\in A}\, p (x))^\sqbullet\ \ourdef\ (\Pi_{x:A^\sqbullet}\ p(x)^\sqbullet\ ,
\ \mathfrak{p}_{\Pi}(A^\sqbullet, \lambda x.\fpr( p(x)^\sqbullet),   \,\lambda x.\spr( p(x)^\sqbullet) ))$\\[5pt]
$(\widehat{\mathrm{Id}}(A,a,b))^\sqbullet\ \ourdef\ (\mathrm{Id}_{A^\sqbullet}(a^\sqbullet,b^\sqbullet)\ ,\ \mathfrak{p}_{Id}(A^\sqbullet, \prs(A^\sqbullet), a^\sqbullet, b^\sqbullet\ ))$ \\[5pt]
\end{tabular}\\[5pt]

\hline
\end{tabular}

\begin{tabular}{|ll|}
\hline
&\\
$ ( A\ \rightarrow \ \mathsf{prop_s} )^{\sqbullet} \ourdef\ A^{\sqbullet}\ \rightarrow \ \mathsf{Prop}_{\funi }$  & $(\lambda x.b(x))^\sqbullet\ \ourdef\ \lambda x.b(x)^{\sqbullet}\ $ \\[5pt]
$\prs( ( A\ \rightarrow \ \mathsf{prop_s} )^{\sqbullet}) \ourdef  \mathfrak{s}_{\Pi} (A^\sqbullet, \lambda x:A^{\sqbullet}.\prop_{\funi}, \mathfrak{s}_{\mathsf{Prop_0}}\ )$
 &
$(\mathsf{Ap}(f,a))^{\sqbullet}\ourdef\ f^\sqbullet(a^{\sqbullet})$\\[5pt]
\hline
\end{tabular}

\begin{tabular}{|l|}
\hline

\\[5pt]
$(\, \Sigma_{x\in A}B(x)\,)^{\sqbullet}\ \ourdef \Sigma_{x: A^{\sqbullet}}\, B(x)^{\sqbullet}$\\[5pt]

$(\langle a,b\rangle)^{\sqbullet}\ \ourdef\ (a^{\sqbullet}, b^{\sqbullet})$\  \  \\[5pt]

$(\mathrm{El}_{\Sigma}(d, c))^{\sqbullet}\ \ourdef\ \mathsf{ind}_{\Sigma}(d^{\sqbullet}, x. y. c(x,y)^{\sqbullet})$\\[5pt]

$\prs ((\Sigma_{x\in A}B(x)\,)^{\sqbullet})\ \ourdef\ \mathfrak{s}_{\Sigma}(A^\sqbullet,   \lambda x: A^\sqbullet. B(x)^\sqbullet, \prs(A^\sqbullet), \lambda x: A^\sqbullet .\prs(B(x)^\sqbullet))\ $\\[5pt]

\hline
\end{tabular}

\begin{tabular}{|ll|}
\hline
&\\
$ (\Pi_{x\in A}B(x))^{\sqbullet}\ \ourdef\ \Pi_{x:A^{\sqbullet}}\, B(x)^{\sqbullet}$,\ &
$(\lambda x.b(x))^{\sqbullet}\ \ourdef\ \lambda x.b(x)^{\sqbullet}$\\[5pt]

$\prs( (\Pi_{x\in A}B(x))^{\sqbullet})\ \ourdef\
\mathfrak{s}_{\Pi}(A^\sqbullet, \lambda x: A^\sqbullet. B(x)^\sqbullet, \lambda x: A^\sqbullet .\prs(B(x)^\sqbullet))\  $  & \begin{tabular}{l} 
 $(\mathsf{Ap}(f,a) )^{\sqbullet}\ \ourdef\  f^\sqbullet(a^{\sqbullet})$
 \end{tabular}\\[5pt]
 \hline
 \end{tabular}
 
 \begin{tabular}{|ll|}
 \hline
 &\\[5pt]
$(\mathsf{N_0})^{\sqbullet}\ \ourdef \textbf{0}$ &
$(\mathsf{emp}_0(c))^\sqbullet\ \ourdef\ \mathsf{ind}_0(c^\sqbullet)$
\\[5pt]
$\prs((\mathsf{N_0})^\sqbullet)\ \ourdef\ \mathfrak{s}_0\ $ & \\[5pt]
\hline
\end{tabular}

\begin{tabular}{|ll|}
\hline
&\\
$(\mathsf{N_1})^{\sqbullet}\ \ourdef\ \textbf{1}$ & $(\star)^{\sqbullet}\ \ourdef\ \star$\\[5pt]
$\prs((\mathsf{N_1})^{\sqbullet})\ \ourdef\  \mathfrak{s}_1\ $ & 
$(\mathrm{El}_{\mathsf{N_1}}(t, c))^\sqbullet\ \ourdef\ \mathsf{ind_1}(t^\sqbullet,  c^\sqbullet)$\\[5pt]
\hline
\end{tabular}
 
 \begin{tabular}{|l|}
 \hline
\\
$(A+B)^{\sqbullet}\ \ourdef\  A^{\sqbullet}+B^{\sqbullet}$,\\[5pt]

$(\mathsf{inl}(a))^\sqbullet\ \ourdef\ \mathsf{inl}(a^\sqbullet)$\\[5pt]

$(\mathsf{inr}(b))^\sqbullet\ \ourdef\ \mathsf{inr}(b^\sqbullet)$\\[5pt]

$(\mathrm{El}_+(c, d_A, d_B))^{\sqbullet}\ \ourdef\ \mathsf{ind}_+(c^{\sqbullet}, x.d_{A}(x)^{\sqbullet}, y. d_B(y)^{\sqbullet})$\\[5pt]

$\prs((A+B)^{\sqbullet})\ \ourdef\  \mathfrak{s}_{+}(A^\sqbullet, B^\sqbullet, \prs(A^\sqbullet), \prs(B^\sqbullet))\ $\\[5pt]
\hline
\end{tabular}
 
 \begin{tabular}{|ll|}
 \hline
&\\
$(\mathsf{List}(A))^{\sqbullet}\ \ourdef\ \mathsf{List}(A^{\sqbullet})$,\ & $(\epsilon )^{\sqbullet}\ \ourdef\ \mathsf{nil}$\qquad
$(\mathsf{cons}(\ell, a))^\sqbullet\ \ourdef\ \mathsf{cons}(\ell^{\sqbullet}, a^{\sqbullet})$\\[5pt]

$\prs((\mathsf{List}(A))^{\sqbullet})\ \ourdef\ 
\mathfrak{s}_{\mathsf{List}}(A^\sqbullet, \prs(A^\sqbullet)) $\ 
& 
$(\mathrm{El}_{\mathsf{List}}(c,d, l)^{\sqbullet}\ \ourdef\ \mathsf{ind}_{\mathsf{List}}( c^{\sqbullet}, d^{\sqbullet}, x.y.z. l(x,y,z)^{\sqbullet})$\\[5pt]
\hline
\end{tabular}
 
\begin{tabular}{|ll|}
 \hline
&\\
$(\bot)^{\sqbullet}\ \ourdef\ \textbf{0}$ & 
\begin{tabular}{l} 
$(\mathsf{r_0}(c))^\sqbullet\ \ourdef\ \mathsf{ind_0}(c^\sqbullet)$
\end{tabular}\\[5pt]

$\prp((\bot)^\sqbullet)\ \ourdef\ \mathfrak{p}_0 $ &\\[5pt] 

$\prs((\bot)^\sqbullet)\ \ourdef\ \mathfrak{s}_{coe}((\bot)^\sqbullet, \prp((\bot)^\sqbullet) )$ &\\[5pt]
\hline
\end{tabular}

\begin{tabular}{|l|}
\hline     
 \\[5pt]
 $(A\lor B)^{\sqbullet}\ \ourdef\ A^{\sqbullet}\lor B^{\sqbullet}$\\[5pt]
 
 $(\mathsf{inl}_{\lor}(a))^\sqbullet\  \ourdef\ \mathsf{inl}_\vee(a^\sqbullet)$\qquad
$(\mathsf{inr}_{\lor}(b))^\sqbullet\ \ourdef\ \mathsf{inr}_\vee(b^\sqbullet)$\\[5pt]

$(\mathrm{El}_{\lor}(d, c_A, .c_B))^\sqbullet\ \ourdef\  \mathsf{ind}_{\vee}( d^\sqbullet,  x.c_1(x)^\sqbullet,  y.c_2(y)^\sqbullet)$\\[5pt]

$\prp((A\lor B)^{\sqbullet})\ \ourdef\  \mathfrak{p}_{\lor}(A^\sqbullet, B^\sqbullet)\ $\\[5pt]

$\prs((A \lor B)^\sqbullet)\ \ourdef\ \mathfrak{s}_{coe}( (A\lor B)^\sqbullet, \prp((A\lor B)^{\sqbullet})) $\\[5pt]
\hline
\end{tabular}

\begin{tabular}{|ll|}
\hline
&\\
$(A\land B)^{\sqbullet} \ourdef\ A^{\sqbullet}\times B^{\sqbullet}$ & 
$(\langle a,_{\land} b\rangle)^\sqbullet\ \ourdef\ (a^\sqbullet, b^\sqbullet)$
 \\[5pt]
  
 $\prp((A\land B)^\sqbullet)\ \ourdef\ \mathfrak{p}_{\times}(A^\sqbullet, B^\sqbullet, \prp(A^\sqbullet), \prp(B^\sqbullet) \, )$ & $(\pi_{i}(c))^\sqbullet\ \ourdef\ \ipr(c^\sqbullet) \mbox{(for $i$=(1,2))}$\\[5pt]
 
 $\prs((A \land B)^\sqbullet)\ \ourdef\ \mathfrak{s}_{coe}((A\land B)^\sqbullet, \prp((A\land B)^\sqbullet))$ &\\[5pt]
\hline
\end{tabular}
 
\begin{tabular}{|ll|}
\hline
&\\
$(A\rightarrow B)^\sqbullet\ \ourdef\ A^\sqbullet\rightarrow B^\sqbullet$ & $(\lambda_{\rightarrow}x.b)^\sqbullet\ \ourdef\ \lambda x.b^\sqbullet$\\[5pt]

$\prp((A\rightarrow B)^\sqbullet)\ \ourdef\ \mathfrak{p}_{\rightarrow}(A^\sqbullet, B^\sqbullet, \prp(A^\sqbullet), \prp(B^\sqbullet))$ &
$(\mathsf{Ap}_{\rightarrow}(f,a))^\sqbullet\ \ourdef\ f^\sqbullet(a^\sqbullet)$\\[5pt]

$\prs((A\rightarrow B)^\sqbullet )\ \ourdef\ \mathfrak{s}_{coe}((A \rightarrow B)^\sqbullet, \prp((A\rightarrow B)^\sqbullet))$& \\[5pt]
\hline
\end{tabular}

\begin{tabular}{|l|}
\hline  
\\[5pt]
$(\exists_{x\in A}B(x))^{\sqbullet}\ \ourdef\  \exists_{x:A^{\sqbullet}}\, B(x)^{\sqbullet}$\\[5pt]
$(\langle a,_\exists b\rangle)^\sqbullet\ \ourdef\ (a^\sqbullet,_{\exists} b^\sqbullet)$\\[5pt] 

$(\mathrm{El}_{\exists}(d, c))^\sqbullet\ \ourdef\ \mathsf{ind}_{ \exists }(d^\sqbullet,  x.  y. c(x,y)^\sqbullet)$\\[5pt] 

$\prp((\exists_{x\in A}B(x))^{\sqbullet})\ \ourdef\ 
\mathfrak{p}_{\exists}(A^\sqbullet, \lambda x: A^\sqbullet. B(x)^\sqbullet \, )$\\[5pt]

$\prs((\exists_{x\in A}B(x))^\sqbullet)\ \ourdef\ \mathfrak{s}_{coe}( (\exists_{x\in A} B(x))^\sqbullet, \,\prp((\exists_{x\in A}B(x))^{\sqbullet}))$ \\[5pt]

\hline

\end{tabular}

\begin{tabular}{|ll|}
\hline
    &\\
$(\forall_{x\in A}B(x))^{\sqbullet}\ \ourdef  \Pi_{x:A^{\sqbullet}}\ B(x)^{\sqbullet}$ & $(\lambda_{\forall}x.b(x))^\sqbullet\ \ourdef\ \lambda x.b(x)^{\sqbullet}$\\[5pt]
$\prp((\forall_{x\in A}B(x))^{\sqbullet})\ \ourdef\ \mathfrak{p}_{\Pi}(A^\sqbullet, \lambda x: A^\sqbullet .B(x)^\sqbullet, \lambda x: A^\sqbullet . \prp(B(x)^\sqbullet))\  $ &
 $(\mathsf{Ap}_{\forall}(f,a))^{\sqbullet}\ \ourdef\ f^\sqbullet(a^{\sqbullet})$\\[5pt]
 $\prs((\forall_{x\in A}B(x))^{\sqbullet})\ \ourdef\ \mathfrak{s}_{coe}( (\forall_{x\in A}B(x))^{\sqbullet}, \, \prp((\forall_{x\in A}B(x))^{\sqbullet} ))$ &\\[5pt]
 \hline
\end{tabular}

\begin{tabular}{|l|}
\hline
\\[5pt]
$(\mathrm{Id}(A,a,b))^{\sqbullet} \ourdef\ \mathrm{Id}_{A^{\sqbullet}}(a^\sqbullet, b^\sqbullet)$\\[5pt]
$(\mathsf{id_A}(a))^\sqbullet\ \ourdef\ \mathsf{refl}_{a^\sqbullet}$\\[5pt]

$(\mathrm{El}_{\mathrm{Id}}(p, c))^{\sqbullet}\ \ourdef\ \mathsf{ind}_\mathrm{Id}(p^\sqbullet,  \ x.c(x)^{\sqbullet})$\\[5pt]

$\prp((\mathrm{Id}(A,a,b))^{\sqbullet})\ \ourdef\ \mathfrak{p}_{Id}(A^\sqbullet, \prs(A^\sqbullet), a^\sqbullet, b^\sqbullet ) $\\[5pt]

$\prs((\mathrm{Id}(A,a,b))^{\sqbullet})\ \ourdef\ \mathfrak{s}_{coe}((\mathrm{Id}(A,a,b))^{\sqbullet}\, ,\, \prp((\mathrm{Id}(A,a,b))^{\sqbullet}))$ \\[5pt]
\hline
\end{tabular}

\subsection*{Appendix B: An alternative translation of  \mtt-syntax in \hott}

\begin{definition}
We define a partial interpretation of the raw syntax of types and terms of \mtt\ in the raw-syntax of \hott\
$$(-)^\inter:  \mbox{Raw-syntax }(\mtt) \ \longrightarrow \ \mbox{Raw-syntax }(\hott) $$
assuming to have defined two auxiliary partial functions:

$$\prp(-):  \mbox{Raw-syntax }(\hott) \ \longrightarrow \ \mbox{Raw-syntax }(\hott) $$
and 
 $$\prs(-):  \mbox{Raw-syntax }(\hott) \ \longrightarrow \ \mbox{Raw-syntax }(\hott) $$

$(-)^\inter$ is defined on contexts and judgements of \mtt\ exactly as the interpretation in definition \ref{mtt-int}.

In the case of \mtt\ term and type constructors all clauses are defined as the corresponding ones in the previous table with the exception of those which are listed below:

\vspace{1.0em}
\begin{tabular}{|l|}
\hline
\\
$(\mathsf{prop_s} )^{\inter} \ourdef  \mathsf{Prop}_{\funi}$\\[5pt]
$\prs ((\mathsf{prop_s} )^{\inter} ) \ourdef \mathfrak{s}_{\mathsf{Prop_0}} $\\[8pt]

$(\tau(p))^{\inter}\ourdef  \vert\vert\fpr (p^\inter) \vert\vert$ \\[5pt]
$\prp ( (\tau(p))^{\inter}) \ourdef \spr(p^\inter)$  \\[5pt]
$\prs ( (\tau(p))^{\inter})\ourdef \mathfrak{s}_{coe} ((\tau(p))^{\inter} , \prp ( (\tau(p))^{\inter}) ) $  \\[5pt]

\begin{tabular}{l}
\\[5pt]
 $(\widehat{\bot})^{\inter} \ourdef \ (\vert\vert \textbf{0}\vert\vert, \ \mathfrak{p}_{\vert\vert\ \vert\vert}(\textbf{0}) )$\qquad $({\widehat{\top} })^\inter \ \ourdef \ ( \vert\vert{\bf 1}\vert\vert , \ \mathfrak{p}_{\vert\vert\ \vert\vert}(\textbf{1}) \, ) $\\[5pt]

$(p\widehat{\lor} q)^\inter  \ourdef (\fpr( p^\inter)  \lor \fpr(q^\inter) \ ,\ \mathfrak{p}_{\lor }(\fpr(p^\inter),\fpr (q^\inter )\, )) $\\[5pt]
$(p\widehat{\land} q)^\inter \ourdef 
(\vert\vert\fpr(p^\inter) \times \fpr(q^\inter)\vert\vert\ ,\ \mathfrak{p}_{\vert\vert\times \vert\vert}(\fpr(p^\inter),\fpr (q^\inter ))) $\\[5pt]
$(p\widehat{\rightarrow}q)^\inter\ \ourdef\ (\vert\vert\fpr(p^\inter) \rightarrow \fpr(q^\inter)\vert\vert\ ,\ \mathfrak{p}_{\vert\vert\rightarrow \vert\vert} (\fpr(p^\inter),\fpr (q^\inter )) )     $\\[5pt]
$(\widehat{\exists}_{x\in A}\, p(x))^\inter\ \ourdef\  (\exists_{x:A^\inter} \ p(x)^\inter,\ \mathfrak{p}_\exists ( A^\inter,\lambda x. \fpr(p(x)^\inter))\  \, ) $\\[5pt]
$(\widehat{\forall}_{x\in A}\, p (x))^\inter\ \ourdef\ (\vert\vert\Pi_{x:A^\inter}\ p(x)^\inter\vert\vert\ ,
\ \mathfrak{p}_{\vert\vert\  \vert\vert}(\Pi_{x: A^\inter}\ \fpr( p(x)^\inter)\, )$\\[5pt]
$(\widehat{\mathrm{Id}}(A,a,b))^\inter\ \ourdef\ (\vert\vert\mathrm{Id}_{A^\inter}(a^\inter,b^\inter)\vert\vert\ ,\ \mathfrak{p}_{\vert\vert\ \vert\vert}(\, \mathrm{Id}_{A^\inter}(a^\inter, b^\inter)\, )$\\[5pt]

\end{tabular}\\[5pt]

\hline
\end{tabular}

\begin{tabular}{|ll|}
 \hline
&\\[5pt]
$(\bot)^{\inter}\ \ourdef\ \vert\vert \textbf{0}\vert\vert$ & 
$(\mathsf{r_0}(c))^\inter\ \ourdef\ \mathsf{ind_{\bot^\inter}}(c^\inter)$
\\[5pt]

$\prp((\bot)^\inter)\ \ourdef\ \mathfrak{p}_{\vert\vert\ \vert\vert}(\textbf{0}) $ &
$\prs((\bot)^\inter)\ \ourdef\ \mathfrak{s}_{coe}((\bot)^\inter, \prp((\bot)^\inter) )$ \\[5pt]
\hline
\end{tabular}

\begin{tabular}{|ll|}
\hline
&\\
$(A\land B)^{\inter} \ourdef\ \vert\vert A^{\inter}\times B^{\inter}\vert\vert $ & 
$(\langle a,_{\land} b\rangle)^\inter\ \ourdef\ (a^\inter,_{\land} b^\inter)$
 \\[5pt]
  
 $\prp((A\land B)^\inter)\ \ourdef\ \mathfrak{p}_{\vert\vert\times \vert\vert}(A^\inter, B^\inter)$ & $(\pi_{i}(c))^\inter\ \ourdef\ {\ipr}_{\land}( c^\inter)\ \mbox{(for $i$=(1,2))}$\\[5pt]
 
 $\prs((A \land B)^\inter)\ \ourdef\ \mathfrak{s}_{coe}((A\land B)^\inter, \prp((A\land B)^\inter))$ &\\[5pt]
 
\hline
\end{tabular}
 
\begin{tabular}{|ll|}
\hline
&\\
$(A\rightarrow B)^\inter\ \ourdef\ \vert\vert A^\inter\rightarrow B^\inter\vert\vert$ & $(\lambda_{\rightarrow}x.b)^\inter\ \ourdef\ \lambda_{\rightarrow} x.b^\inter$\\[5pt]

$\prp((A\rightarrow B)^\inter)\ \ourdef\ \mathfrak{p}_{\vert\vert\rightarrow \vert\vert}(A^\inter, B^\inter)$ &
$(\mathsf{Ap}_{\rightarrow}(f,a))^\inter\ \ourdef\ f_{\rightarrow}^\inter (a^\inter)$\\[5pt]
$\prs((A\rightarrow B)^\sqbullet )\ \ourdef\ \mathfrak{s}_{coe}((A \rightarrow B)^\sqbullet, \prp((A\rightarrow B)^\sqbullet))$& \\[5pt]

\hline
\end{tabular}

\begin{tabular}{|ll|}
\hline
    &\\
$(\forall_{x\in A}B(x))^{\inter}\ \ourdef  \vert\vert \Pi_{x:A^{\inter}}\ B(x)^{\inter}\vert\vert$ & $(\lambda_{\forall}x.b(x))^\inter\ \ourdef\ \lambda_{\forall} x.b(x)^{\inter}$\\[5pt]
$\prp((\forall_{x\in A}B(x))^{\inter})\ \ourdef\ \mathfrak{p}_{\vert\vert\Pi \vert\vert}(A^\inter, \lambda x: A^\inter .B(x)^\inter)\  $ &
 $(\mathsf{Ap}_{\forall}(f,a))^{\inter}\ \ourdef\ f_{\forall}^\inter (a^{\inter})$ \\[5pt]
 
  $\prs((\forall_{x\in A}B(x))^{\sqbullet})\ \ourdef\ \mathfrak{s}_{coe}( (\forall_{x\in A}B(x))^{\sqbullet}, \, \prp((\forall_{x\in A}B(x))^{\sqbullet} ))$ &\\[5pt]
 \hline
\end{tabular}

\begin{tabular}{|l|}
\hline
    \\[5pt]
$(\mathrm{Id}(A,a,b))^{\inter} \ourdef\ \vert\vert \mathrm{Id}_{A^{\inter}}(a^\inter, b^\inter)\vert\vert$ \\[5pt]

$(\mathsf{id_A}(a))^\inter\ \ourdef\ \vert\mathsf{refl}_{a^\inter}\vert$\\[5pt]

$(\mathrm{El}_{\mathrm{Id}}(p, c))^\inter \ \ourdef\ \mathsf{ind}_{\vert\vert\ \vert\vert}( p^\inter,  z.\mathsf{ind}_{\mathrm{Id}}(z, x.c(x)^{\inter})\ )$\\[5pt]

$\prp((\mathrm{Id}(A,a,b))^{\inter})\ \ourdef\ \mathfrak{p}_{\vert\vert\ \vert\vert}(A^\inter, a^\inter, b^\inter, \mathrm{Id}_{A^\inter}(a^\inter, b^\inter)\ ) $\\[5pt]
$\prs((\mathrm{Id}(A,a,b))^{\sqbullet})\ \ourdef\ \mathfrak{s}_{coe}((\mathrm{Id}(A,a,b))^{\sqbullet}\, ,\, \prp((\mathrm{Id}(A,a,b))^{\sqbullet}))$ \\[5pt]

\hline
\end{tabular}
\vspace{1.0em}

 \end{definition}
 
It is possible to show a validity theorem for this interpretation by an argument quite similar to that in \ref{mtt-val}. 
 
\end{document}